\documentclass[notitlepage,11pt,reqno]{amsart}  
\usepackage[foot]{amsaddr}
\usepackage[margin=1in]{geometry} 
\usepackage{amssymb}
\usepackage{amsthm}
\usepackage{tabularx}
\usepackage{physics}
\usepackage{color}
\usepackage{graphicx}
\usepackage[shortlabels]{enumitem}
\usepackage{tikz}
\usepackage{subcaption}   
\usepackage{float}
\usepackage{placeins}

\usepackage[round]{natbib}

\setlist{topsep=0mm}
\theoremstyle{plain} 
\numberwithin{equation}{section}
\newtheorem{theorem}{Theorem}[section]
\newtheorem{corollary}[theorem]{Corollary}
\newtheorem{lemma}[theorem]{Lemma}

\theoremstyle{definition}

\newtheorem{example}[theorem]{Example}
\newtheorem{assumption}[theorem]{Assumption}
\newtheorem{remark}[theorem]{Remark}

\usepackage{hyperref}
\hypersetup{
  colorlinks   = true, 
  urlcolor     = blue, 
  linkcolor    = blue, 
  citecolor   = blue 
}

\allowdisplaybreaks

\newcommand{\ttl}{\Large Sample Path Moderate Deviation Principles for Queues with \\[5pt]Waiting-time Dependent Interarrival and Service Times }

\begin{document}
\title[]{\ttl}
\date{\today}
\author{Chang Feng$^*$}
\author{John J. Hasenbein$^*$}
\address{$^*$The University of Texas at Austin, 
Dept. of Mechanical Engineering, 
University Station C2200, Austin, TX, 78712-0292}
\email{chang.feng@utexas.edu}
\email{has@me.utexas.edu} 

\author{Guodong Pang$^{\dag}$}
\address{$^\dag$ Department of Computational Applied Mathematics and Operations Research,
	George R. Brown School of Engineering and Computing, 
	Rice University,
	Houston, TX 77005}
\email{gdpang@rice.edu}

\begin{abstract}
    We consider a single-server queue where interarrival and service times depend linearly and randomly on customer waiting times, and establish a sample-path moderate deviation principle (MDP) for the waiting time process. The waiting times for the queue can be written as a modified Lindley recursion with a random weight coefficient. Under a natural scaling of the random coefficients, we analyze the fluid behavior of the workload process and derive the stable equilibrium point, which can be zero or a positive value. The moderate-deviation-scaled process is centered around the stable equilibrium point and then represented as a linear stochastic differential equation driven by two random walks together with additional asymptotically negligible error terms and possibly a reflection at zero. The rate functions of MDPs in the two scenarios can be characterized explicitly, and they differ in that the case with zero centering term involves the linearly generalized Skorokhod reflection mapping while the case with positive centering term does not (similar to the corresponding diffusion limits). Our analysis involves the MDP for the associated linearly recursive Markov chains, invoking a perturbation of two independent random walks, and employing martingale techniques to prove the asymptotically exponentially vanishing  error terms.
\end{abstract}


\maketitle

\section{Introduction} \label{sec:intro}

In real-life queueing systems, both arrival processes and service times often depend on system congestion or delay. For example, empirical studies show that overcrowded emergency rooms (ERs) lose a portion of patients due to balking \citep{GreSGG06}. Similarly, when intensive care units (ICUs) are overloaded, physicians may accelerate patient throughput by transferring less severe cases to transitional care units or general wards \citep{ChaYE14}. Comparable workload-dependent behaviors also arise in biology, manufacturing, inventory management, computer networks, and insurance applications.

In this paper, we focus on one such type of workload-dependency structure introduced by  \cite{Whi90}, where the interarrival and service times depend linearly and randomly upon the customer waiting times. In this case, the waiting time of customers can be expressed through the Lindley-type stochastic recursion: 
\begin{equation} \label{eq:lindley_recursion}
    W_{i+1} = (C_i W_i + X_i)^+, \quad i \in \mathbb{N}_0,  
\end{equation} 
where we can interpret $X_i$ as the nominal increment variable and $C_i$ as the variable due to the linear dependence mentioned above (see Section \ref{subsec:the_model} for the precise definition). 
Our goal is to consider a sequence of such queues (indexed by $n$) and establish a sample-path moderate deviation principle (MDP) under various parameter regimes. Although we are motivated by queueing applications, the process defined through \eqref{eq:lindley_recursion} can be regarded more generally as a reflected AR(1) process with random coefficients. Thus, the results in this paper are widely applicable to many other areas of engineering and statistics. 

Most of the research on \eqref{eq:lindley_recursion} so far has focused on analyzing its transient and steady-state distributions using transform methods. \cite{BoxVla07} studied the case where $C_i$ is a Bernoulli-type random variable taking the values $\pm 1$. In this case, a large deviation result is also proved in \cite{VlaPal09} for the tail probabilities of the steady-state distribution. \cite{BoxMJ16} studied the reflected AR(1) process, which is the case when $C_i$ is deterministic. More recently, \citet{BoxLP21, Hua23, DimFie24} studied various cases in which $C_i$ takes on more general or more sophisticated forms. 

There are very limited results on approximations and limiting theorems for \eqref{eq:lindley_recursion}, at least at the sample-path level. \cite{Whi90} built on previous results of \cite{Ver79} and proved a functional central limit theorem (FCLT). However, the limiting diffusion process has a rather complicated form and was not given explicitly. \cite{BoxMJ16} proved an FCLT result for the reflected AR(1) process, in which the limiting diffusion there turned out to be a reflected Ornstein-Uhlenbeck (OU) process. 
 Recently, several sample-path large deviation principle (LDP) results have been established for models related to \eqref{eq:lindley_recursion}. \cite{BazEtal25} proved a sample-path LDP with sublinear rates for the conventional Lindley recursion (corresponding to $C_i = 1$). \cite{CBZ23} proved a sample-path LDP for the affine recursion $W_{i+1}=C_i W_i+X_i$ when the stationary distribution of $W_i$ has heavy tails. To our knowledge, our paper is the first to analyze a sample-path MDP for stochastic models governed by \eqref{eq:lindley_recursion}.

To gain analytical tractability, we make several key modeling choices. To establish an MDP or FCLT, the process needs to be centered around its functional law of large numbers (FLLN, or fluid) limit $\bar{W}$. We show that the fluid limit takes on a complicated form of an exponentially decaying (or growing) function. 
We shall restrict ourselves to the cases where the fluid limit is stable, and prove the MDP results for the moderate-deviation-scaled (MD-scaled) workload processes of the form
\begin{equation}\label{eq:w_tilde_intro}
    \widetilde{W}^n(t) = \frac{1}{b_n \sqrt{n}} (W^n_{\lfloor nt \rfloor} - n \bar{W}^*), \quad t \ge 0, 
\end{equation}
where $b_n$ is some scaling sequence satisfying the conditions in \eqref{eq:MDP_scaling_seq} and $\bar{W}^*$ is the stable fixed point of the fluid limit. In our analysis of the fluid limit's behavior, \(\bar{W}^*\) could be either $0$, 
or a positive value. This leads to different rate functions for the MDP. When \(\bar{W}^* = 0\), due to the non-negativity of $W^n$, the rate function involves optimization over paths that are regulated by a linearly generalized Skorokhod reflection mapping. However, for \(\bar{W}^* > 0\), the limiting path for \eqref{eq:w_tilde_intro} does not need to be regulated. This suggests that the behavior of \eqref{eq:w_tilde_intro} in the limit should be unaffected if the positive part operator in \eqref{eq:lindley_recursion} is removed.  

This provides motivation to establish an MDP for a linearly recursive Markov chain \(V^n\) that satisfies the stochastic recursion  
\begin{equation} \label{eq:intro_markov_chain}
    V^n_{i+1} =  C^n_i V^n_i + X^n_i , \quad i \in \mathbb{N}_0, 
\end{equation} 
with $V^n_0$ being a random variable. Here, we make another key modeling assumption by letting  
\begin{equation}\label{eq:intro_C_scaling}
    C_i^n = 1 - \frac{1}{n} \Theta_i, 
\end{equation}
where $\{\Theta_i\}_{i \in \mathbb{N}_0}$ is a fixed sequence of i.i.d.\ random variables. This type of scaling was used in \cite{BoxMJ16} to establish an FCLT for the reflected AR(1) process, with \(\Theta_i\) being deterministic. { The model defined by \eqref{eq:intro_markov_chain} and \eqref{eq:intro_C_scaling} is related to the recursive Markov systems studied by \cite{DupJoh2015}, where an MDP was established through a weak-convergence approach based on variational representations. However, this method does not seem to translate easily to proving the MDP for \eqref{eq:w_tilde_intro} due to reflection.}
We instead develop an alternative and more direct approach by establishing the MDP for the finite-dimensional distributions (this step is implicit in the MDP for random walks given in Theorem \ref{thm:MDP_RW}) and exponential tightness, together with an application of the contraction principle. 
This involves representing the MD-scaled process \(\widetilde{V}^n\) as a linear stochastic differential equation (SDE) driven by two independent random walks with several asymptotically negligible error terms, see equation \eqref{eq:V_md_representation}. The main technical difficulty is showing that the error terms, including a random walk with random coefficients, are exponentially equivalent to the zero process in space $\mathcal{D}_T$. To tackle this, we devise a sequence of arguments, which proceed in the order of Lemmas \ref{lem:exp_bounded_Vbar}, \ref{lem:exp_tight_esp_V_2}, Theorem \ref{thm:exp_tightness_V}, Corollary \ref{cor:V_bar} and Lemma \ref{lem:eps_V_expeq_0}. The proofs involve developing exponential bounds, showing exponential equivalence of various processes and utilizing martingale techniques, while relying upon characterizations of exponential tightness and exponential equivalence in $\mathcal{D}_T$ given in Appendix \ref{subsec:exp_tightness} and \ref{subsec:sup_exp_cvg_prob}. See for example the proof of Lemma \ref{lem:exp_tight_esp_V_2}. 

Next, we adapt the aforementioned approach to establish the MDP for $\widetilde{W}^n$ defined in \eqref{eq:w_tilde_intro}. The process $\widetilde{W}^n$ can also be represented as a linear SDE given by \eqref{eq:md_representation}. However, compared to \eqref{eq:V_md_representation}, this representation includes an additional term $\widetilde{L}^n$ arising from reflection at the origin. The first step is to establish an exponential stochastic boundedness property for the fluid-scaled process $\bar{W}^n$, stated in Lemma \ref{lem:W_bar_exp_stocastic_bdd}, which corresponds to Lemma \ref{lem:exp_bounded_Vbar} from the earlier analysis. This is achieved by bounding $\bar{W}^n$ using two linearly recursive Markov chains, defined in \eqref{eq:intro_markov_chain}, under different initial conditions. The previous arguments then apply directly to show that the error terms in \eqref{eq:md_representation} are exponentially equivalent to the zero process. We next analyze the additional term $\widetilde{L}^n$, which serves as the regulator process in the linearly generalized Skorokhod mapping when $\bar{W}^* = 0$, and is exponentially equivalent to zero when $\bar{W}^* \neq 0$. Finally, applying the contraction principle yields the MDP, from which the rate functions can be derived explicitly. 

We remark that a representation similar to \eqref{eq:md_representation} can be constructed for the diffusion-scaled wait-time or workload processes (with the same centering term as the MDP), 
which enables us to prove the FCLT results, with the diffusion limit being either an OU process in the case of a positive centering term or a reflected OU process in the case of zero centering. We provide the proofs for these results in Appendix \ref{app:diffusion_approx}, which complement the studies in  \cite{Whi90, BoxMJ16}.

Our work contributes to the limited literature on moderate deviations in queueing theory. For general overviews of MDPs for traffic processes and their connections to large deviations and central limit theorems, see \cite{wischik2001, GCW04, ShW95}. Sample-path MDPs have been established in several settings: GI/GI/1 queues \citep{Puh99}, cumulative fluid processes with many exponential on-off sources \citep{majewski2007}, workload processes in stochastic fluid queues with long-range dependent input \citep{CYZ1999}, infinite-server queues with time-varying service times modeled via shot-noise processes \citep{AP2024SNP}, GI/GI/N queues in the near Halfin-Whitt regime \citep{Puh23}, and GI/GI/1+GI queues \citep{FHP25}.

We also highlight several other closely related areas of the literature. One is the analysis of workload-dependent queues; see, for example, \cite{Har67, Cal73, Bri88, BroS92, BRBK04, BKNN11, Leg18}. Another is the study of the unreflected stochastic recursion given by \eqref{eq:intro_markov_chain}, commonly referred to in the literature as the “Vervaat perpetuity”; relevant results can be found in \cite{Kes73, Bra86, EmbGol94, GlaYao95, GolM01, Hor01, CBZ23}.

\subsection{Organization of the paper} 

In the rest of this section, we introduce the relevant terminologies and notation used in the sequel.  
In Section \ref{sec:model_results}, we formulate the queueing model and present the main MDP results. 
Section \ref{sec:fluid_analysis} is devoted to the analysis of the workload process under fluid scalings and identifying the stable fixed points of the limiting fluid equation. Section \ref{sec:MDP_LRMS} contains a moderate deviation analysis of a linearly recursive Markov system. Section \ref{sec:MDP_W} leverages the results and methods in the previous section to show the main theorems presented in Section \ref{sec:model_results}. Appendix \ref{app:fluid_proofs} contains proofs for the fluid approximation results in Section \ref{sec:fluid_analysis}. The FCLT results mentioned above are presented in Appendix \ref{app:diffusion_approx}. Finally, Appendix \ref{app:useful_facts} contains several useful facts that are used throughout the paper.

\subsection{Preliminaries and notations} \label{subsec:notations_definitions}

Throughout the paper, all random elements are implicitly defined on a probability space \( (\Omega, \mathcal{F}, \mathbb{P}) \). We also adopt the convention \(\mathbb{N}_0 \equiv \mathbb{N} \cup \{0\}\).

Given a Polish space \( \mathcal{X} \) with metric \( d(\cdot, \cdot) \), let \( \mathcal{B}(\mathcal{X}) \) denote the Borel \( \sigma \)-algebra. For a scaling sequence \( \{ a_n \}_{n \in \mathbb{N}} \) with \( a_n \uparrow \infty \), a family of \( \mathcal{X} \)-valued random elements \( \{ x_n\}_{n \in \mathbb{N}} \) is said to satisfy a \textit{large deviation principle} (LDP) in \( \mathcal{X} \) with rate $a_n$ and rate function $I: \mathcal{X} \to [0, \infty]$ if 
\begin{enumerate}
    \item[(i)] \( I \) is lower-semicontinuous and has compact level sets $\{x \in \mathcal{X}: I(x) \le a\}$, for all $a \ge 0$.  
    \item[(ii)] For all  $A \in \mathcal{B}(\mathcal{X})$,
        \begin{equation*} 
            - \inf_{x \in A^{\circ}} I(x) \le  \liminf_{ n \to \infty} \frac{1}{a_n} \log \mathbb{P} (x_n \in A ) 
            \le  \limsup_{n\to \infty} \frac{1}{a_n} \log \mathbb{P}  (x_n \in A) \le - \inf_{x \in \bar{A}} I(x).
        \end{equation*} 
\end{enumerate}

We say that the family $\{x_n\}_{n \in \mathbb{N}}$ is \textit{exponentially tight} in \( \mathcal{X} \) with rate $a_n$ if for all $\alpha \ge 0 $, there exists a compact set $K_\alpha \subset \mathcal{X}$ such that 
\begin{equation*}
    \label{eq:def_ExpTight}
    \limsup_{n \to \infty} \frac{1}{a_n} \log \mathbb{P}(x_n \notin K_\alpha) < - \alpha. 
\end{equation*}
Two families of \( \mathcal{X} \)-valued random elements \( \{ x_{n} \}_{n \in \mathbb{N}} \) and \( \{ y_{n} \}_{n \in \mathbb{N}} \) are said to be \textit{exponentially equivalent}  with rate \( a_n \) if for every \( \delta > 0 \), 
\begin{equation*}
    \lim_{n \to \infty} \frac{1}{a_n} \log \mathbb{P} \left( d(x_n, y_n) > \delta \right) = - \infty. 
\end{equation*}
A special case is when we let \( y_n = x_0 \in \mathcal{X} \) for all \( n \in \mathbb{N} \). Then we say that the $\{x_n\}_{n \in \mathbb{N}}$ \textit{converge super-exponentially in probability} to $x_0$ with rate $a_n$ and write \( x_n \stackrel{P^{1/a_n}}{\longrightarrow} x_0 \).

We shall fix \( T>0 \) and work in the function space \( \mathcal{D}_T \equiv \mathcal{D} ([0,T], \mathbb{R}) \) of c\`adl\`ag processes endowed with the \( J_1 \) Skorokhod topology, which is a Polish space. We use \(d_{J_{1}}\) to denote the \(J_{1}\) metric. The subspace \( \mathcal{C}_T \equiv \mathcal{C}([0,T], \mathbb{R}) \subset \mathcal{D}_T\) consists of processes with continuous paths. The subspace \( \mathcal{AC}_{0} \subset \mathcal{C}_T \) consists of processes with paths that are absolutely continuous and start from \( 0 \). For \( x \in \mathcal{D}_T \), the uniform norm is denoted \( \left\| x \right\|_{T} := \sup_{t \in [0,T]} \left\vert x(t) \right\vert  \) and the supremum map is denoted \( x_{\uparrow}(t) := \sup_{u \in [0,t]} x(u) \).  We use \( \mathfrak{e} \) to denote the identity process, that is, \( \mathfrak{e}(t) = t \) for all $t \ge 0$. When the context is clear, we use $0$ to denote the zero process. 

For the family of processes \( \{ x_n \}_{n \in \mathbb{N}} \) with paths in \( \mathcal{D}_T \), the sample-path LDPs and sample-path MDPs are differentiated through the choice of the scaling sequence \( a_n \) and the scalings used to define \( x_n \). Specifically, in this paper, by a \textit{sample-path moderate deviation principle} (MDP), we mean a sample-path LDP for a centered, MD-scaled process \( \widetilde{x}^{n} \) with speed \( b_n^2 \), where \(b_{n} \to \infty\) and \( b_{n}/\sqrt{ n } \to 0\).   For our MDP results, see Section \ref{subsec:MD_results} for the specific scaling sequence and processes under consideration. We mention that, to highlight the scalings used, the convention in this paper is to use $\bar{x}^n$ for the FLLN-scaled/fluid-scaled processes and $\hat{x}^n$ for the FCLT-scaled (diffusion-scaled) process. The MD-scaled processes are denoted by $\widetilde{x}^n$.

Lastly, we say that  the family \( \{ x_n \}_{n \in \mathbb{N}} \) is \textit{\( \mathcal{C} \)-exponentially tight} with rate \( a_n \) if \( \{ x_n \}_{n \in \mathbb{N}} \) is exponentially tight in \(\mathcal{D}_T\) with rate \( a_n \) and for any \(\delta > 0\), 
\[
    \limsup_{n \to \infty} \frac{1}{a_n} \log \mathbb{P} \left( \sup_{t \in [0,T]} |x_n(t) - x_n(t-)| > \delta \right) = - \infty. 
\]
See Appendix \ref{subsec:exp_tightness} for characterizations of exponential tightness and some further discussions. To facilitate navigation, we collect other main notations used throughout the paper in Appendix \ref{app:notation_table}, together with a brief description and the location of its first appearance.

\section{Model Formulation and Main Results} \label{sec:model_results}

\subsection{The model}\label{subsec:the_model} 
Consider a sequence of single server queues under the FIFO service discipline indexed by $n \in \mathbb{N}$. For the $n$-th queue, 
let $\{(\mathfrak{A}^n_i, \mathfrak{S}^n_i ,  A^n_i, B^n_i),\; i \in \mathbb{N}_0\} $ be an i.i.d.\ sequence of random vectors, where for each $i$, we assume that $\mathfrak{A}_i^n, \mathfrak{S}_i^n, A_i^n, B_i^n$ are mutually independent for simplicity. Let \( W_i^n \) denote the waiting time of \( i \)-th customer and define 
\begin{equation*}
\begin{aligned}
    \mathfrak{A}'^{,n}_i = \mathfrak{A}^n_i + A^n_i W^n_i , \\
    \mathfrak{S}'^{,n}_i = \mathfrak{S}^n_i + B^n_i W^n_i . 
\end{aligned}
\end{equation*}
We shall interpret \( \mathfrak{A}'^{,n}_i \) as the interarrival time between customers \(i\) and \(i+1\), and \( \mathfrak{S}'^{,n}_i\) as the service time of customer \(i\). By definition, the interarrival and service times depend {\it linearly} and {\it randomly} upon the waiting times. Similar to \cite{Whi90}, we shall call \(\mathfrak{A}^n_i\) the {\it nominal interarrival time} and $\mathfrak{S}^n_i$ the {\it nominal service time}. Note that if the state-dependent terms are omitted, they would be the actual interarrival and service times, and we revert to the conventional GI/GI/1 queue. 

For single server queues, the waiting times of customers satisfy a recursion of Lindley type. For our model, it is given by
\begin{equation} \label{eq:lindley_recurison_W}
    W^n_{i+1} = (W_i^n + \mathfrak{S}'^{,n}_i - \mathfrak{A}'^{,n}_i )^+, \quad i \in \mathbb{N}_0,  
\end{equation}
with $W^n_0$ being a non-negative random variable. If we define
\begin{equation*}
    \begin{aligned}
    X^n_i &= \mathfrak{S}^n_i - \mathfrak{A}^n_i , \\
    C_i^n &= 1 + B^n_i - A^n_i,
\end{aligned}
\end{equation*}
then we can further simplify \eqref{eq:lindley_recurison_W} by writing 
\begin{equation}\label{eq:recursion_W}
    W^n_{i+1} = (C^n_i W^n_i + X^n_i)^+ .
\end{equation}

It is of interest to analyze the system's behavior under different parameter regimes. 
Define the \textit{nominal traffic intensity} by \(\rho_n= \mathbb{E}[\mathfrak S_0^n] / \mathbb{E}[\mathfrak A_0^n]\). 
Then, the nominal load condition of the system is classified as follows: 
\begin{itemize}
    \item the system is \textit{overloaded} if \(\rho_n>1\), i.e., if \(\mathbb{E}[X_0^n]>0\) or \(\mathbb{E}[\mathfrak{S}^n_{0}] > \mathbb{E}[\mathfrak{A}^n_{0}]\);
    \item the system is \textit{critically loaded} if \(\rho_n=1\), i.e., if \(\mathbb{E}[X_0^n]=0\) or \(\mathbb{E}[\mathfrak{S}^n_{0}] = \mathbb{E}[\mathfrak{A}^n_{0}]\);
    \item the system is \textit{underloaded} if \(\rho_n<1\), i.e., if \(\mathbb{E}[X_0^n]<0\) or \(\mathbb{E}[\mathfrak{S}^n_{0}] < \mathbb{E}[\mathfrak{A}^n_{0}]\).
\end{itemize}
The coefficient \(C_i^n\)
captures the net effect of the waiting-time dependence. We classify the average feedback effect of the system as follows: 
\begin{itemize}
    \item the system is \textit{feedback-damping} if \(\mathbb{E}[C_0^n]<1\), i.e., if \(\mathbb{E}[A_0^n]>\mathbb{E}[B_0^n]\);
    \item the system is \textit{feedback-neutral} if \(\mathbb{E}[C_0^n]=1\), i.e., if \(\mathbb{E}[A_0^n]=\mathbb{E}[B_0^n]\);
    \item the system is \textit{feedback-amplifying} if \(\mathbb{E}[C_0^n]>1\), i.e., if \(\mathbb{E}[A_0^n]<\mathbb{E}[B_0^n]\).
\end{itemize}
Thus \(X_i^n\) describes the nominal net input, while \(C_i^n\) describes the net feedback effect created by the dependence of future interarrival and service times on the current waiting time.
While keeping in mind the original definitions, it suffices to only work with \eqref{eq:recursion_W} for the rest of the paper. 

The first step is to construct sample paths for the waiting time process in space \(\mathcal{D}_T\). The positive-part operator in \eqref{eq:recursion_W} can be written as:  
\begin{equation*}
    W_{i+1}^n = C^n_i W^n_i + X^n_i + \Psi^n_i ,
\end{equation*}
where $\Psi_i^n = - \min (C_i^n W_i^n + X_i^n, 0 )$. Then we have 
\begin{equation*}
    W^n_{i+1}- W^n_i = (C^n_i-1) W^n_i + X^n_i + \Psi_i^n . 
\end{equation*}
By telescoping the sum on the left and defining $L^n_i = \sum_{j=0}^{i} \Psi_j^n$, we obtain the following representation for the waiting times: 
\begin{equation}\label{eq:waiting_time_representation}
    W_i^n = W_0^n + \sum_{j=0}^{i-1} X^n_j + \sum_{j=0}^{i-1} (C^n_j - 1) W^n_j + L^n_{i-1}, \quad i \in \mathbb{N}_0, 
\end{equation}
with the convention that an empty sum is equal to $0$ (when $i = 0$). 
To formulate limit theorems, we re-index the discrete time process by $\lfloor nt \rfloor,\; t \in [0,T]$ so that its sample paths lie in the space $\mathcal{D}_T$. The waiting time process can then be written as  
\begin{equation}\label{eq:process_W}
    W^{n}_{\lfloor nt \rfloor } = W^{n}_0 + \sum_{i=0}^{\lfloor nt \rfloor -1} X_i^{n} + \sum_{i=0}^{\lfloor nt \rfloor -1} (C_i^{n}- 1) W^{n}_i + L^{n}_{\lfloor nt \rfloor -1}, \quad t \in  [0,T ].  
\end{equation}
Denote $ W^{n}(t) := W^{n}_{\left\lfloor nt \right\rfloor} $ and $ L^{n}(t ) := L^{n}_{\left\lfloor nt \right\rfloor-1} $.  From the definitions, we have $ L^{n}(0 ) = 0 $ and $ L^{n} (t ) \ge 0 $ for all $t$. Moreover, since $ 1_{ \{ \Psi^{n}_{i} > 0 \}} 1_{ \{ W^{n}_{i+1} > 0 \}} = 0 $  for all $ i \in \mathbb{N}_{0} $, it follows that 
\begin{equation} \label{eq:L_regulator}
    \int_{0}^{t} W^{n}(s ) \textrm{d} L^{n}(s ) = \int_{0}^{t} W^{n}_{ \left\lfloor ns \right\rfloor} \textrm{d} L^{n}_{\left\lfloor ns \right\rfloor-1} = \sum_{i=0}^{\left\lfloor nt \right\rfloor-1} W^{n}_{i+1} \Psi^{n}_{i} = 0. 
\end{equation}
Recall that the workload process $ W^{n} $ is nonnegative, therefore we may  interpret $ L^{n} $ as a type of regulator process enforcing reflections at zero. 

In \eqref{eq:process_W}, the waiting time process is driven by two random walks, one of which is randomly weighted by the customer waiting times. Note that the random walks are dependent, since the random weights \(W^n_i\) depend on \(C^n_k\) and \(X^n_k\), for all $0 \le k < i$. 
We shall make several modeling assumptions regarding these random walks.  

\begin{assumption}[Model Assumptions] 
    \label{assump:model}
    \leavevmode
    \begin{enumerate} 
        \item[(i)] Let \( \{\Theta_{i},\; i \in \mathbb{N}_0\} \) be an i.i.d.\ sequence of random variables with finite mean \( \theta \) and variance \( \sigma^{2}_{\Theta} \).  
        \item[(ii)] For each \(n \in \mathbb{N}\), the family of random variables \( \{C^{n}_{i},\; i \in \mathbb{N}_0\} \) has the form  
        \[
            C^{n}_{i} = 1 - \frac{1}{n} \Theta_{i}, \quad \forall i \in \mathbb{N}_0. 
        \]
        \item[(iii)] For each \(n \in \mathbb{N}\), the family of variables \( \{X^{n}_{i}, i \in \mathbb{N}_0\} \) is an i.i.d.\ sequence that is independent of the sequence \( \{ \Theta_i,\ i \in \mathbb{N}_0 \} \). Further, we assume $X^n_0$ has finite mean \( \mu_n \) and variance \( \sigma^{2}_{X,n} \) such that \( \mu_n \to \mu\) and \( \sigma^{2}_{X,n} \to \sigma^{2}_X \) as \( n \to \infty \) for some \( \mu \in \mathbb{R}\) and \( \sigma_X > 0\). 
    \end{enumerate}
\end{assumption}

\begin{remark}
     In Assumption \ref{assump:model}, 
     we assumed a very specific form of scalings for the random variables \( C^{n}_{i} \). This can be viewed as an extension of the scalings used in \cite{BoxMJ16}, which established the FCLT results for $C^n_i = 1 - \frac{\alpha}{n}$, where $\alpha$ is a constant.  In other literature,  for example \cite{Ver79} and \cite{Whi90}, the FCLT results were established for \( C^{n}_{i} = (C_i)^{1/n} \), where \( \{ C_i, i \in \mathbb{N}_0 \} \) is some i.i.d.\ sequence. The techniques used  were to analyze the random walk $n^{-1} \sum_i \log C_i$. These two types of scalings can be seen as close approximations to each other. Specifically, if we let \( - \Theta_i = \log C_i \), then when \( n \) is large, 
        \[
            (C_i)^{1/n} = e^{\frac{1}{n} \log C_i} = e^{ - \frac{1}{n} \Theta_i} \approx e^{\log(1-\frac{\Theta_i}{n} )} = 1- \frac{1}{n}\Theta_i. 
        \]
    Further, using our FCLT results in Appendix \ref{app:diffusion_approx}, we recover the same approximation for the stationary distribution given in \cite{Whi90}; see Remark \ref{rmk:fclt_comparison}. This justifies our choice of scalings for \(C^n_i\).   

    In relation to the original queueing dynamics \eqref{eq:lindley_recurison_W}, the scaling in Assumption 1(ii) corresponds to \(A_i^n-B_i^n= \Theta_i / n\). Thus, if \(\theta > 0\), the system is feedback-damping, but the per-customer feedback effect is of order \(1/n\). Similarly, the system is feedback-amplifying if \(\theta < 0\). In Assumption 1(iii), independence between \( \{ X^n_{i} \}_{i}\) and \( \{ \Theta_{i} \}_{i}\) is implied by independence between the nominal variables \( \{(\mathfrak{A}^n_{i}, \mathfrak{S}^n_{i}) \}_{i}\) and the feedback variables \( \{(A^n_{i}, B^n_{i}) \}_{i}\). Further, by definition of \(X^n_{i}\), the limiting system is overloaded if \( \mu > 0\). Similarly, it is critically loaded if \(\mu=0\), and underloaded if \(\mu < 0\). 
    
\end{remark}

Lastly, for each \( n \in \mathbb{N} \), we also define a  filtration \( \{ \mathcal{F}^{n}_i, i \in \mathbb{N}_0 \} \) where 
\begin{equation}
    \mathcal{F}^{n}_i = \sigma( W_0^{n}, (\Theta_0, X_0^{n}), (\Theta_1, X_1^{n}) \ldots, (\Theta_i, X_i^{n}) ). 
\end{equation}
In particular, by this definition, the waiting time $W_i^{n}$ is $\mathcal{F}^{n}_{i-1}$-measurable.

\subsection{Moderate Deviation Results} \label{subsec:MD_results}

We first introduce the scaling sequence \( \{b_n, n \in \mathbb{N}\}\) that satisfies 
\begin{equation}\label{eq:MDP_scaling_seq}
    b_n \to \infty \,\, \textrm{ and } \,\, \frac{b_n}{\sqrt{n}} \to 0, \quad \textrm{as } n \to \infty.
\end{equation}
We study moderate deviations of the centered and rescaled process of the form
\begin{equation}\label{eq:W_tilde}
    \widetilde{W}^{n}(t) = \frac{\sqrt{n}}{b_n} \left( \bar{W}^{n}(t) -  \bar{W}^{*} \right) . 
\end{equation}
In \eqref{eq:W_tilde}, $\bar{W}^n$ is the fluid-scaled process defined by 
\begin{equation} \label{eq:fluid_scaling_W}
    \bar{W}^{n}(t) = \frac{1}{n} W^{n}_{\lfloor nt \rfloor}, \quad t \in [0,T].   
\end{equation}
We show in Section \ref{sec:fluid_analysis} that \(\bar{W}^n\) converges to a fluid limit \(\bar{W}\). Restricting attention to  cases where \(\bar{W}\) is stable, we take the constant \( \bar{W}^*\) in \eqref{eq:W_tilde} to be the stable fixed point of the fluid limit.

An MDP established for \eqref{eq:W_tilde} describes order \(\mathcal{O}(b_n /\sqrt{n})\) deviations of \( \bar{W}^{n}\) from the fluid equilibrium \(\bar{W}^{*}\). Equivalently, in the original waiting-time scale, these deviations are of the form
\[
    W^{n}_{\lfloor nt \rfloor} - n\bar{W}^{*}
    = \mathcal{O}(b_n\sqrt{n}). 
\] 
Thus, the MDP scale lies between the diffusion and fluid scales: the deviations are larger than the usual diffusion-scale fluctuations of order \( \sqrt{n} \), but smaller than fluid-scale deviations of order \( n \). We emphasize that the results below differ from sample-path LDPs for the fluid-scaled process \( \bar{W}^{n} \) with rate \( n \), which would concern order \(\mathcal{O}(1)\) deviations of \( \bar{W}^{n} \) from its fluid limit.

Now, denoting $\bar{W}^n_0 := n^{-1}W^n_0$ and \(\widetilde{W}^n_0 := b_n^{-1} \sqrt{n} (\bar{W}^n_0 - \bar{W}^*)\), 
we impose the following conditions for the MDP results.

\begin{assumption}[MDP Assumptions]
    \label{assump:MDP_W} 
    \leavevmode
    \begin{enumerate}
        \item[(i)] For all \(n \in \mathbb{N}  \), let the random variable \(W^n_0 \ge 0\) a.s.\ and  
        let \( \widetilde{W}^n_0 \stackrel{P^{1/b_n^2}}{\longrightarrow} w_0 \) for some \(w_0 \in \mathbb{R}\). 
        \item[(ii)] For some $a > 0$, 
        \[
            \mathbb{E}\left[ e^{a |\Theta_0|} \right] < \infty, \quad \text{and}\quad \sup_{n \in \mathbb{N}} \mathbb{E}\left[ e^{a |X^n_0|} \right] < \infty.
        \]
        \item[(iii)] The sequence \( \frac{\sqrt{n}}{b_n} (\mu_n - \mu )  \to  r  \in \mathbb{R}\) as \( n \to  \infty \). 
    \end{enumerate}
\end{assumption}
We will see in Section \ref{sec:MDP_W} that, similar to \eqref{eq:process_W}, the MD-scaled process $\widetilde{W}^n$ is related to the following two random walks:  
\begin{equation} \label{eq:RW_MD_scaling}
        \widetilde{R}_{X}^{n}(t) \equiv \frac{1}{b_n \sqrt{n}} \sum_{i=0}^{\left\lfloor nt \right\rfloor-1} (X^{n}_{i} - \mu_{n}) , 
        \quad 
        \widetilde{R}_{\Theta}^{n}(t) \equiv \frac{1}{b_n \sqrt{n}} \sum_{i=0}^{\left\lfloor nt \right\rfloor-1} (\Theta_{i} - \theta ) ,
        \quad t \in [0,T].  
    \end{equation}
The sample-path MDP for random walks has been proven under more general conditions than Assumption \ref{assump:MDP_W} (ii). 
For example, see \cite{AP24} and Theorem 6.1 in \cite{PuhW97}. 
Here, we simply state the result. 
\begin{theorem}[MDP for Random Walks]
    \label{thm:MDP_RW}
    Under Assumptions \ref{assump:model} (i), (iii) and \ref{assump:MDP_W} (ii), 
    the \( \mathcal{D}_T\)-valued families of processes \( \{\widetilde{R}_{X}^{n},\ n \in \mathbb{N}\}  \) and \( \{\widetilde{R}_{\Theta}^{n},\ n \in \mathbb{N} \} \) respectively satisfy MDPs in \(\mathcal{D}_T\) with rate \( b_n^{2} \) and rate function \( I_X\) and \( I_\Theta \), where 
    \begin{equation*}
    \begin{aligned}
        I_{X}(\phi) &= 
        \begin{cases} \frac{1}{2  \sigma^{2}_{X}} \int_{0}^{T} \vert \dot{\phi} (t) \vert ^{2} \mathrm{d}t , & \phi \in \mathcal{AC}_0, \\ \infty  , & \text{otherwise.} \end{cases}
        \\
        I_{\Theta}(\phi) &= 
        \begin{cases} \frac{1}{2  \sigma^{2}_{\Theta}} \int_{0}^{T} \vert \dot{\phi}(t) \vert ^{2} \mathrm{d}t , & \phi \in \mathcal{AC}_0, \\ \infty  , & \text{otherwise.} \end{cases}
    \end{aligned}
    \end{equation*}
\end{theorem}

To state the main results, we briefly introduce several continuous mappings on \(\mathcal{D}_T\). For \( x \in \mathcal{D}_{T}\), let \( \mathcal{M}_{\theta}(x) = u \) denote the solution  to the integral equation 
\[
 u(t) = x(t ) - \int_{0}^{t} \theta u(s) \, ds, \qquad t \in [0,T].   
\]
We use \((\mathcal{R} , \mathcal{R}' ) \) to denote the conventional one-dimensional Skorokhod map (see \cite{ChenY01} Chapter 6.2) and let \((\mathcal{R}_{\theta}, \mathcal{R}'_{\theta} )(x) = (z,l)\)  denote the one-dimensional linearly generalized Skorokhod map, which satisfies
\[
z(t) = x(t) - \int_{0}^{t} \theta z(s) \, ds + l(t) \ge 0, \qquad t \in [0,T], 
\]
with \(l\) being non-decreasing, \(l(0)=0\), and \( \int_{0}^t z(s) \mathrm{d}l(s)=0\) for all \(t \in [0,T]\). We note here that \(\mathcal{M}_{\theta}\), \((\mathcal{R}_{\theta},\mathcal{R}'_{\theta})\) and \((\mathcal{R} ,\mathcal{R}' )\) are well-defined and continuous under the \(J_{1}\) topology. We refer the reader to Appendix \ref{subsec:LGRM} for further details and proofs of their continuity properties. Now, we state the main results of this paper.

\begin{theorem}
    \label{thm:MDP_W}
    Under Assumptions \ref{assump:model} and \ref{assump:MDP_W}, the family \( \{ \widetilde{W}^{n},\ n \in \mathbb{N} \} \) satisfies an MDP in \(\mathcal{D}_T\) with rate \( b_n^{2} \) and rate function \(I\), where  
    \begin{enumerate}
        \item[(i)] if \( \mu > 0,\ \theta > 0 \) and \( \bar{W}^* = \mu / \theta \), then  
        \[
            I(\phi) = \inf_{\substack{\psi_1, \psi_2 \in \mathcal{D}_T, \\ \phi  = \mathcal{M}_\theta (w_0 + \psi_1 - \frac{\mu}{\theta} \psi_2 + r\mathfrak{e}).}}   I_X(\psi_1) + I_\Theta(\psi_2) ;
        \]
        \item[(ii)] if \( \mu = 0,\ \theta \ge 0 \) and \( \bar{W}^* = 0 \), then 
        \[
            I(\phi) = \inf_{\substack{\psi_1 \in \mathcal{D}_T, \\ \phi  = \mathcal{R}_\theta (w_0 + \psi_1  + r\mathfrak{e}).}}   I_X(\psi_1);
        \] 
    \end{enumerate}
    Furthermore, for the case where \( \mu < 0 \), \(\bar{W}^* = 0\) and \(w_{0 } = 0\), we have 
    \[
        \widetilde{W}^{n} \stackrel{P ^{1/b_n^{2}}}{\longrightarrow} 0. 
    \]
\end{theorem}
The rate functions in Theorem \ref{thm:MDP_W} are given as optimization problems. Due to the simple structures of the rate functions in Theorem \ref{thm:MDP_RW}, our next result shows that these optimization problems can be explicitly solved. 

\begin{theorem}
    \label{thm:MDP_W_explicit_rate_fn}
    Under Assumptions \ref{assump:model} and \ref{assump:MDP_W},  
    the rate functions in Theorem \ref{thm:MDP_W} take the following form:  
    \begin{enumerate}
        \item[(i)] suppose \( \mu > 0,\ \theta > 0 \) and \( \bar{W}^* = \mu / \theta \), then 
        \begin{equation*}
            I(\phi) = \frac{\theta^{2}}{2(\theta^{2}\sigma^{2}_{X}+ \mu^{2} \sigma^{2}_{\Theta})} \int_{0}^{T} (\dot{\phi}(t) - r + \theta \phi(t))^{2} \mathrm{d} t\,,
        \end{equation*}
        for $\phi \in \mathcal{AC}$ with  \(\phi(0) = w_0\). Otherwise, $I(\phi) = \infty$. 
        
        \item[(ii)] suppose \( \mu = 0,\ \theta \ge 0 \) and \( \bar{W}^* = 0 \), then 
        \begin{equation*}
            I(\phi) = \frac{1}{2\sigma^{2}_{X}}\int_{0}^{T}  1_{ \{ \phi(t) > 0 \}} (\dot{\phi}(t) - r + \theta \phi(t))^{2} \mathrm{d}t + \frac{1_{ \{ r > 0 \}}r^{2}}{2\sigma^{2}_{X}}  \int_{0}^{T} 1_{ \{ \phi(t) = 0 \}}  \mathrm{d}t\,, 
        \end{equation*} 
        for $\phi \in \mathcal{AC}$ with \(\phi\) non-negative and \(\phi(0) = w_0\). Otherwise, $I(\phi) = \infty$. 
    \end{enumerate}

\end{theorem}

Being able to obtain explicit rate functions in Theorem \ref{thm:MDP_W_explicit_rate_fn} is useful for estimating rare event probabilities, which we illustrate through the next example. 

\begin{example}

Consider an overloaded and feedback-damping system ($ \mu > 0 $ and $ \theta > 0 $) with 
\[
    \mathfrak{A}'^{,n}_{i} = \mathfrak{A}^{n}_{i} + \frac{1}{n} \Theta_{i} W^{n}_{i}, \quad   \quad \mathfrak{S}'^{,n}_{i} = \mathfrak{S}^{n}_{i}.  
\]
In this model, long waiting times slow down future arrivals, possibly through customer balking, deferrals, or routing away from the system by the controller. By Theorem \ref{thm:MDP_W_explicit_rate_fn}(i), assuming $ w_0 = 0 $ and $ r = 0 $ for simplicity, we have the approximation 
\begin{equation} \label{eq:toy_example_approximation}
    \mathbb{P} \left( \sup_{0\le t\le T} W_{\lfloor nt\rfloor}^n \ge n\frac{\mu}{\theta}+a b_n\sqrt n \right) \approx \exp \left\{ -b_n^2 \inf_{\phi\in\mathcal A_a} I(\phi) \right\}, 
\end{equation}
where  
\[
    \mathcal A_a =
    \left\{
        \phi\in \mathcal{AC}_{0}:\sup_{0\le t\le T}\phi(t)\ge a
    \right\}. 
\]
Further, computation of the infimum in \eqref{eq:toy_example_approximation} can be formulated as a variational problem. Let $ \sigma^{2} = \sigma^{2}_{X} + (\mu / \theta)^{2} \sigma^{2}_{\Theta} $. Then  
\begin{equation}
    \inf_{\phi \in \mathcal{A}_{a}} I(\phi) = J_{a}(T) = \inf_{0 < \tau \le T} \inf_{\phi, u} \left\{ \frac{1}{2 \sigma^{2}} \int_{0}^{\tau} u^{2}(t) \mathrm{d}t:\ \dot{\phi}(t) = u(t)  - \theta \phi(t),\ \phi(0) = 0,\ \phi(\tau) = a   \right\}. 
\end{equation} 
When viewed as an optimal control problem, consider the controlled dynamics 
\[
    \dot{\phi}(t ) = u(t)  - \theta \phi(t), \quad \phi(0 ) = 0,  
\]
with $ u(\cdot) $ being the control. The Hamiltonian is 
\[
    H(\phi, u , p ) = \frac{1}{2 \sigma^{2}} u^{2} + p (  - \theta \phi + u ).  
\]
By Pontryagin maximum principle, the adjoint equation satisfies $ \dot{p}(t) = \theta p(t) $, which yields the solution 
\[
    p(t ) = C e^{\theta t}, 
\]
with $ C $ being some constant. Optimizing the Hamiltonian yields the optimal control as 
\[
    u(t) = - \sigma^{2} p(t) = - C \sigma^{2} e^{\theta t},
\]
and the optimal trajectory can be solved as having the form 
\[
    \phi(t ) =   - \frac{C \sigma^{2}}{\theta} \sinh (\theta t).   
\]
Suppose such a trajectory $ \phi_{\tau} $ hits level $ a $ at time $ \tau $ (i.e.\ $ \phi_\tau(\tau) = a $), then for this trajectory, 
\[
    C_{\tau} = \frac{ - a \theta}{ \sigma^{2} \sinh(\theta \tau)},  \quad  \quad u_{\tau}(t) =   \frac{a \theta  }{ \sinh(\theta \tau )  }  e^{\theta t}, 
\]
and we can compute the optimal cost for hitting $ a $ at time $ \tau $ as 
\[
    J_{a, \tau} = \frac{ a^{2}\theta  }{ \sigma^{2} (1 - e^{-2 \theta \tau})}.  
\]
Observe that $ J_{a, \tau} $ is decreasing in $ \tau $. Therefore
\[
    J_{a}(T) = \inf_{0 < \tau \le T} J_{a, \tau} = \frac{ a^{2}\theta  }{ \sigma^{2} (1 - e^{-2 \theta T})}. 
\]
Plugging the solution into \eqref{eq:toy_example_approximation}, we obtain the following explicit rare event probability formula 
\[
    \mathbb{P} \left( \sup_{0\le t\le T} W_{\lfloor nt\rfloor}^n \ge n\frac{\mu}{\theta}+a b_n\sqrt n \right) \approx \exp \left\{ -b_n^2 \frac{ a^{2}\theta^{3} }{ (\theta^{2} \sigma^{2}_{X} + \mu^{2} \sigma^{2}_{\Theta}) (1 - e^{-2 \theta T})} \right\}. 
\]

\end{example}

\section{Fluid Analysis} \label{sec:fluid_analysis}

\subsection{Fluid Limit} \label{subsec:fluid_limit}

Consider the fluid-scaled process \( \bar{W}^{n} \) given by \eqref{eq:fluid_scaling_W}. 
With a slight abuse of notation, we shall also write $\bar{W}^{n}_i = n^{-1} W^{n}_i$.   
Starting with \eqref{eq:process_W}, we can approximate the sum involving $\bar{W}^{n}_i$ by an integral:  
\begin{equation}
    \bar{W}^{n}(t) = \bar{W}^{n}_0 + \frac{1}{n} \sum_{i=0}^{\lfloor nt \rfloor -1} X_i^{n} - \int_{0}^{t} \theta \bar{W}^{n}(s) ds  + \bar{\epsilon}^{1,n}(t) + \bar{\epsilon}^{2,n}(t) +  \frac{1}{n}  L^{n}_{\lfloor nt \rfloor - 1} , \label{eq:W_fluid_representation}
\end{equation}
while introducing two error terms given by  
\[
\begin{aligned}
    \bar{\epsilon}^{1,n}(t) &= \theta \left(  \int_{0}^{t} \bar{W}^{n}(s) ds - \frac{1}{n} \sum_{i=0}^{\lfloor nt \rfloor -1} \bar{W}^{n}_i \right), \\
    \bar{\epsilon}^{2,n}(t) &= \frac{1}{n} \sum_{i=0}^{\lfloor nt \rfloor -1} (\theta - \Theta_i)  \bar{W}^{n}_i.
\end{aligned} 
\]
We can simplify \eqref{eq:W_fluid_representation} using the linearly generalized Skorokhod mapping \( (\mathcal{R}_\theta, \mathcal{R}'_\theta) \). Let $ \bar{L}^{n}(t) := n^{-1}L^{n}_{\left\lfloor nt \right\rfloor-1} $. Then \eqref{eq:L_regulator} implies that $ \bar{L}^{n}(0) = 0 $, $ \bar{L}^{n}(t) \ge 0 $ and $ \int_{0}^{t} \bar{W}^{n}(s )  \mathrm{d} \bar{L}^{n}(s ) = 0 $, for all $ t \in [0,T] $. Since $ \bar{W}^{n} \ge 0 $, it follows that 
\begin{equation}\label{eq:fluid_W_reflected}
    (\bar{W}^{n}, \bar{L}^{n})  = (\mathcal{R}_\theta, \mathcal{R}_\theta') \bigg( \bar{W}^{n}_0 + \frac{1}{n} \sum_{i=0}^{\left\lfloor n \cdot \right\rfloor-1} X_i^{n}  + \bar{\epsilon}^{1,n} + \bar{\epsilon}^{2,n}  \bigg). 
\end{equation}

The next lemma shows that the error terms are asymptotically negligible. For the proof, see Appendix \ref{subsec:pf_fluid_error}. 
\begin{lemma} \label{lem:fluid_error}
   Let Assumption \ref{assump:model} hold and $\bar{W}^n_0 \to w_0$ in $L^2$ as $n \to \infty$, then \( \|\bar{\epsilon}^{1,n}\|_{T} \) and \( \|\bar{\epsilon}^{2,n}\|_{T} \) converge to \( 0 \)  in probability.   
\end{lemma}

Now we are ready for the fluid limit result. 

\begin{theorem}\label{thm:fluid_limit}
    Let Assumption \ref{assump:model} hold and $\bar{W}^n_0 \to w_0$ in $L^2$ as $n \to \infty$, then 
    \[
        \|\bar{W}^{n} - \bar{W} \|_{T} \to 0 \quad  \text{ in probability},
    \]
    where 
    \begin{equation}\label{eq:fluid_model}
        \bar{W} = \mathcal{R}_\theta \left( w_0 + \mu \mathfrak{e}   \right).
    \end{equation}
\end{theorem}

\begin{proof}
    Denote $ \bar{R}^{n}_{X}(t) := n^{-1} \sum_{i=0}^{\left\lfloor nt \right\rfloor-1} X^{n}_{i} $, $ t \in [0,T] $. By Lemma \ref{lem:FWLLN}, $ \|\bar{R}^{n}_{X} - \mu \mathfrak{e}\|_{T} \to 0 $ in probability.  
    Then by \eqref{eq:fluid_W_reflected}, Lemma \ref{lem:fluid_error} and an application of the continuous mapping theorem, we obtain the desired fluid limit. 
    This concludes the proof. 
\end{proof}

\subsection{Stability of the fluid limit equation} \label{sec:fluid_stability}

To determine the centering term \( \bar{W}^*\) that appears in \eqref{eq:W_tilde}, we examine the stability of fixed points in the fluid limit. It is useful to write  \eqref{eq:fluid_model} in differential form: 
\begin{equation}\label{eq:fluid_model_differential_form}
    \mathrm{d} \bar{W}(t) = \mu - \theta \bar{W}(t) + \mathrm{d} \bar{L}(t)\,. 
\end{equation}

We separately analyze the behavior of \eqref{eq:fluid_model_differential_form} under the different parameter regimes introduced in Section \ref{subsec:the_model}. Recall that the system is  \textit{overloaded} when \( \mu > 0 \), \textit{critically loaded} when \( \mu = 0 \), and \textit{underloaded} when \( \mu < 0 \). Further, the system is \textit{feedback-damping} when \( \theta > 0 \), \textit{feedback-neutral} when \( \theta = 0 \), and \textit{feedback-amplifying} when \( \theta < 0 \).

\subsubsection*{Overloaded system ($\mu > 0$)}
\begin{enumerate}
    \item[(a)] Consider when $\theta > 0$. In this case, the regulator \( \bar{L} \) for the Skorokhod mapping is never activated. To see this, observe that whenever \( 0 \le  \bar{W}(t) < \mu / \theta \), we have \( \mu - \theta \bar{W}(t) > 0 \), which implies that \( \mathrm{d}\bar{W}(t) > 0 \). This positive drift drives the process upward, preventing it from hitting zero.
    We can therefore determine the fluid limit by solving the unreflected differential equation
    \begin{equation} \label{eq:fluid_limit_diffeq}
        \begin{cases} \frac{\mathrm{d}}{\mathrm{d} t} \bar{W}(t) = \mu - \theta \bar{W}(t),  \\ \bar{W}(0) = w_0. \end{cases}
    \end{equation}
    The solution is given by 
    \begin{equation} \label{eq:fluid_limit_case_1}
        \bar{W}(t) = \frac{\mu}{\theta} + \left( w_0 - \frac{\mu}{\theta} \right)e^{-\theta t}, \quad  t \ge 0.
    \end{equation}
    Taking the limit as \( t \to \infty \), we find that the fluid-scaled waiting time has a stable fixed point at \( \mu / \theta > 0 \).

    \item[(b)] Consider when \( \theta \le  0 \). Relation \eqref{eq:fluid_model_differential_form} implies that \( \mathrm{d} \bar{W}(t) > 0 \). Also since $w_0 \ge 0$, the regulator $\bar{L} \equiv 0$, and we can obtain the fluid equation by solving \eqref{eq:fluid_limit_diffeq}. When $\theta < 0$, the solution is given by \eqref{eq:fluid_limit_case_1}. When $\theta =0$, the fluid equation is $\bar{W}(t) = w_0 + \mu t$. By examining the solutions, we see that there are no stable fixed points. 
\end{enumerate}

\begin{figure}[H]
    \centering 
    \begin{subfigure}{0.3\textwidth}
        \centering
        \begin{tikzpicture}[x=0.75pt,y=0.75pt,yscale=-1,xscale=1]
\useasboundingbox (-5,-100) rectangle (105,20);
\draw [color={rgb, 255:red, 155; green, 155; blue, 155 }  ,draw opacity=1 ] (-10.95, 0) -- (98.55, 0)(0, -93.6) -- (0, 10.4) (91.55, -5) -- (98.55, 0) -- (91.55, 5) (-5, -86.6) -- (0, -93.6) -- (5, -86.6)  ;
\draw    (0.05, -45.85) -- (98.3, -45.85) ;
\draw    (-0.45, -80.35) .. controls (14.3, -51.85) and (20.3, -51.35) .. (97.05, -49.35) ;
\draw    (0.05, -13.1) .. controls (14.3, -41.1) and (19.8, -39.6) .. (97.05, -42.85) ;

\draw (-58.95, -51.2) node [anchor=north west][inner sep=0.75pt]  [font=\scriptsize]  {$w_{0} =\mu /\theta $};
\draw (-18.95, -22.2) node [anchor=north west][inner sep=0.75pt]  [font=\scriptsize]  {$w_{0}$};
\draw (-18.95, -83.2) node [anchor=north west][inner sep=0.75pt]  [font=\scriptsize]  {$w_{0}$};
\draw (-9.15, 2.2) node [anchor=north west][inner sep=0.75pt]  [font=\scriptsize]  {$0$};

\end{tikzpicture}
        \caption{ \( \theta > 0 \)}
    \end{subfigure}
    \begin{subfigure}{0.3\textwidth}
        \centering
        \begin{tikzpicture}[x=0.75pt,y=0.75pt,yscale=-1,xscale=1]
\useasboundingbox (-5,-100) rectangle (105,20);
\draw [color={rgb, 255:red, 155; green, 155; blue, 155 }  ,draw opacity=1 ] (-10.95, 0) -- (98.55, 0)(0, -93.6) -- (0, 10.4) (91.55, -5) -- (98.55, 0) -- (91.55, 5) (-5, -86.6) -- (0, -93.6) -- (5, -86.6)  ;
\draw    (0.02, -30.41) -- (88.59, -58.98) ;

\draw (-10.12, 1.42) node [anchor=north west][inner sep=0.75pt]  [font=\scriptsize]  {$0$};
\draw (-18.69, -37.58) node [anchor=north west][inner sep=0.75pt]  [font=\scriptsize]  {$w_{0}$};

\end{tikzpicture}
        \caption{ \( \theta = 0 \)}
    \end{subfigure}
    \begin{subfigure}{0.3\textwidth}
        \centering
        \begin{tikzpicture}[x=0.75pt,y=0.75pt,yscale=-1,xscale=1]
\useasboundingbox (-5,-100) rectangle (105,20);
\draw [color={rgb, 255:red, 155; green, 155; blue, 155 }  ,draw opacity=1 ] (-10.95, 0) -- (98.55, 0)(0, -93.6) -- (0, 10.4) (91.55, -5) -- (98.55, 0) -- (91.55, 5) (-5, -86.6) -- (0, -93.6) -- (5, -86.6)  ;
\draw    (0.31, -17.56) .. controls (24.59, -21.56) and (75.35, -33.15) .. (80.78, -79.44) ;

\draw (-10.12, 1.42) node [anchor=north west][inner sep=0.75pt]  [font=\scriptsize]  {$0$};
\draw (-20.12, -24.44) node [anchor=north west][inner sep=0.75pt]  [font=\scriptsize]  {$w_{0}$};

\end{tikzpicture}
        \caption{ \( \theta < 0 \)}
    \end{subfigure}
    \caption{Fluid limits for overloaded systems ($ \mu > 0 $). }
\end{figure}
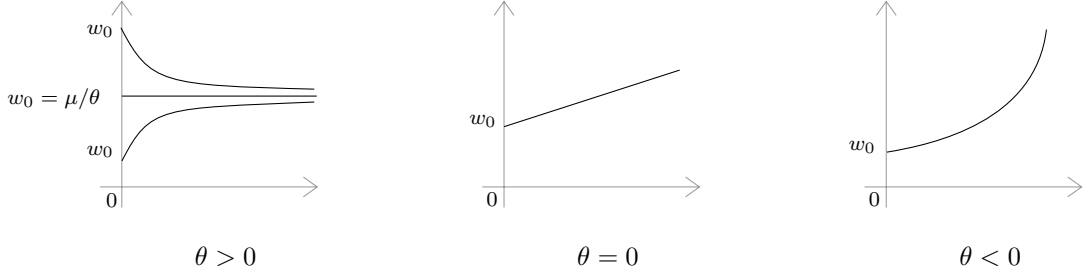

\subsubsection*{Critically loaded system ($\mu = 0$)} \ 

By similar arguments to the overloaded cases, the regulator \( \bar{L} \) is also never activated. Then the fluid limit is again obtained by solving \eqref{eq:fluid_limit_diffeq} and has the form  
\begin{equation}
    \bar{W}(t) =  e^{- \theta t} w_{0} , \quad  t \ge 0. 
\end{equation}
There are two cases that could arise: 
\begin{enumerate}
    \item[(a)] When \( \theta \ge  0 \), the fixed point \( 0 \) is stable.  
    \item[(b)] When \( \theta < 0 \), The fluid limit \( \bar{W} \) goes to infinity asymptotically unless we start at \( 0 \). So there are no stable fixed points. 
\end{enumerate}

\begin{figure}[H]
    \centering 
    \begin{subfigure}{0.3\textwidth}
        \centering
        \begin{tikzpicture}[x=0.75pt,y=0.75pt,yscale=-1,xscale=1]
\useasboundingbox (-5,-100) rectangle (105,20);
\draw [color={rgb, 255:red, 155; green, 155; blue, 155 }  ,draw opacity=1 ] (-10.95, 0) -- (98.55, 0)(0, -93.6) -- (0, 10.4) (91.55, -5) -- (98.55, 0) -- (91.55, 5) (-5, -86.6) -- (0, -93.6) -- (5, -86.6)  ;
\draw    (0.35, -35.95) .. controls (15.1, -7.45) and (21.1, -6.95) .. (97.85, -4.95) ;

\draw (-19.35, -43.6) node [anchor=north west][inner sep=0.75pt]  [font=\scriptsize]  {$w_{0}$};
\draw (-10.55, 0.6) node [anchor=north west][inner sep=0.75pt]  [font=\scriptsize]  {$0$};

\end{tikzpicture}
        \caption{ \( \theta > 0 \)}
    \end{subfigure}
    \begin{subfigure}{0.3\textwidth}
        \centering
        \begin{tikzpicture}[x=0.75pt,y=0.75pt,yscale=-1,xscale=1]
\useasboundingbox (-5,-100) rectangle (105,20);
\draw [color={rgb, 255:red, 155; green, 155; blue, 155 }  ,draw opacity=1 ] (-10.95, 0) -- (98.55, 0)(0, -93.6) -- (0, 10.4) (91.55, -5) -- (98.55, 0) -- (91.55, 5) (-5, -86.6) -- (0, -93.6) -- (5, -86.6)  ;
\draw    (0.02, -30.41) -- (90.31, -30.41) ;
\draw    (0, 0) -- (90.29, 0) ;

\draw (-37.55, 1.42) node [anchor=north west][inner sep=0.75pt]  [font=\scriptsize]  {$w_{0} =0$};
\draw (-18.69, -37.58) node [anchor=north west][inner sep=0.75pt]  [font=\scriptsize]  {$w_{0}$};

\end{tikzpicture}
        \caption{ \( \theta = 0 \)}
    \end{subfigure}
    \begin{subfigure}{0.3\textwidth}
        \centering
        \begin{tikzpicture}[x=0.75pt,y=0.75pt,yscale=-1,xscale=1]
\useasboundingbox (-5,-100) rectangle (105,20);
\draw [color={rgb, 255:red, 155; green, 155; blue, 155 }  ,draw opacity=1 ] (-10.95, 0) -- (98.55, 0)(0, -93.6) -- (0, 10.4) (91.55, -5) -- (98.55, 0) -- (91.55, 5) (-5, -86.6) -- (0, -93.6) -- (5, -86.6)  ;
\draw    (0.31, -17.56) .. controls (24.59, -21.56) and (75.35, -33.15) .. (80.78, -79.44) ;

\draw (-10.12, 1.42) node [anchor=north west][inner sep=0.75pt]  [font=\scriptsize]  {$0$};
\draw (-20.12, -24.44) node [anchor=north west][inner sep=0.75pt]  [font=\scriptsize]  {$w_{0}$};

\end{tikzpicture}
        \caption{ \( \theta < 0 \)}
    \end{subfigure}
    \caption{Fluid limits for critically loaded systems ($ \mu = 0 $). }
\end{figure}
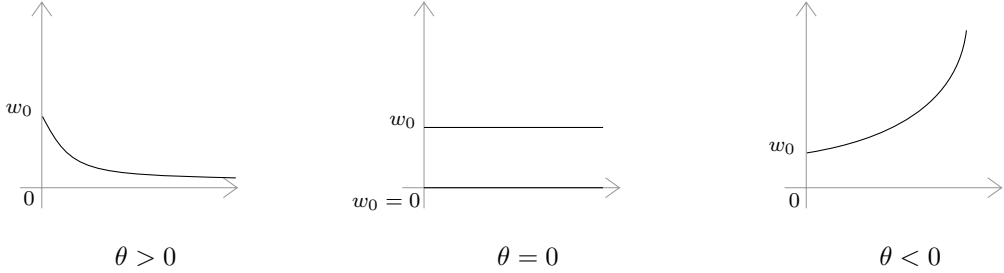

\subsubsection*{Underloaded system ($\mu < 0$).}

\begin{enumerate}
    \item[(a)] Suppose \( \theta < 0 \). In this case, $\mu/\theta$ and $0$ are the two fixed points and their stability depends on the initial condition \( \bar{W}_0 \). 
    
    \begin{enumerate}
        \item[(i)] If \( 0 \le  w_0 < \mu / \theta \), then by \eqref{eq:fluid_model_differential_form} we have \( \mathrm{d} \bar{W}(0) = \mu - \theta w_0 < 0 \). Since the regulator \( \bar{L} \) is not activated until  \( \bar{W} \) hits \( 0 \), the fluid limit again evolves according to \eqref{eq:fluid_limit_case_1}. By direct computation, we see that the fluid limit hits zero at time 
        \begin{equation} \label{eq:fluid_limit_hit_0}
            t_0 = - \frac{1}{\theta} \left[ \ln(\frac{\mu}{\theta}) - \ln (\frac{\mu}{\theta} - w_{0}) \right].
        \end{equation}
        When the process hits 0, the compensator \( \bar{L} \) activates with \( \mathrm{d} \bar{L}(t) = - \mu \), resulting in a stable fixed point of 0. Together, the trajectory of the fluid limit is given by 
        \begin{equation} \label{eq:fluid_limit_case_2a}
            \bar{W}(t) = 
            \begin{cases} \frac{\mu}{\theta} + \left( w_0 - \frac{\mu}{\theta} \right)e^{-\theta t}, & 0 \le t \le t_0,  \\ 0 , & t > t_0. \end{cases}
        \end{equation}
        \item[(ii)] If \( w_0 > \mu / \theta \), by \eqref{eq:fluid_model_differential_form}, we have \( \mathrm{d} \bar{W}(0) =  \mu - \theta w_0  > 0 \). The system evolves according to \eqref{eq:fluid_limit_case_1} and there are no fixed points. 
        \item[(iii)] If \( w_0 =  \mu / \theta > 0 \), then \( \mathrm{d} \bar{L}(0) = 0 \). From \eqref{eq:fluid_model_differential_form} we have \( \mathrm{d} \bar{W}(0) = 0  \), and hence   \( \bar{W}(t) =  \mu / \theta \) for all \( t \ge 0\). But based on our analysis earlier, the fixed point \( \mu/\theta \) is not stable. 
    \end{enumerate}
    
    \item[(b)] Suppose \( \theta \ge  0 \). Again by \eqref{eq:fluid_model_differential_form}, we have \( \mathrm{d} \bar{W}(t) < 0 \) for all \( t \ge 0 \). Similar to  the underloaded system under case (a)(i), the activation of the Skorokhod regulator \( \bar{L} \) leads to a stable fixed point at 0. 
        \begin{enumerate}
            \item[(i)] When \( \theta > 0 \), the system evolves according to the solution given by \eqref{eq:fluid_limit_case_2a}. 
            \item[(ii)] When \( \theta = 0 \), the trajectory is given by 
            \begin{equation}
                \bar{W}(t) = 
                \begin{cases} w_{0} + \mu t, & 0 \le t \le - w_{0} / \mu, \\ 0, & t \ge - w_{0} / \mu.  \end{cases}
            \end{equation}
        \end{enumerate}
\end{enumerate}

\begin{figure}[H]
    \centering 
    \begin{subfigure}{0.3\textwidth}
        \centering
        \begin{tikzpicture}[x=0.75pt,y=0.75pt,yscale=-1,xscale=1]
\useasboundingbox (-5,-100) rectangle (105,20);
\draw [color={rgb, 255:red, 155; green, 155; blue, 155 }  ,draw opacity=1 ] (-10.8, -32.6) -- (98.7, -32.6)(0, -93.6) -- (0, 10.4) (91.7, -37.6) -- (98.7, -32.6) -- (91.7, -27.6) (-5, -86.6) -- (0, -93.6) -- (5, -86.6)  ;
\draw [color={rgb, 255:red, 155; green, 155; blue, 155 }  ,draw opacity=1 ]   (1, -3.85) -- (99.25, -3.85) ;
\draw  [dash pattern={on 0.84pt off 2.51pt}]  (-0.07, -54.02) .. controls (4.47, -31.47) and (35.75, -9.47) .. (99.02, -6.93) ;
\draw    (-0.07, -54.02) .. controls (1.42, -46.6) and (5.81, -39.25) .. (13.27, -32.62) .. controls (39.02, -32.56) and (55.75, -32.56) .. (98.29, -32.75) ;

\draw (-26, -10.4) node [anchor=north west][inner sep=0.75pt]  [font=\scriptsize]  {$\mu /\theta $};
\draw (-18.44, -62.47) node [anchor=north west][inner sep=0.75pt]  [font=\scriptsize]  {$w_{0}$};
\draw (-10.2, -32.2) node [anchor=north west][inner sep=0.75pt]  [font=\scriptsize]  {$0$};

\end{tikzpicture}
        \caption{ \( \theta > 0 \)}
    \end{subfigure}
    \begin{subfigure}{0.3\textwidth}
        \centering
        \begin{tikzpicture}[x=0.75pt,y=0.75pt,yscale=-1,xscale=1]
\useasboundingbox (-5,-100) rectangle (105,20);
\draw [color={rgb, 255:red, 155; green, 155; blue, 155 }  ,draw opacity=1 ] (-10.8, -32.6) -- (98.7, -32.6)(0, -93.6) -- (0, 10.4) (91.7, -37.6) -- (98.7, -32.6) -- (91.7, -27.6) (-5, -86.6) -- (0, -93.6) -- (5, -86.6)  ;
\draw  [dash pattern={on 0.84pt off 2.51pt}]  (-0.07, -54.02) -- (87.95, 7.43) ;
\draw    (-0.07, -54.02) -- (30.51, -32.67) -- (98.45, -32.57) ;

\draw (-18.44, -62.47) node [anchor=north west][inner sep=0.75pt]  [font=\scriptsize]  {$w_{0}$};
\draw (-10.2, -32.2) node [anchor=north west][inner sep=0.75pt]  [font=\scriptsize]  {$0$};

\end{tikzpicture}
        \caption{ \( \theta = 0 \)}
    \end{subfigure}
    \begin{subfigure}{0.3\textwidth}
        \centering
        \begin{tikzpicture}[x=0.75pt,y=0.75pt,yscale=-1,xscale=1]
\useasboundingbox (-5,-100) rectangle (105,20);
\draw [color={rgb, 255:red, 155; green, 155; blue, 155 }  ,draw opacity=1 ] (-10.29, 0) -- (99.21, 0)(0, -90.83) -- (0, 13.17) (92.21, -5) -- (99.21, 0) -- (92.21, 5) (-5, -83.83) -- (0, -90.83) -- (5, -83.83)  ;
\draw    (0.65, -34.79) -- (98.9, -34.79) ;
\draw  [dash pattern={on 0.84pt off 2.51pt}]  (0.44, -22.96) .. controls (32.85, -21.14) and (48.85, -7.43) .. (76.85, 29.14) ;
\draw    (0.53, -48.06) .. controls (35.23, -51.38) and (57.23, -61.95) .. (76.66, -87.67) ;
\draw    (0.44, -22.96) .. controls (21.94, -21.76) and (36.76, -14.8) .. (52.11, 0.07) .. controls (67.56, 0.14) and (77.11, -0.08) .. (97.11, 0.07) ;

\draw (-56.3, -41.11) node [anchor=north west][inner sep=0.75pt]  [font=\scriptsize]  {$w_{0} =\mu /\theta $};
\draw (-16.12, -56.99) node [anchor=north west][inner sep=0.75pt]  [font=\scriptsize]  {$w_{0}$};
\draw (-11.12, 1.71) node [anchor=north west][inner sep=0.75pt]  [font=\scriptsize]  {$0$};
\draw (-16.45, -28.99) node [anchor=north west][inner sep=0.75pt]  [font=\scriptsize]  {$w_{0}$};

\end{tikzpicture}
        \caption{ \( \theta < 0 \)}
    \end{subfigure}
    \caption{Fluid limits for  underloaded systems ($ \mu < 0 $). }
\end{figure}
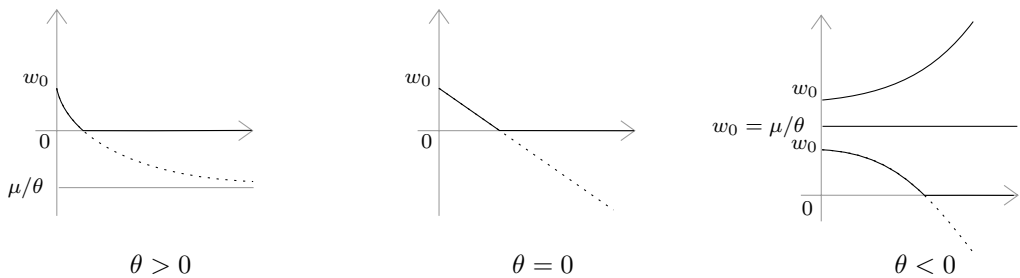

For ease of reference, we summarize the above discussions on the stable fixed points in Table \ref{table:fluid_stability}.  

\begin{table}[ht]
    \begin{tabular}{|c|c|c|c|}  
        \hline 
        & \( \mu > 0 \) & \( \mu = 0 \) & \( \mu < 0 \) \\
        \hline
        \( \theta > 0 \) & \( \mu/\theta \) & 0 & 0 \\
        \hline
        \( \theta = 0 \) & unstable & 0 & 0 \\
        \hline
        \( \theta < 0 \) & unstable & unstable & \( 0 \) \\
        \hline
    \end{tabular}
    \caption{Stable fixed points of the fluid limit \( \bar{W} \) under various parameter regimes.}
    \label{table:fluid_stability}
\end{table}

\FloatBarrier

\section{Sample-path MDP for a Linearly Recursive Markov System} \label{sec:MDP_LRMS}

In this section, we establish a sample-path MDP for the recursive system  
\begin{equation} \label{eq:linear_recursion}
    V^{n}_{i+1} = C^{n}_i V^{n}_i + X^{n}_i , \quad i \in \mathbb{N}_0, 
\end{equation}
where $V^n_0$ is a random variable. Under Assumption \ref{assump:model}, this recursion is a special case of the recursive Markov systems considered by \cite{DupJoh2015}; see also Chapter 5 of \cite{DupBud2019}. Those works establish MDPs for a broader class of recursions in the space \( \mathcal{C}_T \) using a weak-convergence approach based on variational representations. However, this method does not appear to adapt easily to the reflected recursion \eqref{eq:lindley_recursion}. 
We therefore develop a direct proof of the MDP in \( \mathcal{D}_T \), tailored to the linear recursion \eqref{eq:linear_recursion}. This proof produces exponential-equivalence and exponential-tightness arguments that are later used in Section \ref{sec:MDP_W} to establish MDPs for the reflected linear recursion \eqref{eq:recursion_W}. We will also use the MDP result for \eqref{eq:linear_recursion} to provide a workload bound in Section \ref{subsec:workload_bound_aux_system}.  Moreover, because of the specific form of the recursion in \eqref{eq:linear_recursion}, the corresponding rate functions can be identified explicitly.

\subsection{Fluid analysis} \label{subsec:fluid_analysis_V}

Similar to \eqref{eq:process_W}, we first re-index \( i \) by \( \left\lfloor nt \right\rfloor \), \( t\in [0,T] \) and analyze the fluid limit of the process 
\begin{equation} \label{eq:V_bar_rep}
    \bar{V}^{n}(t) = \frac{1}{n} V^{n}_{\left\lfloor nt \right\rfloor}, \quad t \in [0,T]. 
\end{equation}
\begin{theorem}\label{thm:fluid_limit_V}
    Let Assumption \ref{assump:model} hold and $\bar{V}^n_0 \to v_0$ in $L^2$  for some $v_0 \in \mathbb{R}$ as $n \to \infty$,  then 
    \[
        \|\bar{V}^{n} - \bar{V} \|_{T} \to 0 \quad \text{in probability},
    \]
    where 
    \begin{equation} \label{eq:fluid_limit_V}
        \bar{V}(t) =  \frac{\mu}{\theta} + \left( v_0 - \frac{\mu}{\theta} \right)e^{-\theta t}, \quad  t \in [0,T].
    \end{equation}
\end{theorem}

\begin{proof}

    Similar to \eqref{eq:W_fluid_representation}, we can write 
    \begin{equation*} 
        \bar{V}^{n}(t)  
        = \bar{V}^{n}_0 + \frac{1}{n} \sum_{i=0}^{\lfloor nt \rfloor -1} X_i^{n} - \int_{0}^{t} \theta \bar{V}^{n}(s) ds  + \bar{\epsilon}_{V}^{1,n}(t) + \bar{\epsilon}_{V}^{2,n}(t), 
    \end{equation*}
    where 
    \[
    \begin{aligned}
        \bar{\epsilon}_{V}^{1,n}(t) &= \theta \left(  \int_{0}^{t} \bar{V}^{n}(s) ds - \frac{1}{n} \sum_{i=0}^{\lfloor nt \rfloor -1} \bar{V}^{n}_i \right), \\
        \bar{\epsilon}_{V}^{2,n}(t) &= \frac{1}{n} \sum_{i=0}^{\lfloor nt \rfloor -1} (\theta - \Theta_i)  \bar{V}^{n}_i.
    \end{aligned} 
    \]
    By the same arguments used to show Lemma \ref{lem:fluid_error}, we can show  \( \| \bar{\epsilon}_{V}^{1,n}\|_{T} \) and \( \|\bar{\epsilon}^{2,n}_{V}\|_{T} \) converge to \( 0 \) in probability. Then we apply the continuous mapping theorem, as in the proof of Theorem \ref{thm:fluid_limit}, to obtain that the fluid limit \( \bar{V} = \mathcal{M}_\theta(v_0 + \mu \mathfrak{e}) \). By writing this in differential form, the explicit formula in \eqref{eq:fluid_limit_V} is derived by solving the differential equation 
    \[
        \begin{cases} d\bar{V}(t) = \mu - \theta \bar{V}(t),  \\ \bar{V}(0) = v_0.   \end{cases}
    \]
    This concludes the proof.
\end{proof}

Similar to Section \ref{sec:fluid_stability}, we need to analyze the stability of fixed points in the fluid equation \eqref{eq:fluid_limit_V}. It is simpler here, as we do not have the reflection term that appeared in \eqref{eq:fluid_model_differential_form}. Therefore, we simply summarize the results in Table \ref{table:fluid_stability_V}. 
\begin{table}[H]
    \begin{tabular}{|c|c|c|c|}  
        \hline 
        & \( \mu > 0 \) & \( \mu = 0 \) & \( \mu < 0 \) \\
        \hline
        \( \theta > 0 \) & \( \mu/\theta \) & 0 & \( \mu/\theta \) \\
        \hline
        \( \theta = 0 \) & unstable & 0 & unstable \\
        \hline
        \( \theta < 0 \) & unstable & unstable & unstable \\
        \hline
    \end{tabular}
    \caption{Stable fixed points of the fluid limit \( \bar{V} \) under various parameter regimes.}
    \label{table:fluid_stability_V}
\end{table}

\subsection{MDP results} \label{subsec:V_MDP_results}

Now, we are ready to study sample-path moderate deviations for this model. Let the scaling sequence \( \{b_n, n \in \mathbb{N}\} \) be defined by \eqref{eq:MDP_scaling_seq}. We also restrict ourselves to the cases where the fluid limit is stable and study moderate deviations for processes of the form 
\begin{equation} \label{eq:md_V}
    \widetilde{V}^{n}(t) = \frac{\sqrt{n}}{b_n} \left( \bar{V}^{n}(t) - \bar{V}^* \right) , \quad t \in [0,T],  
\end{equation}
where \( \bar{V}^* \) are the stable fixed points identified in Table \ref{table:fluid_stability_V}. First, we need an assumption similar to Assumption \ref{assump:MDP_W} (i). However, we do not require non-negativity of \(V^n_0 \). 

\begin{assumption} 
    \label{assump:V} 
    Let the random variables \( \widetilde{V}^n_0 \stackrel{P^{1/b_n^2}}{\longrightarrow} v_0 \) for some \(v_0 \in \mathbb{R}\).  
\end{assumption} 

Here is the main MDP result. Recall that \( I_X \) and \( I_\Theta \) are the rate functions in Theorem \ref{thm:MDP_RW}. 
\begin{theorem}
    \label{thm:MDP_V}
    Under Assumptions \ref{assump:model}, \ref{assump:MDP_W} (ii), (iii) and \ref{assump:V}, the family  \( \{ \widetilde{V}^{n},\; n \in \mathbb{N} \} \) satisfies an MDP with rate \( b_n^{2} \) and rate function \(I\), where 
    \begin{enumerate}
        \item[(i)] if \( \theta > 0 \), \( \mu \neq 0 \) and \( \bar{V}^{*} = \mu / \theta \), then
        \[
            I(\phi) = \inf_{\substack{\psi_1, \psi_2 \in \mathcal{D}_T,\\ \phi  = \mathcal{M}_\theta ( v_0 + \psi_1 - \frac{\mu}{\theta} \psi_2 + r \mathfrak{e}      ) .}}   I_X(\psi_1)+I_\Theta(\psi_2) ;
        \]
        \item[(ii)] if \( \mu = 0 \), \( \theta \ge 0 \) and \( \bar{V}^* = 0 \), then  
        \[
            I(\phi) = \inf_{\substack{\psi_1 \in \mathcal{D}_T,  \\  \phi = \mathcal{M}_\theta ( v_0 + \psi_1 +   r \mathfrak{e}  ).  }}  I_X(\psi_1) .
        \]
    \end{enumerate} 
\end{theorem}

Due to the simplicity of \( I_X \) and \( I_\Theta \), we can explicitly solve the optimization problems for the rate functions in Theorem \ref{thm:MDP_V}. 

\begin{theorem}
    \label{thm:MDP_V_explicit_rate_fn}
    Under Assumptions \ref{assump:model}, \ref{assump:MDP_W} (ii), (iii) and \ref{assump:V}, the rate functions in Theorem \ref{thm:MDP_V} take the form: 
    \begin{enumerate}
        \item[(i)] suppose \( \theta > 0 \),  \( \mu \neq 0 \) and \( \bar{V}^{*} = \mu / \theta \),  then 
        \[
            I(\phi) = \frac{\theta^{2}}{2(\theta^{2}\sigma^{2}_{X}+ \mu^{2} \sigma^{2}_{\Theta})}  \int_{0}^{T}   (\dot{\phi}(t) - r + \theta \phi(t))^{2} \mathrm{d}t, 
        \]
        for $\phi \in \mathcal{AC}$ with \(\phi(0) = v_0\). Otherwise, $I(\phi) = \infty$. 
        \item[(ii)] suppose \( \mu = 0 \), \( \theta \ge 0 \) and \( \bar{V}^{*} = 0 \),  then 
        \[
            I(\phi) = \frac{1}{2\sigma^{2}_{X}} \int_{0}^{T}  (\dot{\phi}(t) - r + \theta \phi(t))^{2} \mathrm{d}t, 
        \]
        for $\phi \in \mathcal{AC}$ with \(\phi(0) = v_0\). Otherwise, $I(\phi) = \infty$. 
    \end{enumerate}
    
\end{theorem}

\subsection{Exponential Tightness}

In this section, we prove Theorem \ref{thm:exp_tightness_V}, which establishes exponential tightness for the family \( \{\widetilde{V}^n\}_{n} \). 
With some algebra, we can write \eqref{eq:md_V} as
\begin{equation}\label{eq:V_md_representation}
    \begin{aligned} 
    \widetilde{V}^{n}(t) 
    &=  \widetilde{V}^{n}(0) +  \widetilde{R}^n_X(t) -   \bar{V}^{*} \widetilde{R}^n_\Theta(t) - \int_{0}^{t} \theta \widetilde{V}^{n} (s) \mathrm{d}s   \\
        & \quad \quad + \frac{\sqrt{n}}{b_n}(\mu_n - \mu)t  + \frac{\sqrt{n}}{b_n}  \left( \mu - \theta \bar{V}^{*} \right) t + \widetilde{\epsilon}^{1,n}_{V}(t) + \widetilde{\epsilon}^{2,n}_{V}(t) + \widetilde{\epsilon}^{3,n}_{V}(t),
\end{aligned}
\end{equation}
where $\widetilde{R}^n_X$ and $\widetilde{R}^n_\Theta$ are the random walks defined in \eqref{eq:RW_MD_scaling} and the error terms are given by  
\[
\begin{aligned}
    \widetilde{\epsilon}^{1,n}_{V}(t) &= \theta \left(  \int_{0}^{t} \widetilde{V}^{n}(s) \mathrm{d} s  - \frac{1}{n} \sum_{i=0}^{\lfloor nt \rfloor -1} \widetilde{V}^{n}_i \right), \\
    \widetilde{\epsilon}^{2,n}_{V}(t) &= \frac{1}{b_n\sqrt{n}} \sum_{i=0}^{\left\lfloor nt \right\rfloor-1} \left( \theta - \Theta_{i} \right) \left( \bar{V}^{n}_{i} - \bar{V}^{*} \right), \\
    \widetilde{\epsilon}^{3,n}_{V}(t) &= \frac{\left\lfloor nt \right\rfloor - nt}{b_n\sqrt{n}}    \left( \mu_n - \theta \bar{V}^{*} \right).
\end{aligned} 
\]
Then it suffices to show exponential tightness for each of the terms in \eqref{eq:V_md_representation}.  
The terms that require substantial analysis are \( \widetilde{\epsilon}^{1,n}_V \) and \( \widetilde{\epsilon}^{2,n}_V \). To do so, we shall need the following lemma.

\begin{lemma}
\label{lem:exp_bounded_Vbar}  
  Under Assumptions \ref{assump:model}, \ref{assump:MDP_W} (ii), (iii) and \ref{assump:V}, we have  
  \begin{equation} \label{eq:exp_bounded_Vbar}
     \lim_{K \to \infty} \limsup_{n \to \infty} \frac{1}{b_n^{2}} \log \mathbb{P} \left( \| \bar{V}^{n} \|_T > K \right) = - \infty. 
  \end{equation}
\end{lemma}

\begin{proof}
    By expanding the recursion, we obtain 
    \begin{equation*}
        \bar{V}^{n}_{i} = \bar{X}^{n}_{i-1} + C^{n}_{i-1} \bar{X}^{n}_{i-2} + \cdots + C^{n}_{i-1} \cdots C^{n}_{1} \bar{X}^{n}_{0} + C^{n}_{i-1} \cdots C^{n}_{1} C^{n}_{0} \bar{V} ^{n}_{0} . 
    \end{equation*}
    Since \( \log(1+x) \le x \), we have the following bound: 
    \begin{equation*}
        C^{n}_{i} = 1- \frac{1}{n}\Theta_i  \le 1 + \frac{1}{n}|\Theta_i|  = e^{\log(1 + \frac{1}{n}|\Theta_i| )}   \le e^{\frac{1}{n} |\Theta_i|}. 
    \end{equation*}
    Further using the fact that \( \exp(\frac{1}{n} |\Theta_i|) \ge 1 \) a.s., we apply the above bounds to \( \bar{V} ^{n}_{\left\lfloor nt \right\rfloor} \) and get 
    \begin{equation}
        \vert \bar{V} ^{n}_{\left\lfloor nt \right\rfloor}  \vert 
        \le \exp \left\{\frac{1}{n} \sum_{i=0}^{\left\lfloor nt \right\rfloor-1} |\Theta_i| \right\} \vert \bar{V} ^{n }_{0} \vert  + \exp \left\{\frac{1}{n} \sum_{i=0}^{\left\lfloor nt \right\rfloor-1} |\Theta_i| \right\} \left( \frac{1}{n}  \sum_{i=0}^{\left\lfloor nt \right\rfloor-1} \vert X_{i}^{n} \vert  \right),
        \label{eq:exp_tightness_V_bar}
    \end{equation}
    Denote \(  \theta'  := \mathbb{E} \vert \Theta_0 \vert \) and \(\mu_n'  :=\mathbb{E} \vert X^{n}_0 \vert  \). Under Assumption \ref{assump:MDP_W} (ii), the families of random walks 
    \[
        \frac{1}{b_n \sqrt{n}}  \sum_{i=0}^{\left\lfloor nt \right\rfloor-1} \left( |\Theta_i| - \theta'nt \right)
        \quad \textrm{and} \quad 
        \frac{1}{b_n \sqrt{n}} \sum_{i=0}^{\left\lfloor nt \right\rfloor-1} \left( |X^{n}_i| - \mu_n'nt \right)
    \]
    obey an MDP in \(\mathcal{D}_T\) with rate \( b_n^{2} \). Then by Lemma 4.2 (b) in \cite{PuhW97}, we have 
    \begin{equation*}
        \frac{1}{n}\sum_{i=0}^{\left\lfloor n \cdot \right\rfloor-1} |\Theta_i| -  \theta' \mathfrak{e} \stackrel{p^{1/b_n^{2}}}{\longrightarrow} 0 
        \quad \textrm{and} \quad 
        \frac{1}{n}\sum_{i=0}^{\left\lfloor n \cdot \right\rfloor-1} |X_i^{n}| - \mu_n' \mathfrak{e} \stackrel{p^{1/b_n^{2}}}{\longrightarrow} 0. 
    \end{equation*}
    By applying the contraction principle, the right hand side in \eqref{eq:exp_tightness_V_bar} is exponentially equivalent to the deterministic process 
    \[
        e^{\theta' t } \vert \bar{V}^* \vert  +  e^{\theta't } \mu_n' t, \quad t\in [0,T].  
    \]
    Further, \( (\mu'_n)^2 \le \mathbb{E}[(X_n)^{2}] = \sigma^{2}_{X,n} + \mu_n^2 \)
    implies that $\mu'_n$ is a bounded sequence, and hence \eqref{eq:exp_bounded_Vbar} holds. 
\end{proof}

\begin{lemma}
    \label{lem:exp_tight_esp_V_2}
    Under Assumptions \ref{assump:model}, \ref{assump:MDP_W} (ii), (iii) and \ref{assump:V}, 
    the family of processes \( \{\widetilde{\epsilon}_{V} ^{2, n}, n \in \mathbb{N} \} \)   is exponentially tight in \(\mathcal{D}_T\) with rate \( b_n^{2} \). 
\end{lemma}

\begin{proof}
    We shall check the two conditions given by Theorem \ref{thm:C_exp_tightness_in_D}.  
    First we check that for each \( t \in [0,T] \), the family of variables  \( \{ \widetilde{\epsilon}^{2,n}_{V}(t) \}_{n \in \mathbb{N}} \) is exponentially tight with rate \( b_n^{2} \). Let \( t \in [0,T] \) and \( \alpha > 0 \). This requires us to find some constant \( K'_{\alpha} \) such that 
    \begin{equation}
        \label{eq:lem_V_eps2_exp_tight_cond_1}
        \limsup_{ n \to \infty} \frac{1}{b_n^{2}} \log \mathbb{P} \left( \vert \widetilde{\epsilon}^{2,n}_{V}(t) \vert > K'_{\alpha}  \right) < - \alpha. 
    \end{equation}
    By Lemma \ref{lem:exp_bounded_Vbar},  there exists \( K_{\alpha} > 0 \) such that 
    \begin{equation}
        \label{eq:lem_V_eps2_exp_tight_outside_gamma}
        \limsup_{n \to \infty} \frac{1}{b_n^{2}} \log \mathbb{P} \left( \| \bar{V}^{n} - \bar{V}^* \|_T > K_{\alpha} \right) < - \alpha. 
    \end{equation}
    Define the event  
    \begin{equation}
        \label{eq:lem_V_eps2_event_bounded_fluid}
        \Gamma^{n} = \{ \| \bar{V}^{n} - \bar{V}^* \| _T \le  K_{\alpha} \} . 
    \end{equation}
    Then using Remark \ref{rmk:analysis_facts}, we can bound the left hand side in \eqref{eq:lem_V_eps2_exp_tight_cond_1} in the following way: 
    \begin{align*}
        & \limsup_{ n \to \infty} \frac{1}{b_n^{2}} \log \mathbb{P} \left( \vert \widetilde{\epsilon}^{2,n}_{V}(t) \vert > K'_{\alpha}  \right) 
        \\
         \le  & 
        \limsup_{n \to \infty} \frac{1}{b_n^{2}} \log \mathbb{P} \left( \{ \widetilde{\epsilon}^{2,n}_{V}(t) > K'_{\alpha}  \} \cap \Gamma^{n} \right)   
        \\
        & \quad  \vee 
        \limsup_{n \to \infty} \frac{1}{b_n^{2}} \log \mathbb{P} \left( \{ - \widetilde{\epsilon}^{2,n}_{V}(t) > K'_{\alpha}  \} \cap \Gamma^{n} \right) \vee
        \limsup_{n \to \infty} \frac{1}{b_n^{2}} \log \mathbb{P} \left( \| \bar{V}^{n} - \bar{V}^* \|_T > K_{\alpha} \right)\,.
    \end{align*}
    Due to \eqref{eq:lem_V_eps2_exp_tight_outside_gamma}, it suffices to find some \( K'_{\alpha} \) such that  
    \begin{equation}
        \label{eq:lem_V_eps2_exp_tight_cond_1_restricted}
        \limsup_{n \to \infty} \frac{1}{b_n^{2}} \log \mathbb{P} \left( \{ \widetilde{\epsilon}^{2,n}_{V}(t) > K'_{\alpha}  \} \cap \Gamma^{n} \right)   < - \alpha.   
    \end{equation}
    A similar statement for the term involving \( - \widetilde{\epsilon}^{2,n}_{V} \) can be shown exactly in the same way. Let \( \Gamma_i^{n} = \{ \vert \bar{V}^{n}_{i} - \bar{V}^* \vert  \le K_{\alpha} \} \). We  define a \( \{ \mathcal{F}^{n}_{k} \}\)-martingale:
    \begin{equation}
        \label{eq:lem_V_eps2_exp_tight_martingale_Z}
        Z ^{n}_k =  \sum_{i=0}^{k} (\theta - \Theta_{i}) (\bar{V}^{n}_{i}- \bar{V}^{*}) 1_{\Gamma^{n}_i}, \quad k \in \mathbb{N}_{0}. 
    \end{equation}
    An application of the Markov's inequality yields that    
    \begin{align}
        \frac{1}{b_n^{2}} \log \mathbb{P} \left( \{ \widetilde{\epsilon}^{2,n}_{V}(t) > K'_{\alpha} \} \cap  \Gamma^{n} \right) 
        & = 
        \frac{1}{b_n^{2}} \log \mathbb{P} \left( \left\{  \frac{1}{b_n \sqrt{n}}  Z^{n}_{\left\lfloor nt \right\rfloor-1} > K'_{\alpha} \right\} \cap \Gamma^{n}   \right)  \nonumber
        \\
        & \le   
        \frac{1}{b_n^{2}} \log \mathbb{P} \left(  \frac{1}{b_n \sqrt{n}}  Z^{n}_{\left\lfloor nt \right\rfloor-1} > K'_{\alpha}   \right)  \nonumber
        \\
        & \le 
        -  K'_{\alpha} + \frac{1}{b_n^{2}} \log \mathbb{E} \left[   \exp \left\{\frac{b_n}{\sqrt{n}}  Z^{n}_{\left\lfloor nt \right\rfloor-1} \right\}  \right].
        \label{eq:lem_V_eps2_exp_tight_cond_1_markov_ineq}
    \end{align}
    Next, we make the observation that for each \( n \) large enough, 
    \begin{equation*}
        \zeta^{n}_k = \exp \left\{ \frac{b_n}{\sqrt{n}}   Z^{n}_{k} -  \frac{b_n^{2}}{n}   K_{\alpha}^{2} \sigma^{2}_\Theta k  \right\}, \quad k \in \mathbb{N}_{0},
    \end{equation*}
    is an \(\{ \mathcal{F}_k^n\} \)-supermartingale. To see this, simply observe that for \( n \) sufficiently large, we have 
    \begin{align}
        & \log \mathbb{E} \left[ \exp \left\{\frac{b_n}{\sqrt{n}} (\theta-\Theta_i)(\bar{V}^{n}_{i} - \bar{V}^*) 1_{\Gamma^{n}_{i}} \right\}  \Big| \mathcal{F}_{i-1}^n \right] \nonumber
        \\
        = \ &  \frac{1}{2} \frac{b_n^{2}}{n} (\bar{V}^{n}_{i} - \bar{V}^*)^{2} 1_{\Gamma^{n}_{i}} \mathbb{E} (\theta-\Theta_i)^{2} + \mathcal{O} \left( \frac{b_n^{3}}{n\sqrt{n}} (\bar{V}^{n}_{i} - \bar{V}^*)^{3} 1_{\Gamma^{n}_{i}} \mathbb{E} (\theta-\Theta_i)^{3}  \right) \nonumber
        \\
        \le \  & \frac{1}{2} \frac{b_n^{2}}{n} K_{\alpha}^{2} \sigma_{\Theta}^{2} + \mathcal{O} \left( \frac{b_n^{3}}{n\sqrt{n}} K_{\alpha} ^{3}  \mathbb{E} \vert \theta-\Theta_i\vert ^{3}  \right) \nonumber
        \\
        \le \  & \frac{b_n^{2}}{n} K_{\alpha}^{2} \sigma_{\Theta}^{2}. \label{eq:lem_V_eps2_exp_tight_supermtg_check} 
    \end{align}
    In the first equality above, because of Assumption \ref{assump:MDP_W} (ii) and the fact that \( b_n n^{-1/2} \to 0 \) as \( n \to  \infty \), we can perform a series expansion for the cumulant moment generating function. The last equality comes from taking \( n \) large enough such that  
    \begin{equation*}
        \mathcal{O} \left( \frac{b_n^{3}}{n\sqrt{n}} K_{\alpha} ^{3}  \mathbb{E} \vert \theta-\Theta_i \vert ^{3}  \right) 
        = \mathcal{O} \left( \frac{b_n^{3}}{n\sqrt{n}} K_{\alpha} ^{2} \sigma_{\Theta}^{2} \right)  
         \le \frac{1}{2} \frac{b_n^{2}}{n} K_{\alpha}^{2} \sigma_{\Theta}^{2}. 
    \end{equation*}
    Therefore \eqref{eq:lem_V_eps2_exp_tight_cond_1_markov_ineq} and the fact that \( \zeta^{n}_{k} \) is a $\{\mathcal{F}_k^n \}$-supermartingale give 
    \begin{equation*}
        \limsup_{n \to \infty} \frac{1}{b_n^{2}} \log \mathbb{P} \left( \{ \widetilde{\epsilon}^{2,n}_{V}(t) > K'_{\alpha} \} \cap  \Gamma^{n} \right) 
        \le 
        - K'_{\alpha} +  K^{2}_{\alpha} \sigma^{2}_{\Theta} t\,. 
    \end{equation*}
    Lastly, we simply choose \( K'_{\alpha} \ge \alpha + K^{2}_{\alpha} \sigma^{2}_{\Theta} t\) to obtain \eqref{eq:lem_V_eps2_exp_tight_cond_1_restricted}. This concludes our proof of \eqref{eq:lem_V_eps2_exp_tight_cond_1}.

    We next check the second condition in Theorem \ref{thm:C_exp_tightness_in_D}. We will show that for any $\epsilon>0$,
    \begin{equation*}
        \lim_{\delta \to 0} \limsup_{ n \to \infty} \sup_{t \in [0,T]} \frac{1}{ b_n^2} \log \mathbb{P} \left( \sup_{s \in [0,\delta]} \vert \widetilde{\epsilon}^{2,n}_{V}(t+s) - \widetilde{\epsilon}^{2,n}_{V}(t) \vert > \epsilon  \right)  = - \infty. 
    \end{equation*}
    Similar to \eqref{eq:lem_V_eps2_exp_tight_cond_1_restricted}, let \( \alpha > 0 \), and \( \Gamma^{n} \) be as defined in \eqref{eq:lem_V_eps2_event_bounded_fluid}. Then it suffices to show 
    \begin{equation}
        \label{eq:lem_V_eps2_exp_tight_cond_2_restricted}
        \lim_{\delta \to 0} \limsup_{ n \to \infty} \sup_{t \in [0,T]} \frac{1}{b_n^{2}} \log \mathbb{P} \left( \left\{ \sup_{s \in [0,\delta]}  \widetilde{\epsilon}^{2,n}_{V}(t+s) - \widetilde{\epsilon}^{2,n}_{V}(t)  > \epsilon  \right\} \cap \Gamma^{n} \right)  = - \infty. 
    \end{equation}
    Again, let \( \Gamma_i^{n} = \{ \vert \bar{V}^{n}_{i} - \bar{V}^* \vert  \le K_{\alpha} \} \), and recall the \( \{ \mathcal{F}^{n}_{k} \}_{k\in \mathbb{N}_{0}}\)-martingale \( \{Z^{n}_{k}, k \in \mathbb{N}_{0}\} \) defined in \eqref{eq:lem_V_eps2_exp_tight_martingale_Z}. The key observation is that for any \( t \in [0,T] \) and $n$, the process \( \{ Z^{n}_{\left\lfloor nt \right\rfloor + k} - Z^{n}_{\left\lfloor nt \right\rfloor  },\ k \in \mathbb{N}_0   \} \) is a \(\{ \mathcal{F}^{n}_{\lfloor nt \rfloor+k}\}_{k\in \mathbb{N}_{0}}\)-martingale.

    Let \( \rho > 0 \) and \( n \) sufficiently large. We have 
    \begin{align*}
        & \frac{1}{b_n^{2}} \log \mathbb{P} \left( \left\{ \sup_{s \in [0,\delta]}  \widetilde{\epsilon}^{2,n}_{V}(t+s) - \widetilde{\epsilon}^{2,n}_{V}(t)  > \epsilon  \right\} \cap \Gamma^{n} \right) 
        \\
        \le \ & 
        \frac{1}{b_n^{2}} \log \mathbb{P} \left( \max_{0 \le k \le \left\lfloor n(t+\delta) \right\rfloor - \left\lfloor nt \right\rfloor-1} Z^{n}_{\left\lfloor nt \right\rfloor + k} - Z^{n}_{\left\lfloor nt \right\rfloor  }    > \epsilon    \right) 
        \\
        = \ & 
        \frac{1}{b_n^{2}} \log \mathbb{P} \left( \max_{0 \le k \le \left\lfloor n(t+\delta) \right\rfloor - \left\lfloor nt \right\rfloor-1} \exp \left\{b_n^{2}\rho\left( Z^{n}_{\left\lfloor nt \right\rfloor + k} - Z^{n}_{\left\lfloor nt \right\rfloor  } \right) \right\}    > e^{b_n^{2} \rho \epsilon}  \right) 
        \\
        \le  \ & 
        - \rho \epsilon + \frac{1}{b_n^{2}} \log \mathbb{E} \left[ \exp \left\{b_n^{2}\rho\left( Z^{n}_{\left\lfloor n(t+\delta) \right\rfloor -1} - Z^{n}_{\left\lfloor nt \right\rfloor  } \right) \right\} \right] 
        \\
        \le \ & 
        - \rho \epsilon + K^{2}_\alpha \rho^{2} \sigma^{2}_{\Theta} \delta. 
    \end{align*}
    In the relations above, the second inequality is obtained by Doob's submartingale inequality. The last inequality uses the following supermartingale:  
    \[
        \exp \left\{b_n^{2}\rho\left( Z^{n}_{\left\lfloor nt \right\rfloor+k} - Z^{n}_{\left\lfloor nt \right\rfloor  } \right)  -  \frac{b_n^{2}}{n} \rho^{2}   K_{\alpha}^{2} \sigma^{2}_\Theta k  \right\}  , \quad k \in \mathbb{N}_0,
    \]
    and the fact that its expectation is less than or equal to one.  
    Finally, taking the limit gives 
    \begin{equation*}
        \lim_{\delta \to 0} \limsup_{ n \to \infty} \sup_{t \in [0,T]} \frac{1}{b_n^{2}} \log \mathbb{P} \left( \left\{ \sup_{s \in [0,\delta]}  \widetilde{\epsilon}^{2,n}_{V}(t+s) - \widetilde{\epsilon}^{2,n}_{V}(t)  > \epsilon  \right\} \cap \Gamma^{n} \right) \le  - \rho \epsilon. 
    \end{equation*}
    Since \( \rho > 0 \) was taken arbitrarily, we take \( \rho \to \infty \) to obtain \eqref{eq:lem_V_eps2_exp_tight_cond_2_restricted}. This concludes the proof.
\end{proof}

\begin{theorem} 
    \label{thm:exp_tightness_V}
    Under Assumptions \ref{assump:model}, \ref{assump:MDP_W} (ii), (iii) and \ref{assump:V}, the family \(\{ \widetilde{V} ^{ n},\; n  \in \mathbb{N} \}  \) is exponentially tight in the space \( \mathcal{D}_T \) with rate \( b_n^{2} \). 
\end{theorem}

\begin{proof}
    We write \eqref{eq:V_md_representation} as  
    \begin{equation*}
        \widetilde{V}^{n} = \mathcal{M}_\theta \left( \widetilde{V}^{n}_0 + \widetilde{R}^{n}_{X}  - \overline{V}^{*} \widetilde{R}^{n}_{\Theta}  + \frac{\sqrt{n}}{b_n}(\mu_n - \mu) \mathfrak{e} + \widetilde{\epsilon}^{1,n}_V  + \widetilde{\epsilon}^{2,n}_V   + \widetilde{\epsilon}^{3,n}_V  \right), 
    \end{equation*}
    and analyze each of the terms. First, by Assumption \ref{assump:V} and Theorem \ref{thm:MDP_RW}, the families of processes \( \{\widetilde{V}^{n}_0 \}_{n \in \mathbb{N}} \), \( \{\widetilde{R}^{n}_{X} \}_{n \in \mathbb{N}} \) and \( \{\widetilde{R}^{n}_{\Theta} \}_{n \in \mathbb{N}} \)   are exponentially tight in \(\mathcal{D}_T\) with rate \(b_n^2\). Also, by Assumption \ref{assump:MDP_W} (iii), the term \( b_n^{-1}\sqrt{n} (\mu_n - \mu) \mathfrak{e} \stackrel{P ^{1/b_n^{2}}}{\longrightarrow} r \mathfrak{e} \) and therefore is also exponentially tight in \(\mathcal{D}_T\) with rate \(b_n^2\). 
    Now we claim that 
    \begin{equation}
        \label{eq:lem_V_esp_1}
        \widetilde{\epsilon}_V^{1,n} \stackrel{P ^{1/b_n^{2}}}{ \longrightarrow } 0 \quad \textrm{and} \quad  \widetilde{\epsilon}_V^{3,n} \stackrel{P ^{1/b_n^{2}}}{ \longrightarrow } 0. 
    \end{equation}
    Observe that  
    \begin{equation*}
         \big\| \widetilde{\epsilon} ^{1,n}_V \big\|_{T} \le  \frac{|\theta|}{n} \big\| \widetilde{V}  ^{n} \big\|_T    =   \frac{\left\vert \theta \right\vert }{b_n \sqrt{n}} \left\| \bar{V}^{n} - \bar{V}^* \right\| _{T}. 
    \end{equation*}
    Let \( \alpha > 0 \). By Lemma \ref{lem:exp_bounded_Vbar}, there exists \( K_\alpha > 0  \) such that 
    \[
        \limsup_{ n \to \infty} \frac{1}{b_n^{2}} \log \mathbb{P} \left( \left\| \bar{V} ^{n} - \bar{V} ^* \right\|_T > K_{\alpha} \right) < - \alpha. 
    \]
    Let \( \epsilon > 0 \), for \( n \) large enough such that \( \epsilon > \frac{|\theta|}{b_n \sqrt{n}} K_\alpha  \), we have 
    \begin{equation*}
        \mathbb{P} \left( \big\| \widetilde{\epsilon}_{V} ^{1,n} \big\|_T > \epsilon \right) \le 
        \mathbb{P} \left( \big\| \widetilde{\epsilon}_{V} ^{1,n} \big\|_T > \frac{|\theta| }{b_n \sqrt{n}} K_{\alpha} \right) \le 
        \mathbb{P} \left( \left\| \bar{V} ^{n} - \bar{V}^*  \right\|_T >   K_{\alpha} \right). 
    \end{equation*}
    This implies 
    \[
        \limsup_{ n \to \infty} \frac{1}{b_n^{2}} \log \mathbb{P}\left( \big\| \widetilde{\epsilon}_{V} ^{1,n} \big\|_T > \epsilon \right) < -\alpha.
    \]
    Since \( \alpha \) is arbitrary, we obtain the first statement in  \eqref{eq:lem_V_esp_1} by taking \( \alpha \to \infty \) and using Lemma \ref{lem:supexp_char}.  For the second statement involving \( \widetilde{\epsilon}^{3,n}_V \), since \( \mu_{n} \to \mu \) and \( b_{n}  \sqrt{n} \to \infty \) as \( n \to  \infty \), then for any \( \epsilon > 0 \), 
    \begin{equation*}
        \limsup_{n \to \infty} \frac{1}{b_n^{2}} \log \mathbb{P} \left( \| \widetilde{\epsilon}^{3,n}_{V} \|_T > \epsilon  \right) \le \limsup_{n \to \infty} \frac{1}{b_n^{2}} \log \mathbb{P} \left( \frac{1}{b_{n} \sqrt{n} } \vert \mu_{n} - \theta \bar{V}^{*} \vert   > \epsilon  \right) = -\infty. 
    \end{equation*}
    This proves \eqref{eq:lem_V_esp_1} by Lemma \ref{lem:supexp_char}. 
    
    Finally, exponential tightness of \( \{ \widetilde{\epsilon}^{2,n}_{V} \}_{n \in \mathbb{N}} \) in \( \mathcal{D}_T \) is given in Lemma \ref{lem:exp_tight_esp_V_2}. Then since \( \mathcal{M}_{\theta} \) is continuous in \( \mathcal{D}_T \), we can use Lemma \ref{lem:exp_tightness_cts_map} to conclude \( \{ \widetilde{V}^{n} \}_{n \in \mathbb{N}} \) is exponentially tight in \( \mathcal{D}_T \).
\end{proof}

By Lemma \ref{lem:exptight_fluid0}, a consequence of exponential tightness of \( \{ \widetilde{V}^{n} \}_{n \in \mathbb{N}} \) in \( \mathcal{D}_T \) is the following corollary, which will be next used to further analyze the error terms in \eqref{eq:V_md_representation}. We use a slight abuse of notation by letting \(\bar{V}^*\) be the constant process in \(\mathcal{D}_T\) instead of a constant. 

\begin{corollary}
    \label{cor:V_bar}
   Under Assumptions \ref{assump:model}, \ref{assump:MDP_W} (ii), (iii) and \ref{assump:V}, we have 
    \[
        \bar{V}^{n}  \stackrel{p^{1/b_n^{2}}}{\longrightarrow}   \bar{V}^* . 
    \]
\end{corollary}

\subsection{Proofs for MDP Results in Section \ref{subsec:V_MDP_results}}  \label{subsec:MDP_V_proof} 

\begin{lemma}
    \label{lem:eps_V_expeq_0}
    Under Assumptions \ref{assump:model}, \ref{assump:MDP_W} (ii), (iii) and \ref{assump:V}, the families of processes \( \{\widetilde{\epsilon}^{1,n}_{V} \}_{n \in \mathbb{N}} \), \( \{\widetilde{\epsilon}^{2,n}_V \}_{n \in \mathbb{N}} \) and \( \{ \widetilde{\epsilon}^{3,n}_{V} \}_{n \in \mathbb{N}} \) are exponentially equivalent to the \( 0 \) process with rate \( b_n^{2} \). 
\end{lemma}

\begin{proof}
    We have already shown the assertion for \( \widetilde{\epsilon}^{1,n}_V \) and \( \widetilde{\epsilon}^{3,n}_V \)  in the proof for Theorem \ref{thm:exp_tightness_V}, see \eqref{eq:lem_V_esp_1}. To prove the statement for \( \widetilde{\epsilon}^{2,n}_{V} \), we use arguments similar to those in the proof of Lemma \ref{lem:exp_tight_esp_V_2}. 
    
    Let \( \epsilon > 0 \) and \( \eta > 0 \). Define the event 
    \begin{equation}
        \Gamma^{n} = \{ \| \bar{V}^{n} - \bar{V}^{*} \|_T \le  \eta \}. 
    \end{equation}
    Due to Remark \ref{rmk:analysis_facts}, Lemma \ref{lem:supexp_char} and Corollary \ref{cor:V_bar}, it suffices to show 
    \begin{equation}
        \label{eq:lem_V_eps2_exp_0_condn}
        \limsup_{n \to \infty} \frac{1}{b_n^{2}} \log \mathbb{P} \left( \left\{ \sup_{t \in [0,T]} \widetilde{\epsilon}^{2,n}_{V} (t) > \epsilon \right\} \cap \Gamma^{n} \right) = - \infty. 
    \end{equation} 
    For each \(i \in \mathbb{N}_0\), let \( \Gamma_i^{n} = \{ \vert \bar{V}^{n}_{i} - \bar{V}^* \vert  \le \eta \} \). We can define a \( \{ \mathcal{F}^{n}_{k} \}_{k\in \mathbb{N}_0}\)-martingale 
    \begin{equation}
        \label{eq:lem_V_eps2_exp_0_martingale_Z}
        Z ^{n}_k =  \sum_{i=0}^{k} (\theta - \Theta_{i}) (\bar{V}^{n}_{i}- \bar{V}^{*}) 1_{\Gamma^{n}_i}, \quad k \in \mathbb{N}_{0}. 
    \end{equation}
    Letting \( \rho > 0 \), we have the following:  
    \begin{align}
        & \limsup_{n \to \infty} \frac{1}{b_n^{2}} \log \mathbb{P} \left( \left\{ \sup_{t \in [0,T]} \widetilde{\epsilon}^{2,n}_{V}(t) > \epsilon \right\} \cap  \Gamma^{n} \right) \nonumber 
       \\ & \le   
        \limsup_{n \to \infty} \frac{1}{b_n^{2}} \log \mathbb{P} \left( \max_{0 \le k \le \left\lfloor nT \right\rfloor-1} \frac{1}{b_n \sqrt{n}}  Z^{n}_{\left\lfloor nt \right\rfloor-1} > \epsilon   \right)  \nonumber
        \\
        & \le 
        - \rho \epsilon + \limsup_{n \to \infty} \frac{1}{b_n^{2}}  \log \mathbb{E} \left[   \exp \left\{\frac{b_n}{\sqrt{n}} \rho Z^{n}_{\left\lfloor nT \right\rfloor-1} \right\}  \right] 
        \label{eq:lem_V_eps2_exp_0_doob} 
        \\
        & \le 
        - \rho \epsilon + \rho^{2} \eta^{2} \sigma^{2}_\Theta T. \label{eq:lem_V_eps2_exp_0_supermtg} 
    \end{align}
    In the above derivations, \eqref{eq:lem_V_eps2_exp_0_doob} is due to an application of Doob's submartingale inequality for the \( \{ \mathcal{F}^{n}_{k} \}_{k\in \mathbb{N}_0}\)-submartingale \( \{ \exp(b_n n^{-1/2} \rho Z^{n}_{k}), k \in \mathbb{N}_0 \}\). \eqref{eq:lem_V_eps2_exp_0_supermtg} uses the fact that the process 
    \begin{equation*}
        \zeta^{n}_k = \exp\left\{ \frac{b_n}{\sqrt{n}} \rho  Z^{n}_{k} -  \frac{b_n^{2}}{n} \rho^{2} \eta^{2} \sigma^{2}_\Theta k   \right\}, \quad k \in \mathbb{N}_0, 
    \end{equation*}
    is a supermartingale when \( n \) is large. One can check this by following similar steps as in \eqref{eq:lem_V_eps2_exp_tight_supermtg_check}. Finally, since \eqref{eq:lem_V_eps2_exp_0_supermtg} holds for arbitrary \( \eta \) and \( \rho \), we can first take \( \eta  \to  0 \) and then \( \rho \to  \infty \) to obtain \eqref{eq:lem_V_eps2_exp_0_condn}, as desired.
\end{proof}

Now we prove the results in Section \ref{subsec:V_MDP_results}. 

\begin{proof}[Proof of Theorem \ref{thm:MDP_V}.]
    This is simply a consequence of \eqref{eq:V_md_representation}, Theorem \ref{thm:MDP_RW}, Lemma \ref{lem:eps_V_expeq_0} and the contraction principle applied to the continuous map \( \mathcal{M}_\theta \).
\end{proof}

\begin{proof}[Proof of Theorem \ref{thm:MDP_V_explicit_rate_fn}.]
    For case (i), consider the optimization problem in Theorem \ref{thm:MDP_V} (i). It suffices to optimize over the set \( \{ ( \psi_1, \psi_2):  \psi_1 \in \mathcal{AC}_0 , \psi_2 \in \mathcal{AC}_0  \} \), otherwise the rate function is infinite. On this set, let \( \phi \in \mathcal{D}_T \) satisfy \( \phi(t) = v_0  + \psi_1(t) - \frac{\mu}{\theta} \psi_2(t) + r t - \int_{0}^{t} \theta \phi(s) \mathrm{d}s  \). Then it is clear that \( \phi  \in \mathcal{AC}  \) and \(\phi(0) = v_0\). 
    The problem reduces to solving the following convex optimization problem a.e. in time \( t \in [0,T] \): 
    \begin{equation*}
    \begin{aligned}
        \min_{\dot{\psi}_{1}, \dot{\psi}_{2} \in \mathbb{R}} \quad & \frac{1}{2 \sigma^{2}_{X}} \dot{\psi}_{1}(t)^{2} + \frac{1}{2 \sigma^{2}_{\Theta}} \dot{\psi}_{2}(t)^{2} \\
        \textrm{s.t.}\quad & \dot{\phi}(t) = \dot{\psi}_{1}(t) - \frac{\mu}{\theta} \dot{\psi}_{2}(t) +r - \theta \phi (t). 
    \end{aligned}
    \end{equation*}
    Let \( f(t) = \dot{\phi}(t) - r + \theta\phi(t) \). The solution is 
    \begin{equation*}
    \begin{aligned}
        \dot{\psi}_1(t) &= \frac{\theta^{2} \sigma^{2}_{X}}{\theta^{2}\sigma^{2}_{X} + \mu^{2} \sigma^{2}_{\Theta}} f (t) ,\\
        \dot{\psi}_2(t) &= - \frac{\mu \theta \sigma^{2}_{\Theta}}{\theta^{2}\sigma^{2}_{X} + \mu^{2} \sigma^{2}_{\Theta}} f (t). 
    \end{aligned} 
    \end{equation*}
    Plugging the solution into \( I_X (\psi_1) + I_\Theta(\psi_2) \) yields the form of the rate function. Case (ii) is solved similarly. This concludes the proof.
\end{proof}

\section{Proofs for MDP Results in Section \ref{subsec:MD_results}} \label{sec:MDP_W} 

First, recall the definition of \( \widetilde{W}^{n} \) in \eqref{eq:W_tilde}. Similar to \eqref{eq:V_md_representation}, we can obtain the following representation: 
\begin{equation}\label{eq:md_representation}
\begin{aligned}  
    \widetilde{W}^{n}(t) 
    &=  \widetilde{W}^{n}(0) +  \widetilde{R}^n_X(t) -   \bar{W}^{*} \widetilde{R}^n_\Theta(t) - \int_{0}^{t} \theta \widetilde{W}^{n} (s) \mathrm{d}s  + \frac{\sqrt{n}}{b_n}(\mu_n - \mu)t  \\
        & \quad \quad  + \frac{\sqrt{n}}{b_n}  \left( \mu - \theta \bar{W}^{*} \right) t + \widetilde{\epsilon}^{1,n}(t) + \widetilde{\epsilon}^{2,n}(t) + \widetilde{\epsilon}^{3,n}(t) +  \widetilde{L}^n(t). 
\end{aligned}
\end{equation}
Above, \( \widetilde{R}^{n}_X \) and \( \widetilde{R}^{n}_{\Theta} \) are the random walks given in \eqref{eq:RW_MD_scaling}. Recalling $\bar{L}^n$ in  \eqref{eq:fluid_W_reflected}, we define $\widetilde{L}^n  :=  b_n^{-1}\sqrt{n}  \bar{L}^{n}$  and the error terms as 
\[
\begin{aligned}
    \widetilde{\epsilon}^{1,n}(t) &= \theta \left(  \int_{0}^{t} \widetilde{W}^{n}(s) \mathrm{d} s  - \frac{1}{n} \sum_{i=0}^{\lfloor nt \rfloor -1} \widetilde{W}^{n}_i \right), \\
    \widetilde{\epsilon}^{2,n}(t) &= \frac{1}{b_n\sqrt{n}} \sum_{i=0}^{\left\lfloor nt \right\rfloor-1} \left( \theta - \Theta_{i} \right) \left( \bar{W}^{n}_{i} - \bar{W}^{*} \right), \\
    \widetilde{\epsilon}^{3,n}(t) &= \frac{\left\lfloor nt \right\rfloor - nt}{b_n\sqrt{n}}    \left( \mu_n - \theta \bar{W}^{*} \right).
\end{aligned} 
\] 

Our goal is to imitate the proof for the MDP results in Section \ref{sec:MDP_LRMS}. However, compared to \eqref{eq:V_md_representation}, equation \eqref{eq:md_representation} contains the extra term \( \widetilde{L}^{n} := (b_n \sqrt{n})^{-1} L^{n} \). In \eqref{eq:fluid_W_reflected}, we saw that \( (\bar{W}^{n}, \bar{L}^{n}) \) is the linearly generalized reflection map of a particular process. Under MD-scalings with the centering term \( \bar{W}^* = 0 \), the same is true for the pair \( (\widetilde{W}^{n}, \widetilde{L}^{n}) \). However, this is no longer the case when \( \bar{W}^* \neq 0 \). This creates some difficulties if we want to apply the contraction principle. To address this, we first provide a way to bound \( W^{n} \) by auxiliary systems. 

\subsection{Bounding the workload by auxiliary systems}
\label{subsec:workload_bound_aux_system} 

In this section, we provide a bound for the workload process \( W^{n} \) defined recursively by 
\begin{equation}
    \label{eq:Wn_lindley_def}
    W^{n}_{i+1} = \max \{ 0,\ C^{n}_{i} W^{n}_{i} + X^{n}_{i} \}, \quad  i \in \mathbb{N}_0.   
\end{equation}
Then we show that the bound can be related to the supremum of certain linearly recursive Markov systems that were studied in Section \ref{sec:MDP_LRMS}.  

First, define the process \( \{ \Upsilon^{n}(t),\ t \in [0,T] \} \) by \( \Upsilon^{n}(t) = \Upsilon^{n}_{\left\lfloor nt \right\rfloor}  \) where 
\begin{equation*}
    \begin{aligned}
    \Upsilon^{n}_0 &=  W^{n}_0, \\
    \Upsilon^{n}_{1} &= \max \big\{ 0,\ X^{n}_{0} + C^{n}_{0}W^{n}_{0} \big\}, \\
    \Upsilon^{n}_{2} &= \max \big\{ 0,\ X^{n}_{1},\ X^{n}_{1} + C^{n}_{1}X^{n}_{0} + C^{n}_{1}C^{n}_{0}W^{n}_{0} \big\},
\end{aligned}
\end{equation*}
and more generally for \( i \ge 2 \),  
\begin{equation*} 
\begin{aligned}
    \Upsilon^{n}_{i+1} &= \max \big\{ 0,\ X^{n}_{i},\ X^{n}_{i} + C^{n}_{i} X^{n}_{i-1},\ \cdots, \  X^{n}_{i} + C^{n}_{i} X^{n}_{i-1} + \cdots + C^{n}_{i} \cdots C^{n}_{2} X^{n}_{1}, \\
    & \quad \quad  X^{n}_{i} + C^{n}_{i} X^{n}_{i-1} + \cdots + C^{n}_{i} \cdots C^{n}_{2} X^{n}_{1} + C^{n}_{i} \cdots C^{n}_{1} X^{n}_{0} + C^{n}_{i} \cdots C^{n}_{0} W^{n}_{0} \big\}. 
\end{aligned}
\end{equation*}
\begin{lemma}
    For all \( i \ge 0 \), we have 
    \begin{equation*}
        0 \le W^{n}_{i} \le \Upsilon^{n}_{i}. 
    \end{equation*}
\end{lemma}
\begin{proof}
    We show this by induction. The case where \( i = 0 \) is obvious. Now let \( i \ge 0 \) and suppose \( 0 \le W^{n}_{i} \le \Upsilon^{n}_{i} \). On the event \( \{ C^{n}_{i} < 0 \}  \), definition \eqref{eq:Wn_lindley_def} implies that $W^{n}_{i+1} \le \max \{ 0, X^{n}_{i} \} \le \Upsilon^{n}_{i+1}$. 
    And on the event \( \{ C^{n}_{i} \ge  0 \}  \), we have $W^{n}_{i+1} \le \max \{ 0, C^{n}_{i} \Upsilon^{n}_{i} + X^{n}_{i} \}  = \Upsilon^{n}_{i+1}$. 
    This concludes the proof. 
\end{proof}

Now consider the following two families of linearly recursive Markov systems \( V^{n}(t) \equiv V^{n}_{\left\lfloor nt \right\rfloor} \) and \( U^{n}(t) \equiv U^{n}_{\left\lfloor nt \right\rfloor} \) with different initial conditions: 
\begin{equation*}
    \begin{cases}
        V^{n}_{i+1} = C^{n}_{i} V^{n}_{i} + X^{n}_{i}, & i \in \mathbb{N}, \\
        V^{n}_{0} = W^{n}_{0}, 
    \end{cases}
\end{equation*}
and 
\begin{equation*} 
    \begin{cases}
        U^{n}_{i+1} = C^{n}_{i} U^{n}_{i} + X^{n}_{i}, & i \in \mathbb{N}, \\
        U^{n}_{0} = 0. 
    \end{cases}
\end{equation*}
By simple induction, we have 
\[
\begin{aligned}
    V^{n}_{i+1} &= X^{n}_{i} + C^{n}_{i} X^{n}_{i-1} + \cdots + C^{n}_{i} \cdots C^{n}_{1} X^{n}_{0} + C^{n}_{i} \cdots C^{n}_{1} C^{n}_{0} W^{n}_{0} ,
    \\
    U^{n}_{i+1} &= X^{n}_{i} + C^{n}_{i} X^{n}_{i-1} + \cdots + C^{n}_{i} \cdots C^{n}_{1} X^{n}_{0}   . 
\end{aligned}
\]
Take mutually independent, i.i.d.\ sequences \( \{ \Theta'_i,\ i \ge 0 \} \) and \( \{ X'^{,n}_{i},\ i \ge 0 \} \) that have the same distributions as \( \{ \Theta_i, \ i \ge 0 \} \) and \( \{ X^{n}_{i},\ i\ge 0 \} \) respectively. Similar to the definition of \( C^n_i\) in Assumption \ref{assump:model}, we let \(C'^{,n}_i \equiv 1 - n^{-1} \Theta'_i\).  Using these random variables, we define for \( i \in \mathbb{N}_0 \), 
\[
\begin{aligned}
    V'^{,n}_{i+1}  &\equiv X'^{,n}_{0} + C'^{,n}_{0} X'^{,n}_{1} + \cdots + C'^{,n}_{0} \cdots C'^{,n}_{i-1} X'^{,n}_{i} + C'^{,n}_{0} \cdots C'^{,n}_{i-1} C'^{,n}_{i} W^{n}_{0},
    \\
    U'^{,n}_{i+1}  &\equiv X'^{,n}_{0} + C'^{,n}_{0} X'^{,n}_{1} + \cdots + C'^{,n}_{0} \cdots C'^{,n}_{i-1} X'^{,n}_{i}, 
    \\
    \Upsilon'^{,n}_{i+1} &\equiv \max \{  0,\ X'^{,n}_{0},\ X'^{,n}_{0} + C'^{,n}_{0} X'^{,n}_{1},\ \ldots, \  X'^{,n}_{0} + C'^{,n}_{0} X'^{,n}_{1} + \cdots + C'^{,n}_{0} \cdots C'^{,n}_{i-2} X'^{,n}_{i-1}, \\
    & \quad \quad  X'^{,n}_{0} + C'^{,n}_{0} X'^{,n}_{1} + \cdots + C'^{,n}_{0} \cdots C'^{,n}_{i-2} X'^{,n}_{i-1} + C'^{,n}_{0} \cdots C'^{,n}_{i-1} X'^{,n}_{i} + C'^{,n}_{0} \cdots C'^{,n}_{i} W^{n}_{0} \},
\end{aligned} 
\]
with \( V'^{,n}_{0} = W^{n}_{0} \), \( U'^{,n}_{0} = 0 \) and \( \Upsilon'^{,n}_{0} = W^{n}_{0}\).

Observe that for any \( i \in \mathbb{N}_0 \), 
\begin{equation*}
    \Upsilon'^{,n}_{i} = \max \{ U'^{,n}_{0}, U'^{,n}_{1}, \ldots , U'^{,n}_{i-1}, V'^{,n}_{i} \} \le  V'^{,n}_{i} \vee \max_{0 \le k \le i} U'^{,n}_{k}. 
\end{equation*}
Combining these observations, we have 
\begin{equation}\label{eq:Wn_bounds}
    0 \le  W^{n}_i \le \Upsilon^{n}_{i} 
    \stackrel{(d)}{=} 
    \Upsilon'^{,n}_{i} 
    \le 
      V'^{,n}_{i} \vee \max_{0 \le k \le i} U'^{,n}_{k}  
    \stackrel{(d)}{=} 
     V^{n}_{i} \vee \max_{0 \le k \le i}U^{n}_{k}. 
\end{equation}
where $\stackrel{(d)}{=}$ is used to denote equality in distribution. 

\begin{lemma} \label{lem:W_bar_exp_stocastic_bdd}
  Under Assumptions \ref{assump:model} and \ref{assump:MDP_W}, 
    \begin{equation*}
        \lim_{K \to \infty} \limsup_{ n \to \infty} \frac{1}{b_n^{2}} \log \mathbb{P} \left( \| \bar{W}^{n} \|_T > K   \right) < - \infty. 
    \end{equation*}
\end{lemma}

\begin{proof}
    Following our convention for fluid scalings, denote \( \bar{V}^{n} = n^{-1} V^{n} \), \( \bar{V}'^{,n} = n^{-1} V'^{,n} \), \( \bar{U}^{n} = n^{-1} U^{n} \), \( \bar{U}'^{,n} = n^{-1} U'^{,n} \), \( \bar{\Upsilon}^{n} = n^{-1} \Upsilon^{n} \), \( \bar{\Upsilon}'^{,n} = n^{-1} \Upsilon'^{,n} \), \( \bar{\Lambda}^{n} = n^{-1} \Lambda^{n} \), \( \bar{\Lambda}'^{,n} = n^{-1} \Lambda'^{,n} \). By Lemma \ref{lem:exp_bounded_Vbar}, we have 
    \[
    \begin{aligned}
        \lim_{K \to \infty} \limsup_{n \to \infty} \frac{1}{b_n^{2}} \log \mathbb{P} \left( \| \bar{V}^{n} \|_T > K \right) = - \infty, \\
        \lim_{K \to \infty} \limsup_{n \to \infty} \frac{1}{b_n^{2}} \log \mathbb{P} \left( \| \bar{U}^{n} \|_T > K \right) = - \infty. 
    \end{aligned}
    \]
    Observe that for any process \( \phi \) in \( \mathcal{D}_T \), denoting \( \phi_{\uparrow}(t) = \sup_{0 \le u \le t} \phi(u) \), we have 
    \begin{equation*}
        \| \phi_{\uparrow} \|_T = \sup_{t \in [0,T]} | \sup_{u \in [0,t]} \phi(u) | \le \sup_{t \in [0,T]}  \sup_{u \in [0,t]} |\phi(u)|  \le  \| \phi \|_T  \,.
    \end{equation*}
    Then \eqref{eq:Wn_bounds} implies that 
    \begin{align*}
        \mathbb{P} \left( \| \bar{W}^{n} \|_{T} > K_{\alpha} \right)
        &\le 
         \mathbb{P} \left( \| \bar{\Upsilon}^{n} \|_T   > K_{\alpha} \right) 
        \\
        &=  \mathbb{P} \left(  \| \bar{\Upsilon}'^{,n} \|_T   > K_{\alpha} \right)
        \\
        &\le \mathbb{P} \left( \| \bar{V}'^{,n}  \|_T \vee \| \bar{U}'^{,n}_{\uparrow} \|_T  > K_{\alpha} \right)  
        \\
        &\le \mathbb{P} \left( \| \bar{V}'^{,n} \|_T \vee \| \bar{U}'^{,n} \|_T  > K_{\alpha} \right)  
        \\
        &\le \mathbb{P} \left( \| \bar{V}^{n} \|_T > K_{\alpha}  \right) + \mathbb{P} \left( \| \bar{U}^{n} \|_T > K_{\alpha}  \right). 
    \end{align*}
    Then the lemma follows from Remark \ref{rmk:analysis_facts}.
\end{proof}

The following lemma is a consequence of Lemma \ref{lem:W_bar_exp_stocastic_bdd}, using the same arguments that were used to show Lemma \ref{lem:exp_tight_esp_V_2} and \eqref{eq:lem_V_esp_1}. 
\begin{lemma}
    \label{lem:eps_tilde_123}
   Under Assumptions \ref{assump:model} and \ref{assump:MDP_W}, with rate \( b_n^{2} \), the family \(\{ \widetilde{\epsilon}^{2,n},\; n \in \mathbb{N} \} \) is exponentially tight in $\mathcal{D}_T$, and the families \( \{\widetilde{\epsilon}^{1,n},\; n \in \mathbb{N}\} \), \( \{\widetilde{\epsilon}^{3,n},\; n \in \mathbb{N} \} \) are exponentially equivalent to the zero process. 
\end{lemma}

\subsection{Proof of Theorem \ref{thm:MDP_W} under non-zero centering}

The goal of this section is to show Theorem \ref{thm:MDP_W} (i), which is the case where \( \mu > 0 \), \( \theta > 0 \) and \( \bar{W}^* = \mu/\theta \). Recalling \eqref{eq:W_fluid_representation}, 
let \( \xi \in \mathcal{D}_T \) be defined by 
\begin{equation} \label{eq:xi_def}
    \xi^{n}(t) = \bar{W}^{n}_{0} + \frac{1}{n} \sum_{i=0}^{\left\lfloor nt \right\rfloor -1} X^{n}_{i} + \bar{\epsilon}^{1,n}(t) +  \bar{\epsilon}^{2,n}(t), \quad t \in [0,T]. 
\end{equation}
 
We analyze each term in \eqref{eq:xi_def}. 
First, Assumption \ref{assump:MDP_W} (i), Theorem \ref{thm:MDP_RW} and Lemma 4.2(b) in \cite{PuhW97} imply that
\begin{equation} \label{eq:xi_1}
    \bar{W}^{n}_{0} \stackrel{P ^{1/b_n^{2}}}{\longrightarrow} \bar{W}^* \quad \textrm{and} \quad \frac{1}{n} \sum_{i=0}^{\left\lfloor n \cdot \right\rfloor-1} X^{n}_{i} - \mu_n \mathfrak{e} \stackrel{P ^{1/b_n^{2}}}{\longrightarrow} 0 . 
\end{equation}
With some algebra, we obtain 
\begin{equation*}
    \widetilde{\epsilon}^{1,n}(t)  = \theta \left( \int_{0}^{t} \widetilde{W}^{n}(s) \mathrm{d}s - \frac{1}{n} \sum_{i=0}^{\left\lfloor nt \right\rfloor -1 } \widetilde{W}^{n}_{i}  \right)  = \frac{\sqrt{n}}{b_n} \bar{\epsilon}^{1,n}(t) +  \mu \frac{nt - \left\lfloor nt \right\rfloor}{b_n \sqrt{n}}  .  
\end{equation*}
In the relation above, \( \{ \widetilde{\epsilon}^{1,n} \}_{n \in \mathbb{N}}\) is exponentially equivalent to \( 0 \) with rate \(b_n^2\) by Lemma \ref{lem:eps_tilde_123}. The deterministic process \( (b_n \sqrt{n})^{-1}(n \mathfrak{e} - \left\lfloor n \mathfrak{e} \right\rfloor)\mu \) uniformly converges to \( 0 \), hence it is also exponentially equivalent to \( 0 \) with rate \(b_n^2\). Then since \( \sqrt{n}/b_n \to \infty \), Lemma 4.2 (b) in \cite{PuhW97} implies 
\begin{equation} \label{eq:xi_2}
    \bar{\epsilon}^{1,n} \stackrel{P ^{1/b_n^{2}}}{\longrightarrow} 0. 
\end{equation} 
Next, with some algebra we can also obtain  
\begin{equation*}
    \widetilde{\epsilon}^{2,n}(t) = \frac{1}{b_n \sqrt{n}}  \sum_{i=0}^{\left\lfloor nt \right\rfloor-1} (\theta - \Theta_{i})  ( \bar{W}^{n}_{i} - \bar{W}^{*}) = \frac{\sqrt{n}}{b_n} \bar{\epsilon}^{2,n}(t) +  \bar{W}^{*} \cdot \widetilde{R}^n_\Theta(t).  
\end{equation*}
Lemma \ref{lem:eps_tilde_123} and Theorem \ref{thm:MDP_RW} imply that \( \{b_n^{-1}\sqrt{n} \bar{\epsilon}^{2,n} \}_{n \in \mathbb{N}}\) is exponentially tight in \(\mathcal{D}_T\) with rate \( b_n^{2} \). By Lemma \ref{lem:exptight_fluid0}, we have 
\begin{equation} \label{eq:xi_3}
    \bar{\epsilon}^{2,n} \stackrel{P ^{1/b_n^{2}}}{\longrightarrow} 0. 
\end{equation}
Finally, combining \eqref{eq:xi_def} with \eqref{eq:xi_1}, \eqref{eq:xi_2} and \eqref{eq:xi_3} yields  
\begin{equation}\label{lem:xi_expequiv}
    \xi^{n} - (\bar{W}^* + \mu_{n}\mathfrak{e}) \stackrel{P ^{1/b_n^{2}}}{\longrightarrow} 0. 
\end{equation}

Next, we claim that 
\begin{equation} \label{eq:Ltilde_expeq_0}
    \frac{\sqrt{n}}{b_n} \bar{L}^{n} \stackrel{P ^{1/b_n^{2}}}{\longrightarrow} 0.  
\end{equation} 
To see this, consider the event \( \Gamma^{n} := \{ \| \xi^{n} - (\bar{W}^{*} + \mu_{n} \mathfrak{e}) \|_T \le  \bar{W}^* \} \). Since \( \mu_n \to  \mu > 0 \), we have \( \mu_n > 0 \) when $n$ is large enough. 
Hence on \( \Gamma^{n} \) with $n$ large, \( \xi^{n} (t) > 0 \) for all \( t > 0 \). By \eqref{eq:fluid_W_reflected}, we can write \( \bar{L}^{n} = \mathcal{R}'_{\theta}(\xi^{n}) \). This implies \(\xi^n  \mathrm{d} \bar{L}^n = 0\) and therefore \( \bar{L}^{n} \equiv 0 \). Therefore, for any \( \epsilon > 0 \) and \( n \) large enough, 
\begin{align*}
    \mathbb{P} \left( \| \frac{\sqrt{n}}{b_n}  \bar{L}^{n} \| > \epsilon \right) 
    &\le \mathbb{P} \left( \| \frac{\sqrt{n}}{b_n}  \bar{L}^{n} \|_T > \epsilon ,\  \Gamma^{n} \right) + \mathbb{P} \left(\| \xi^{n} - (\bar{W}^{*} + \mu_{n} \mathfrak{e}) \|_T > \bar{W}^{*}  \right) 
    \\
    &= \mathbb{P} \left(\| \xi^{n} - (\bar{W}^{*} + \mu_{n} \mathfrak{e}) \|_T > \bar{W}^{*}  \right) . 
\end{align*}
Then by \eqref{lem:xi_expequiv} and Lemma \ref{lem:supexp_char}, we have 
\begin{equation*}
    \limsup_{n \to \infty} \frac{1}{b_n^{2}} \log \mathbb{P}  \left( \| \frac{\sqrt{n}}{b_n}  \bar{L} \| > \epsilon \right) 
    \le 
    \limsup_{n \to \infty} \frac{1}{b_n^{2}} \log \mathbb{P} \left(\| \xi^{n} - (\bar{W}^{*} + \mu_{n} \mathfrak{e}) \|_T > \bar{W}^{*}  \right) = - \infty, 
\end{equation*}
which implies \eqref{eq:Ltilde_expeq_0}. 

Finally, we write \eqref{eq:md_representation} as 
\begin{equation*}
    \widetilde{W}^{n}
    = \mathcal{M}_\theta \left(  \widetilde{W}^{n}_0 +  \widetilde{R}^n_X -   \bar{W}^{*} \widetilde{R}^n_\Theta    + \frac{\sqrt{n}}{b_n}(\mu_n - \mu)\mathfrak{e}  + \widetilde{\epsilon}^{1,n} + \widetilde{\epsilon}^{2,n} + \widetilde{\epsilon}^{3,n} +  \frac{\sqrt{n}}{b_n} \bar{L}^{n}  \right). 
\end{equation*}
By Assumption \ref{assump:MDP_W}, Theorem \ref{thm:MDP_RW}, Lemma \ref{lem:eps_tilde_123} and \eqref{eq:Ltilde_expeq_0}, we first apply Lemma \ref{lem:exp_tightness_cts_map} and conclude that \(\{ \widetilde{W}^{n}\}_{n \in \mathbb{N}} \) is exponentially tight in $\mathcal{D}_T$. By Lemma \ref{lem:exptight_fluid0}, this implies 
\begin{equation*}
     \bar{W}^{n} - \bar{W}^* \stackrel{P^{1/b_n^{2}}}{\longrightarrow}  0. 
\end{equation*}
Then by the arguments used in the proof of Lemma \ref{lem:eps_V_expeq_0}, we can show 
\begin{equation} \label{eq:eps_tilde_2_expeq_0}
    \widetilde{\epsilon}^{2,n} \stackrel{P^{1/b_n^{2}}}{\longrightarrow}  0.
\end{equation}
Since the map \( \mathcal{M}_\theta \) is continuous, once again using Assumption \ref{assump:MDP_W}, Theorem \ref{thm:MDP_RW}, Lemma \ref{lem:eps_tilde_123} and \eqref{eq:Ltilde_expeq_0}, along with \eqref{eq:eps_tilde_2_expeq_0}, we apply the contraction principle and obtain the MDP result given by Theorem \ref{thm:MDP_W} (i).

\subsection{Proof of Theorem \ref{thm:MDP_W} under zero centering}
Now we turn to the cases where the centering term \( \bar{W}^* = 0 \). By Table \ref{table:fluid_stability}, this occurs when \( \mu \le 0 \). 

We first consider the case where \( \mu = 0 \), which corresponds to Theorem \ref{thm:MDP_W} (ii). 
The relation \eqref{eq:md_representation} and the same arguments used to derive \eqref{eq:fluid_W_reflected} imply that 
\begin{equation} 
    \label{eq:W_tilde_rep_mu_0}
    \widetilde{W}^{n} 
    = \mathcal{R}_\theta \left( \widetilde{\Phi}^{n} \right) ,
\end{equation}
where we let  
\[
    \widetilde{\Phi}^{n} := \widetilde{W}^{n}_0 +  \widetilde{R}^n_X  -   \bar{W}^{*} \widetilde{R}^n_\Theta    + \frac{\sqrt{n}}{b_n}(\mu_n - \mu)\mathfrak{e}    + \widetilde{\epsilon}^{1,n}  + \widetilde{\epsilon}^{2,n}  + \widetilde{\epsilon}^{3,n},  
\]
and use \( \mathcal{R}_{0} \) to denote the conventional Skorokhod reflection mapping \( \mathcal{R} \) when \( \theta = 0 \). 

By Assumption \ref{assump:MDP_W}, Theorem \ref{thm:MDP_RW}, Lemma \ref{lem:eps_tilde_123}, continuity of the mapping \( \mathcal{R}_\theta \) and Lemma \ref{lem:exp_tightness_cts_map}, we conclude that \(\{\widetilde{W}^{n} \}_{n \in \mathbb{N}}\) is exponentially tight with rate \( b_n^{2} \). We can once again use the arguments in Section \ref{subsec:MDP_V_proof}: first concluding \( \bar{W}^{n} \stackrel{P^{1/b_n^{2}}}{\longrightarrow} 0 \), then \( \widetilde{\epsilon}^{2,n} \stackrel{P^{1/b_n^{2}}}{\longrightarrow} 0 \), and finally applying the contraction principle to obtain the MDP result given in Theorem \ref{thm:MDP_W} (ii). 

Now let \( \mu < 0 \), instead of \eqref{eq:W_tilde_rep_mu_0}, we have 
\begin{equation} 
    \label{eq:W_tilde_rep_mu_ll_0}
    \widetilde{W}^{n} 
    =    \mathcal{R}_\theta \left(  \widetilde{\Phi}^{n} +  \frac{\sqrt{n}}{b_n} \mu \mathfrak{e}   \right). 
\end{equation}
Since \( \widetilde{W}^{n}(t) = \frac{\sqrt{n}}{b_n} \bar{W}^{n}(t) \), we can write 
\begin{equation}
    \label{eq:W_bar_rep_mu_ll_0}
    \bar{W}^{n}(t) = \frac{b_n}{\sqrt{n}} \widetilde{\Phi}^{n} + \mu t - \int_{0}^{t} \theta \bar{W}^{n} \mathrm{d}s + \bar{L}^{n}(t)   = \mathcal{R}_\theta \left( \frac{b_n}{\sqrt{n}} \widetilde{\Phi}^{n}(t) + \mu t  \right). 
\end{equation}
Similar to the above, we can use Assumption \ref{assump:MDP_W}, Theorem \ref{thm:MDP_RW}, Lemma \ref{lem:eps_tilde_123} to conclude \( \{\widetilde{\Phi}^{n} \}_{n \in \mathbb{N}}\) is exponentially tight in \(\mathcal{D}_T\) with rate \( b_n^{2} \). We note that \( \{\widetilde{\Phi}^{n} \}_{n \in \mathbb{N}}\) is in fact \( \mathcal{C} \)-exponentially tight since in the proof of Lemma \ref{lem:exp_tight_esp_V_2}, we checked the conditions for \( \mathcal{C} \)-exponential tightness in Theorem \ref{thm:C_exp_tightness_in_D}. Lemma \ref{lem:exptight_fluid0} then implies that 
\begin{equation*}
    \frac{b_n}{\sqrt{n}} \widetilde{\Phi}^{n}  \stackrel{P ^{1/b_n^{2}}}{\longrightarrow} 0 .  
\end{equation*}
By \eqref{eq:W_bar_rep_mu_ll_0}, the contraction principle and the fact that \( \mathcal{R}_{\theta}(\mu \mathfrak{e}) \equiv 0 \), we have 
\begin{equation}\label{eq:mdp_W_pf_fluid_0}
    \bar{W}^{n} \stackrel{P ^{1/b_n^{2}}}{\longrightarrow} 0 .  
\end{equation} 
Let \( \epsilon > 0 \) such that \( \mu + |\theta| \epsilon < 0 \). On the event \( \Lambda^{n} = \{ \| \bar{W}^{n} \|_T < \epsilon \} \), observe that  \( 0 \le  \widetilde{W}^{n} < b_n^{-1} \sqrt{n}\epsilon  \), and therefore 
\begin{equation*}
     \widetilde{\Phi}^{n}(t) + \frac{\sqrt{n}}{b_n} \mu t + \int_{0}^{t} (-\theta) \widetilde{W}^{n}(s) \mathrm{d}s  
    \le   \widetilde{\Phi}^{n}(t) + \frac{\sqrt{n}}{b_n} \mu t + \int_{0}^{t} \vert \theta \vert \frac{\sqrt{n}}{b_n} \epsilon \  \mathrm{d} s = \widetilde{\Phi}^{n}(t) + \frac{\sqrt{n}}{b_n} (\mu + \vert \theta \vert \epsilon ) t .    
\end{equation*}
Since \( |\theta| b_n^{-1} \sqrt{n} \epsilon +  \theta \widetilde{W}^{n}(t) \ge 0 \)  on the event \( \Lambda^{n} \), we use \eqref{eq:lgrm} and Lemma \ref{lem:Skorokhod_comparison} to get 
\begin{align*}
    0 \le  \widetilde{W}^{n} = \mathcal{R}_\theta \left( \widetilde{\Phi}^{n} + \frac{\sqrt{n}}{b_n} \mu \mathfrak{e} \right)  
    &= \mathcal{R} \left( \widetilde{\Phi}^{n} + \frac{\sqrt{n}}{b_n} \mu \mathfrak{e} + \int_{0}^{t} (-\theta) \widetilde{W}^{n}(s) \mathrm{d}s \right) \\
    & \le \mathcal{R} \left( \widetilde{\Phi}^{n} + \frac{\sqrt{n}}{b_n} (\mu + \vert \theta \vert \epsilon) \mathfrak{e}  \right). 
\end{align*}
Let \( \delta > 0 \) be arbitrary. We have 
\begin{align*}
    \mathbb{P} \left( \| \widetilde{W}^{n} \|_T > \delta \right)
    & \le \mathbb{P} \left(\| \widetilde{W}^{n} \|_T > \delta,\ \Lambda^{n}  \right) + \mathbb{P} \left( \| \bar{W}^{n} \|_T > \epsilon \right) \\
    & \le \mathbb{P} \left(\left\|  \mathcal{R} \left( \widetilde{\Phi}^{n} + \frac{\sqrt{n}}{b_n} (\mu + \vert \theta \vert \epsilon) \mathfrak{e}  \right) \right\|_T > \delta \right) + \mathbb{P} \left( \| \bar{W}^{n} \|_T > \epsilon \right). 
\end{align*}
By Lemma \ref{lem:reflection_expequiv_0}, \eqref{eq:mdp_W_pf_fluid_0}, Lemma \ref{lem:supexp_char} and Remark \ref{rmk:analysis_facts}, we obtain  
\begin{equation*}
    \limsup_{ n \to \infty} \frac{1}{b_n^{2}} \log \mathbb{P} \left( \| \widetilde{W}^{n} \|_T > \delta \right) = - \infty.  
\end{equation*}
This concludes the proof for the last statement in Theorem \ref{thm:MDP_W} where \( \mu < 0 \).

\subsection{Proof of Theorem \ref{thm:MDP_W_explicit_rate_fn}} 

The proof here mirrors that of Theorem \ref{thm:MDP_V_explicit_rate_fn}. Case (i) is exactly the same. For cases (ii) and (iii), it suffices to optimize over the set \( \{  \psi_1  \subseteq   \mathcal{D}_T:  \psi_1 \in \mathcal{AC}_0 \} \), with the rate function being infinite everywhere else. Let \( \phi \in \mathcal{D}_T \) be given such that \( \phi = \mathcal{R}_\theta(w_0 + \psi_1+ r \mathfrak{e}) \). Then \(\phi \) is non-negative with \(\phi(0) = w_0\). By Lemma A.1 in \cite{FHP25}, we further have \( \phi  \in \mathcal{AC}\) and there exists a \( y \in \mathcal{AC}_0 \) such that \( \dot{\phi} + \theta \phi = \dot{\psi}_1 + r + \dot{y} \),  \( \dot{y}(t) \ge 0 \) and \( \phi(t) \dot{y}(t) = 0 \) a.e.  
Then the problem reduces to solving the following convex optimization problem a.e. in time \( t \):   
\begin{equation*}
\begin{aligned}
    \min_{\dot{\psi}_{1}(t) \in \mathbb{R}} \quad & \frac{1}{2 \sigma^{2}_{X}} \dot{\psi}_{1}(t)^{2}  \\
    \textrm{s.t.}\quad & \dot{\phi}(t) = \dot{\psi}_{1}(t) +r - \theta \phi (t) + \dot{y}(t). 
\end{aligned}
\end{equation*}
On the event \( \{t: \phi(t) > 0 \} \), we have \( \dot{y}(t) = 0 \) a.e., and then the solution is the same as case (i) with \( \mu = 0 \). On \( \{t:\phi(t) = 0 \}\), again by \cite{FHP25} Lemma A.1, we also have \( \dot{\phi} = 0 \) a.e. So \( \dot{\psi}_1(t) = - (r+ \dot{y} (t)) \) and the problem is equivalent to solving 
\begin{equation*}
\begin{aligned}
    \min_{ \dot{y}(t) \in \mathbb{R}} \quad & \frac{1}{2 \sigma^{2}_{X}} (r  + \dot{y}(t))^{2} \\
    \textrm{s.t.}\quad &   \dot{y}(t) \ge 0. 
\end{aligned}
\end{equation*}
By standard techniques, we see that when \( r \ge 0 \), \( \dot{y} (t) = 0 \) and when \( r < 0 \), \( \dot{y}(t) = -r \). Combining the above arguments, we obtain the rate function for cases (ii) and (iii).

\clearpage

\appendix

\section{Notation Table}\label{app:notation_table}
\begin{table}[H]
\centering
\caption{Frequently used notation.}
\label{tab:notation}
\begin{tabularx}{\textwidth}{@{}l p{0.6\textwidth} X@{}}
\hline
Notation & Meaning & First appearance \\
\hline

\(W_i^n\) 
& Waiting time of the \(i\)-th customer in the \(n\)-th system 
& Section \ref{sec:intro} Eq.\eqref{eq:lindley_recursion} \\

\(W^n(t)\) 
& Re-indexed waiting-time process 
& Section \ref{subsec:the_model} Eq.\eqref{eq:process_W} \\

\(\bar{W}^n\) 
& Fluid-scaled waiting-time process
& Section \ref{subsec:MD_results} Eq.\eqref{eq:fluid_scaling_W} \\

\(\widetilde{W}^n\) 
& MD-scaled waiting-time process centered at \(\bar{W}^*\) 
& Section \ref{subsec:MD_results} Eq.\eqref{eq:W_tilde}\\

\(\widehat{W}^n\) 
& Diffusion-scaled waiting-time process centered at \(\bar{W}^*\) 
& Appendix \ref{app:diffusion_approx} Eq.\eqref{eq:fclt_scalings}\\

\(\bar{W}^*\) 
& Stable fixed point of the fluid limit of \(\bar{W}^n\) 
& Section \ref{subsec:MD_results},  Table \ref{table:fluid_stability}\\

\hline
\(L^n\) 
& Regulator process associated with the reflected waiting-time recursion 
& Section \ref{subsec:the_model} Eq.\eqref{eq:waiting_time_representation}\\

\(\bar{L}^n\) 
& Fluid-scaled regulator process 
& Section \ref{subsec:fluid_limit} Eq.\eqref{eq:fluid_W_reflected}\\

\(\widetilde{L}^n\) 
& MD-scaled regulator process 
& Section \ref{sec:MDP_W} Eq.\eqref{eq:md_representation}\\

\hline
\(V_i^n\) 
& Unreflected linearly recursive process 
& Section \ref{sec:intro} Eq.\eqref{eq:intro_markov_chain} \\

\(\bar{V}^n\) 
& Fluid-scaled unreflected process 
& Section \ref{subsec:fluid_analysis_V} Eq.\eqref{eq:V_bar_rep}\\

\(\widetilde{V}^n\) 
& MD-scaled unreflected process centered at \(\bar{V}^*\) 
& Section \ref{subsec:V_MDP_results} Eq.\eqref{eq:md_V}\\

\(\bar{V}^*\) 
& Stable fixed point of the fluid limit of \(\bar{V}^n\) 
& Section \ref{subsec:fluid_analysis_V} Table \ref{table:fluid_stability_V} \\

\hline
\(X_i^n\) 
& Nominal increment variable, \(X_i^n=\mathfrak{S}_i^n-\mathfrak{A}_i^n\) 
& Section \ref{subsec:the_model} Eq.\eqref{eq:recursion_W}\\

\(C_i^n\) 
& Random coefficient in the recursion, \(C_i^n=1+B_i^n-A_i^n\) 
& Section \ref{subsec:the_model} Eq.\eqref{eq:recursion_W}\\

\(\Theta_i\) 
& Random variable satisfying \(C_i^n=1-\Theta_i/n\) 
& Assumption \ref{assump:model} \\

\(\mu_n\), \(\mu\) 
& Mean of \(X_0^n\), and its limit \(\mu_n\to\mu\) 
& Assumption \ref{assump:model} \\

\(\theta\) 
& Mean of \(\Theta_0\) 
& Assumption \ref{assump:model} \\

\(\sigma_X^2\), \(\sigma_\Theta^2\) 
& Limiting variance parameters of \(X_0^n\) and \(\Theta_0\) 
& Assumption \ref{assump:model} \\

\(\widetilde{R}_X^n\) 
& MD-scaled centered random walk generated by \(X_i^n\) 
& Section \ref{subsec:MD_results} Eq.\eqref{eq:RW_MD_scaling}\\

\(\widetilde{R}_\Theta^n\) 
& MD-scaled centered random walk generated by \(\Theta_i\) 
& Section \ref{subsec:MD_results} Eq.\eqref{eq:RW_MD_scaling}\\

\hline 
\(\mathcal{R}\), \(\mathcal{R}'\) 
& Conventional Skorokhod reflection map and regulator map 
& Section \ref{subsec:MD_results} \\

\(\mathcal{R}_\theta\), \(\mathcal{R}'_\theta\) 
& Linearly generalized Skorokhod map and regulator map 
& Section \ref{subsec:MD_results} \\

\(\mathcal{M}_\theta\) 
& Unreflected map with linear drift term  
& Section \ref{subsec:MD_results} \\
\hline
\end{tabularx}
\end{table}

\section{Proofs for Section \ref{sec:fluid_analysis}} \label{app:fluid_proofs}

First, Lemma 1 in \cite{Whi90} provides an alternative system that can be used to bound \(\bar{W}^n\). Specifically, consider the unreflected recursion:  
\begin{equation}
    \label{eq:Y_def}
    \begin{cases}
        \bar{Y}^{n}_{i+1} = (C^{n}_{i})^+ \bar{Y}^{n}_i + \frac{1}{n} (X^{n}_i)^+ , & \quad i \ge 0, \\
        \bar{Y}_{0}^{n} = \bar{W}^{n}_{0} . 
    \end{cases}
\end{equation}
Clearly, $\bar{W}_i^{n} \le \bar{Y}_{i}^{n}$ almost surely. It is beneficial to analyze the second moment of $\bar{Y}^{n}_i$, which is the content of the next lemma. 

\begin{lemma}\label{lem:estimate_Y}
    Let Assumption \ref{assump:model} hold, $\bar{W}^n_0 \to w_0$ in $L^2$ as $n \to \infty$, and $\bar{Y}^{n}_i$ be given by the recursion \eqref{eq:Y_def}. Then,   
    \[
        \sup_{n \ge 0} \sup_{0 \le i \le \left\lfloor nT \right\rfloor} \mathbb{E} \big[ (\bar{Y}^{n}_i)^2\big]  < \infty. 
    \]
\end{lemma}

\begin{proof}
    The assumptions imply that $\bar{Y}^n_0 \to w_0$ in $L^2$. Therefore, there exists some $\kappa_0 > 0$ such that 
    \begin{equation*}
        \mathbb{E}[\bar{Y}^{n}_{0}] < \kappa_0 \quad \textrm{and} \quad \mathbb{E}\big[(\bar{Y}^{n}_{0})^2\big] < \kappa_0.
    \end{equation*}
    Now denote \( \theta' :=  \mathbb{E} |\Theta_0|  \). Then, 
    \begin{equation*}
        \mathbb{E}[(C^{n}_{0})^{+}] \le \mathbb{E} |C^{n}_{0}| \le \mathbb{E} \left[ 1+ \frac{|\Theta_0|}{n} \right] = 1 + \frac{\theta'}{n}  \,. 
    \end{equation*}
    Also observe that 
    \begin{equation*}
        \left( \mathbb{E} [ (X^{n}_{0})^{+}] \right)^{2} \le \mathbb{E} [ ((X^{n}_{0})^{+})^{2}] \le  \mathbb{E} \left[ (X^{n}_{0})^{2} \right] = \sigma_{X,n}^{2} + \mu_n^{2}. 
    \end{equation*}
    By Assumption \ref{assump:model}, \( \sigma_{X,n}^{2} + \mu_{n}^{2} \) converges as \( n \to \infty \), hence there exists some \( \kappa_1 > 0 \) such that 
    \begin{equation*}
        \mathbb{E} [ (X^{n}_{0})^{+}] \le  \kappa_1 \quad \textrm{and} \quad \mathbb{E} \big[ ((X^{n}_{0})^{+})^{2} \big] \le \kappa_1. 
    \end{equation*}
    Then we have the recursive inequality
    \begin{equation*}
        \mathbb{E} [\bar{Y}^{n}_{i+1}] = \mathbb{E} [(C^{n}_{i})^{+}] \mathbb{E} [\bar{Y}^{n}_{i}] + \frac{1}{n} \mathbb{E} [ (X^{n}_{i})^{+}] \le (1+\frac{\theta'}{n} ) \mathbb{E} [\bar{Y}^{n}_{i}] + \frac{\kappa_{1}}{n} \,. 
    \end{equation*}
    Let \( \kappa_2  := e^{\theta'T} (\kappa_0 + \kappa_1 T) \). Then by Lemma \ref{lem:discrete_gronwall}, 
    \begin{equation*}
        \max_{0 \le i \le \left\lfloor nT \right\rfloor} \mathbb{E} [\bar{Y}^{n}_{i+1}] \le e^{\theta' T} \left( \mathbb{E}[\bar{Y}^{n}_{0}] + \kappa_{1} T \right) \le \kappa_{2}. 
    \end{equation*}
    We can obtain a similar bound for the second moment. Let \( \kappa_3 := 2 |\theta| + \sigma^{2}_{\Theta} + \theta^{2} \). Then,  
    \begin{equation*}
        \mathbb{E}[ ((C^{n}_{i})^{+})^{2}] \le \mathbb{E}[  (C^{n}_{i} )^{2}] = 1 - \frac{2\theta}{n} + \frac{\sigma^{2}_{\Theta}+ \theta^{2}}{n}  \le 1 + \frac{\kappa_3}{n}. 
    \end{equation*} 
    Now letting \( \kappa_4 := 2(1+\theta')\kappa_1\kappa_2 + \kappa_1 \), we can obtain another recursive inequality for the second moment: 
    \begin{align*}
        \mathbb{E}\left[ (\bar{Y}^{n}_{i+1})^2 \right] 
        &= 
        \mathbb{E} \left[ \left( (C^{n}_{i})^{+} \bar{Y}^{n}_i + \frac{1}{n} (X^{n}_i)^{+} \right)^2 \right] \nonumber\\
        &= 
        \mathbb{E}[ ((C^{n}_{i})^{+})^{2}] \mathbb{E} [(\bar{Y}^{n}_i)^2]  + \frac{2}{n} \mathbb{E}[  (C^{n}_{i})^{+} ]\mathbb{E}[(X^{n}_i)^{+}] \mathbb{E} [\bar{Y}^{n}_i] + \frac{1}{n^2} \mathbb{E}[([X^{n}_i]^{+})^2] \nonumber\\
        &\le 
        \left( 1+\frac{\kappa_3}{n}  \right) \mathbb{E} [(\bar{Y}^{n}_i)^2] + \frac{1}{n} \left( 2 (1 + \frac{\theta'}{n} ) \kappa_1 \kappa_2   + \frac{\kappa_1}{n}  \right)   \nonumber \\
        &\le 
        \left( 1+\frac{\kappa_3}{n}  \right) \mathbb{E} [(\bar{Y}^{n}_i)^2] + \frac{1}{n} \kappa_4 . 
    \end{align*}
    Again by Lemma \ref{lem:discrete_gronwall}, we obtain 
    \begin{equation*}
        \max_{0 \le i \le \left\lfloor nT \right\rfloor} \mathbb{E} \left[ (\bar{Y}^{n}_{i+1})^2 \right]  
        \le 
        e^{\kappa_3 T} \left( \mathbb{E}[ (\bar{Y}^{n}_0)^{2}] + \kappa_4 T  \right)  
        \le 
        e^{\kappa_3 T} \left( \kappa_0 + \kappa_4 T  \right). 
    \end{equation*}
    Finally we observe that the above bound does not depend on $n$ and conclude the proof. 
\end{proof}

\subsection{Proof of Lemma \ref{lem:fluid_error}} \label{subsec:pf_fluid_error}
 
    For the first statement, observe that 
    \begin{equation*}
        \bar{\epsilon}^{1,n}(t) = \theta \frac{nt - \lfloor nt \rfloor }{n} \bar{W}^{n}_{\lfloor nt \rfloor }, 
    \end{equation*}
   and hence,
    \[
        \sup_{t \in [0,T]} | \bar{\epsilon}^{1,n}(t) | \le  \max_{i = 1, \ldots, \lfloor nT \rfloor } \frac{|\theta| }{n} \bar{W}^{n}_i . 
    \]
    Then using the union bound and Chebyshev's inequality, we have 
    \begin{equation}
        \mathbb{P} \left( \sup_{t \in [0,T]} | \bar{\epsilon}^{1,n}(t)| > \epsilon \right) \le   \mathbb{P} \left( \max_{0 \le k \le \lfloor nT \rfloor } \frac{|\theta|}{n} \bar{W}^{n}_k > \epsilon \right) 
        \le  \sum_{k=0}^{\lfloor nT \rfloor } \mathbb{P} \left( \bar{W}^{n}_k > \frac{n \epsilon}{|\theta|} \right) 
        \le \sum_{k=0}^{\lfloor nT \rfloor } \frac{\theta^2 \mathbb{E}[(\bar{W}^{n}_k)^2]}{\epsilon^2 n^2} \,. \label{eq:fluid_pf_eps1}
    \end{equation}
    Since $\bar{W}^{n}_i \le \bar{Y}^{n}_i$ for all $i \ge 0$, by Lemma \ref{lem:estimate_Y}, there exists $\kappa > 0$ such that 
    \begin{equation} \label{eq:2nd_moment_bound_Wbar}
        \sup_{n \ge 1} \max_{0 \le k \le \lfloor nT\rfloor} \mathbb{E}[(\bar{W}^{n}_k)^2] \le \kappa. 
    \end{equation}
    Therefore, the expression in \eqref{eq:fluid_pf_eps1} converges to $0$ as $n \to \infty$. This concludes the proof for the first statement.  

    For the second statement, we analyze the partial sums   
    \[
        S_{i}^n := \sum_{m=0}^{i} (\Theta_m - \theta) \bar{W}^{n}_m. 
    \]
    Define the maximum of the partial sums by  
    \[
        M_i^n := \max_{0 \le m \le i} S_m^n. 
    \]
    Let $0 \le i,j \le \lfloor nT \rfloor$, observe that 
    \begin{align}
        \mathbb{E} |S_j^n - S_i^n |^2 &= \mathbb{E} \left(  \sum_{m=i+1}^{j} \left( \Theta_m  - \theta \right) \bar{W}^{n}_m \right)^2  \nonumber\\
        &= \sum_{i+1 \le m \le j} \mathbb{E} \left( \Theta_m - \theta \right)^2 \mathbb{E} \left( \bar{W}^{n}_m \right)^2 + \sum_{\substack{i+1 \le l,m \le j \\ l \neq m}} \mathbb{E} \left[ \left( \Theta_m - \theta \right) \left( \Theta_l - \theta \right) \bar{W}^{n}_m \bar{W}^{n}_l \right] \nonumber \\
        &= \sum_{i+1 \le m \le j} \mathbb{E} \left( \Theta_m - \theta \right)^2\mathbb{E} \left( \bar{W}^{n}_m \right)^2. 
        \label{eq:fluid_pf_eps2_1}
    \end{align}
    In \eqref{eq:fluid_pf_eps2_1}, we use the fact that the expectation of the off-diagonal terms is 0. To see this, we can assume without loss of generality that $m > l$. Then observe that $\Theta_m$ is independent of $\Theta_l \bar{W}^{n}_m \bar{W}^{n}_l $ and $\Theta_m - \theta$ has zero expectation.
    
    Define $u_m \equiv u :=  \sigma^2_{\Theta} \kappa \vee 1$. Then, by \eqref{eq:2nd_moment_bound_Wbar} and \eqref{eq:fluid_pf_eps2_1}, we have 
    \[
        \mathbb{E} |S_j^n - S_i^n |^2 \le \sum_{i+1 \le m \le j} \sigma^2_{\Theta} \kappa \le \left( \sum_{i+1 \le m \le j} u_m \right)^{3/2}. 
    \] 
    Using Markov's inequality, the conditions of Theorem 10.2 in \cite{Bill99} are satisfied with $\alpha = 3/4$ and $\beta = 1/2$. Therefore, there exists some $K'>0$ such that 
    \[
        \mathbb{P}\left( \sup_{t \in [0,T]} \left\lvert \epsilon^{n,2}(t) \right\rvert > \epsilon  \right) = \mathbb{P} \left( M_{\lfloor nT \rfloor }^n \ge n \epsilon \right) \le \frac{K'}{(nT) ^2 \epsilon^2} (nT u)^{3/2}. 
    \] 
    The above expression converges to $0$ as $n \to \infty$ and we conclude the proof.

\section{Diffusion Approximation}  \label{app:diffusion_approx}

Similar to the MDP setting, we once again limit ourselves to the cases under which the fluid limit \( \bar{W} \) has a stable fixed point and establish functional central limit theorems (FCLTs) for processes of the form 
\begin{equation} \label{eq:fclt_scalings}
    \hat{W}^{n} (t) = \sqrt{n}  \left( \bar{W}^{n} (t)- \bar{W}^{*} \right), \quad t \in [0,T], 
\end{equation}
with \( \bar{W}^{*} \) being the stable fixed points identified in Table \ref{table:fluid_stability}. Although the developments in this section are not  necessary for analyzing moderate deviations, we include them here to illustrate how proofs for MDP and FCLT are related. Further, the results in this section extend those of \cite{Whi90}, which focused on deriving normal approximations for the stationary distribution and did not provide an explicit diffusion limit. However, by analyzing \eqref{eq:fclt_scalings}, we show that this can be achieved in the present setting. 

First, we shall make several additional assumptions.  
\begin{assumption}[FCLT Assumptions] \label{assump:fclt} 
    \leavevmode
    \begin{enumerate} 
        \item[(i)] \( \hat{W}^{n}(0) \Rightarrow \hat{W}_0 \), where \( \hat{W}_0 \) is some proper random variable. 
        \item[(ii)] \( \sqrt{n} ( \mu_n - \mu) \to  \eta \) for some \( \eta \in \mathbb{R} \). 
    \end{enumerate}
\end{assumption}

\begin{remark} \label{rmk:fclt_constant_fluid}
    Assumption \ref{assump:fclt} (ii) specifies the rate at which the system reaches some nominal load regime. When \( \mu = 0 \), it can be associated with the heavy traffic condition for single server queues. To see this, note
    \begin{equation*}
        \sqrt{n} \mu_{n} = \sqrt{n} (\mathbb{E} \mathfrak{S}_0^{n} - \mathbb{E} \mathfrak{A}_{0}^{n}) = \sqrt{n} (\rho_{n} - 1) \mathbb{E} \mathfrak{A}_{0}^{n}. 
    \end{equation*}
    Suppose \( \mathbb{E} \mathfrak{A}^{n}_{0} \to  1/\lambda \). Then, \( \sqrt{n} \mu_{n} \to  \eta \) if and only if \( \sqrt{n} (\rho_{n} - 1) \to \eta \lambda \), as \( n \to  \infty \). 
\end{remark}

Similar to \eqref{eq:md_representation}, 
the first step is to approximate \eqref{eq:fclt_scalings} by a linear stochastic differential equation driven by two centered random walks, along with several asymptotically negligible terms: 
\begin{equation}\label{eq:fclt_representation}
\begin{aligned} 
    \hat{W}^{n}(t) 
    &=  \hat{W}^{n}(0) +  \frac{1}{\sqrt{n}} \sum_{i=0}^{\left\lfloor nt \right\rfloor-1 } \left( X^{n}_{i} - \mu_{n}  \right) + \frac{1}{\sqrt{n}} \sum_{i=0}^{\left\lfloor nt \right\rfloor -1} (\theta - \Theta_{i}) \bar{W}^{*} - \int_{0}^{t} \theta \hat{W}^{n} (s) \mathrm{d}s   \\
        & \quad \quad + \sqrt{n}  \left( \mu_n - \mu\right) t + \sqrt{n} \left( \mu - \theta \bar{W}^{*} \right) t + \hat{\epsilon}^{1,n}(t) + \hat{\epsilon}^{2,n}(t) + \hat{\epsilon}^{3,n}(t) +  \frac{1}{\sqrt{n}}{L}^{n}_{\left\lfloor nt \right\rfloor-1} ,
\end{aligned}
\end{equation}

where the error terms are 
\begin{equation*}
    \begin{aligned}
    \hat{\epsilon}^{1,n}(t) &= \theta \left(  \int_{0}^{t} \hat{W}^{n}(s) \mathrm{d} s  - \frac{1}{n} \sum_{i=0}^{\lfloor nt \rfloor -1} \hat{W}^{n}_i \right), \\
    \hat{\epsilon}^{2,n}(t) &= \frac{1}{\sqrt{n}} \sum_{i=0}^{\left\lfloor nt \right\rfloor-1} \left( \theta - \Theta_{i} \right) \left( \bar{W}^{n}_{i} - \bar{W}^{*} \right), \\
    \hat{\epsilon}^{3,n}(t) &=  \frac{\left\lfloor nt \right\rfloor-nt}{\sqrt{n}}  \left( \mu_n - \theta \bar{W}^{*} \right).
\end{aligned} 
\end{equation*}

Here are the FCLT results under various parameter settings. 

\begin{theorem}
    \label{thm:diffusion_W}
     Let \( \hat{W}^{n} \) be defined as in \eqref{eq:fclt_scalings} and \( B \) be a standard Brownian motion, then under Assumptions \ref{assump:model} and \ref{assump:fclt}, we have the following:  
    \begin{enumerate}
        \item[(i)] if \( \mu>0\), \( \theta > 0 \) and \( \bar{W}^{*} = \mu / \theta \), then 
            \[
                \hat{W}^{n} \Rightarrow \hat{W} :=\mathcal{M}_\theta \left( \hat{W}_0 + \eta \mathfrak{e} + \sqrt{\sigma_{X}^{2} +  \frac{\mu^{2}}{\theta^{2}} \sigma_{\Theta}^{2}} \ B \right);
            \]
        \item[(ii)] if \(\mu = 0\), \( \theta \ge 0 \) and \( \bar{W}^{*} = 0 \), then 
            \[
                \hat{W}^{n} \Rightarrow \hat{W}:= \mathcal{R}_\theta( \hat{W}_0 + \eta \mathfrak{e} + \sigma_X B);
            \]
        \item[(iii)] if \( \mu < 0\), \( \bar{W}^{*} = 0 \) and $\hat{W}_0 = 0$ a.s., then 
            \[
                \hat{W}^{n} \Rightarrow 0.
            \]
    \end{enumerate}
\end{theorem}

\begin{remark}\label{rmk:fclt_comparison}
    \leavevmode
    \begin{enumerate}  
        \item [(a)] Theorem \ref{thm:diffusion_W} (i) corresponds to Theorem 3 in \cite{Whi90}. The limiting process here is an OU process with stationary distribution \( \mathcal{N}(m, \sigma^{2}) \) where 
        \[
            m = \frac{\eta}{\theta}, \quad \sigma^{2} =  \frac{\sigma^{2}_X}{2\theta} + \frac{\mu^{2} \sigma^{2}_\Theta}{2 \theta^{3}}.  
        \]
        Compared to Whitt's result, we have the same variance, but the mean in his paper is \( 0 \). This is because Whitt assumed \( \mathbb{E} X^{n}_0 = \mu \), however, as we assumed it to be \( \mu_n \) in Assumption \ref{assump:model}(iii) and imposed the condition on the rate of convergence in Assumption \ref{assump:fclt}. 
        \item [(b)] In Theorem \ref{thm:diffusion_W} (ii), the limiting process is a reflected OU process when $\theta >0$ and a reflected Brownian motion when $\theta = 0$. We mention that in order for the limiting diffusion process to have a stationary distribution, we need $\eta < 0$ when $\theta =0$.  
    \end{enumerate}
    
\end{remark}

Before giving the proof of Theorem \ref{thm:diffusion_W}, we first prove a lemma on the error terms. 

\begin{lemma} \label{lem:fclt_errors}
   Under Assumptions \ref{assump:model} and \ref{assump:fclt}, \( \|\hat{\epsilon} ^{1,n}\|_{T}  \), \( \|\hat{\epsilon} ^{2,n} \|_{T}  \) and  \( \| \hat{\epsilon} ^{3,n}\|_{T}  \) converge to 0  
    in probability. 
\end{lemma}

\begin{proof}

    Let \( \epsilon > 0 \). Similar to the proof of Lemma \ref{lem:fluid_error}, we have  
    \begin{equation*}
        \left\| \hat{\epsilon}^{1,n} \right\|_{T} 
        \le  \max_{i = 1} ^{\left\lfloor nt \right\rfloor} \left\vert \frac{\theta}{n} \hat{W}^{n}_i \right\vert  
        = \max_{i = 1} ^{\left\lfloor nt \right\rfloor} \left| \frac{\theta}{\sqrt{n}} \bar{W}^{n}_{i} - \frac{ 1}{\sqrt{n}} \bar{W}^{*}  \right|  
        \le  \frac{\left| \theta \right|}{\sqrt{n}} \max_{i = 1} ^{\left\lfloor nt \right\rfloor} \bar{W}^{n}_{i} + \frac{1}{\sqrt{n}}   \bar{W}^{*}.  
    \end{equation*}
    Then for \( n \) large enough such that \( n^{-1 /2} \bar{W}^{*} < \epsilon /2 \), we have that 
    \begin{equation*}
        \mathbb{P} \left(  \left\| \hat{\epsilon}^{1,n} \right\|_{T} > \epsilon \right)  \le \mathbb{P} \left( \max_{i = 1} ^{\left\lfloor nt \right\rfloor} \bar{W}^{n}_{i} > \sqrt{n} \frac{\epsilon}{2 \left\vert \theta \right\vert }  \right) = \mathbb{P} \left( \left\| \bar{W}^{n} \right\|_{T}  > \sqrt{n} \frac{\epsilon}{2 \left\vert \theta \right\vert }  \right). 
    \end{equation*}
    From Theorem~\ref{thm:fluid_limit}, the convergence of the fluid-scaled process \( \bar{W}^{n} \) implies that the sequence forms a tight family. Consequently, the probability on the right-hand side can be made arbitrarily small by taking \( n \) sufficiently large. It follows that \(\| \hat{\epsilon}^{1,n} \|_{T}\) converges to zero in probability.

    To handle \( \hat{\epsilon}^{2,n} \), we once again appeal to the fluid limit results.  
    Let \( \epsilon > 0 \), \( \eta > 0 \) and consider the event 
    \[
        \Gamma^{n} = \bigl\{ \|\bar{W}^{n} - \bar{W}^* \|_{T} \le \eta \bigr\}. 
    \]
    We first bound the probability of the event \( \{ \|\hat{\epsilon}^{2,n}\|_T > \epsilon \} \) by splitting it into two cases using \( \Gamma^{n} \):   
    \begin{equation}
        \label{eq:fclt_errors_split_event}
        \mathbb{P}  \left( \|\hat{\epsilon}^{2,n}\| _{T} > \epsilon \right)
        \le \mathbb{P} \left( \{ \|\hat{\epsilon}^{2,n}\| _{T} > \epsilon \} \cap   \Gamma^{n} \right) 
        + \mathbb{P} \!\left( \|\bar{W}^{n} - \bar{W}^*\|_{T} >  \eta  \right).
    \end{equation}
    As \( n \to \infty \), the second term vanishes by Theorem \ref{thm:fluid_limit}. Therefore, it suffices to estimate the first term. First let \( \Gamma_i^{n} = \{ \vert \bar{W}^{n}_{i} - \bar{W}^* \vert  \le \eta \} \) and observe the following equivalence of events: 
    \begin{align}
        \label{eq:fclt_errors_equivalent_events} 
        \{ \| \hat{\epsilon}^{2,n} \|_T > \epsilon \} \cap  \Gamma 
        &=  \left\{ \sup_{t \in [0,T]} \left\vert \sum_{i=0}^{\left\lfloor nt \right\rfloor-1} (\theta - \Theta_{i}) (\bar{W}^{n}_{i}- \bar{W}^{*}) 1_{\Gamma_i^{n}} \right\vert > \sqrt{n}\epsilon  \right\} \cap \Gamma^{n} \nonumber 
        \\
        &= \left\{ \sup_{t \in [0,T]}   \left( \sum_{i=0}^{\left\lfloor nt \right\rfloor-1} (\theta - \Theta_{i}) (\bar{W}^{n}_{i}- \bar{W}^{*}) 1_{\Gamma_i^{n}} \right)^{2}   >  n \epsilon^{2}  \right\} \cap \Gamma^{n}\,.
    \end{align}
    For \( k \ge 0 \), let us denote \( Z ^{n}_k =  \sum_{i=0}^{k} (\theta - \Theta_{i}) (\bar{W}^{n}_{i}- \bar{W}^{*}) 1_{\Gamma_i^{n}} \).  It is easy to check that \( \{ (Z^{n}_{k})^{2} \}_k \) is a  \( \{ \mathcal{F}^{n}_k \} \)-submartingale. Then by Doob's martingale inequality, 
    \begin{equation}
        \label{eq:fclt_errors_doob_ineq} 
        \mathbb{P} \left( \{  \|\hat{\epsilon}^{2,n}\|_{T} > \epsilon \} \cap  \Gamma^{n} \right)
         \le   
        \mathbb{P} \left( \max_{0 \le k \le \left\lfloor nT \right\rfloor-1} (Z^{n}_{k})^{2} > n \epsilon^{2}  \right)  \le
        \frac{1}{n \epsilon^{2}} \mathbb{E} \left[  (Z^{n}_{\left\lfloor nT \right\rfloor-1})^{2}\right] . 
    \end{equation}
    By expanding the squares and using the fact that the cross terms have zero expectation (see \eqref{eq:fluid_pf_eps2_1}), 
    \begin{align}
        \label{eq:fclt_errors_moment_bd} 
        \mathbb{E} \left[  (Z^{n}_{\left\lfloor nT \right\rfloor-1})^{2} \right]  
        &=  
         \sum_{i=0}^{\lfloor nT \rfloor - 1} \mathbb{E} \left[ (\theta - \Theta_{i})^{2} \bigl(\bar{W}^{n}_{i} - \bar{W} \bigr)^{2} 1_{\Gamma^{n}_{i}} \right] \nonumber \\
        &\qquad + \sum_{\substack{0 \le i,j \le \lfloor nT \rfloor - 1 \\ i \neq j}} \mathbb{E}\!\left[ (\theta - \Theta_{i})(\theta - \Theta_{j}) \bigl(\bar{W}^{n}_{i} - \bar{W} \bigr) \bigl(\bar{W}^{n}_{j} - \bar{W} \bigr) 1_{\Gamma^{n}_{i}} 1_{\Gamma^{n}_{j}} \right] 
        \le  nT \sigma_{\Theta}^{2} \eta^{2}. 
    \end{align}
    Therefore, combining \eqref{eq:fclt_errors_doob_ineq}, \eqref{eq:fclt_errors_split_event}, \eqref{eq:fclt_errors_doob_ineq} and \eqref{eq:fclt_errors_moment_bd}, we have 
    \begin{equation*}
        \lim_{n \to \infty} \mathbb{P}  \left( \|\hat{\epsilon}^{2,n}\| _{T} > \epsilon \right)
        \le  \frac{T \sigma^{2}_{\Theta}\eta^{2}}{\epsilon^{2}}.
    \end{equation*}
    Since \( \eta > 0 \) is arbitrary, letting \( \eta \to 0 \) implies that \( \|\hat{\epsilon}^{2,n}\|_{T} \to 0 \) in probability.

    Finally, the convergence of \( \hat{\epsilon}^{3,n}\) simply uses the assumption that \( \mu_n \to \mu \).  This concludes the proof.  
\end{proof}

Now we are ready to prove the FCLT results. 

\begin{proof}[Proof of Theorem \ref{thm:diffusion_W}.]
    We shall examine the convergence for each of the terms in \eqref{eq:fclt_representation}, and then use the continuous mapping theorem to obtain the limit for \( \hat{W}^{n} \). We already have weak convergence of the initial condition by Assumption \ref{assump:fclt} and weak convergence of the error terms by Lemma \ref{lem:fclt_errors}. 
    By Donsker's theorem (see for example \cite{Bill99} Section II.8), we have 
    \begin{equation} \label{eq:fclt_pf_brownian_term}
        \frac{1}{\sqrt{n}} \sum_{i=0}^{\left\lfloor nt \right\rfloor-1 } \left( X^{n}_{i} - \mu_{n}  \right) + \frac{1}{\sqrt{n}} \sum_{i=0}^{\left\lfloor nt \right\rfloor -1} (\theta - \Theta_{i}) \bar{W}^{*} \Rightarrow  \sqrt{\sigma_X^{2} + (\bar{W}^{*} \sigma_{\Theta})^{2}} B. 
    \end{equation}
    So we group these terms and denote 
    \begin{align*}
        \Phi^{n}(t)  
        &= \hat{W}^{n}(0) +  \frac{1}{\sqrt{n}} \sum_{i=0}^{\left\lfloor nt \right\rfloor-1 } \left( X^{n}_{i} - \mu_{n}  \right) + \frac{1}{\sqrt{n}} \sum_{i=0}^{\left\lfloor nt \right\rfloor -1} (\theta - \Theta_{i}) \bar{W}^{*} \\
        & \qquad \qquad + \sqrt{n}\left( \mu_n - \mu \right) t   + \hat{\epsilon}^{1,n}(t) + \hat{\epsilon}^{2,n}(t)   + \hat{\epsilon}^{3,n}(t).
    \end{align*}
    The continuous mapping theorem implies that 
    \begin{equation}\label{eq:fclt_Phi}
        \Phi^{n} \Rightarrow  \Phi := \hat{W}_0 + \eta \mathfrak{e} + \sqrt{ \sigma_X^{2} + (\bar{W}^{*})^{2} \sigma_\Theta^{2}} \  B.
    \end{equation}
    We write \eqref{eq:fclt_representation} more compactly as 
    \begin{equation} \label{eq:fclt_compact}
        \hat{W}^{n}(t) = \Phi^{n}(t) - \int_{0}^{t} \theta \hat{W}^{n} (s) \mathrm{d}s    + \sqrt{n}  \left( \mu  - \theta \bar{W}^{*} \right) t  +  \frac{1}{\sqrt{n}}{L}^{n}_{\left\lfloor nt \right\rfloor-1} ,
    \end{equation}

    Now we analyze the behavior of the leftover terms in each of the three cases. 
    
    \textbf{Case i:  } Suppose  \( \mu > 0, \theta > 0 \) with \( \bar{W}^{*}= \mu/\theta \). On the event \( \{ \| \bar{W}^{n} - \bar{W}^* \|_T < \frac{\mu}{2\theta} \} \), it must be that \( \bar{L}^{n}  \equiv 0 \). Hence, 
    \begin{align*}
        &\mathbb{P} \left( \left\| n^{1/2}  \bar{L}^{n} \right\|_{T} > \epsilon \right) \\
        & \qquad \le \mathbb{P} \left( \left\| n^{1/2}  \bar{L}^{n} \right\|_{T} > \epsilon,\ \left\| \bar{W}^{n} - \bar{W}^* \right\|_{T} < \frac{\mu}{2 \theta}  \right) + \mathbb{P} \left( \left\| \bar{W}^{n} - \bar{W}^* \right\|_{T} \ge  \frac{\mu}{2 \theta} \right) \\
        & \qquad = \mathbb{P} \left( \left\| \bar{W}^{n} - \bar{W}^* \right\|_{T} \ge  \frac{\mu}{2 \theta} \right). 
    \end{align*}
    Using Theorem \ref{thm:fluid_limit}, the above probability goes to \( 0 \) as \( n \to \infty \). Therefore we conclude 
    \begin{equation} \label{eq:fclt_pf_L_1}
        n^{-1/2} L^{n} \Rightarrow 0 \quad \textrm{as } n \to  \infty.
    \end{equation}
    Using the mapping $\mathcal{M}_\theta:\mathcal{D}_T \to \mathcal{D}_T$, \eqref{eq:fclt_compact} simplifies to  
    \begin{equation*}
        \hat{W}^{n}(t) 
        =  \mathcal{M}_\theta \Bigg(  \Phi^{n}(t)   +  \frac{1}{\sqrt{n}}{L}^{n}_{\left\lfloor nt \right\rfloor-1}  \Bigg) . 
    \end{equation*}
    Since \( \mathcal{M}_\theta \) is Lipschitz continuous, we use \eqref{eq:fclt_pf_L_1} and apply the continuous mapping theorem to conclude the proof for case 1.   

    For the rest of the cases, the stability point \( \bar{W}^{*} = 0 \). 
    From Remark \ref{rmk:fclt_constant_fluid}, we know that \( \bar{W} \equiv 0 \). Define the process \( \hat{L}^{n} \) by  
    \[
        \hat{L}^{n}(t) = \frac{1}{\sqrt{n}} L^{n}_{\left\lfloor nt \right\rfloor-1} , \quad t \ge 0. 
    \]
    By the same arguments for \eqref{eq:fluid_W_reflected} applied to \eqref{eq:fclt_compact}, we have 
    \begin{equation} \label{eq:fclt_linref_rep}
        (\hat{W}^{n}, \hat{L}^{n}) = (\mathcal{R}_\theta, \mathcal{R}'_{\theta}) \Bigg(   \Phi^{n} + \sqrt{n} \mu t   \Bigg). 
    \end{equation}

    \textbf{Case ii:} Let \( \mu = 0 \), \( \theta \ge 0 \) and \( \bar{W}^{*} = 0 \). 
    Assumptions \ref{assump:fclt}, \eqref{eq:fclt_Phi}, \eqref{eq:fclt_linref_rep} and the continuous mapping theorem give the desired result. 

    \textbf{Case iii (a):} Now suppose \( \mu < 0 \) with \( \bar{W}^{*}= 0 \) and \( \theta \ge  0 \). 
    Let \( c_n = \sqrt{n} \mu  \), then by assumption, \( c_n \to - \infty \) as \( n \to \infty \). Then we can write \eqref{eq:fclt_compact} as
    \begin{equation*}
        \hat{W}^{n}(t) = \Phi^{n}(t) - \int_{0}^{t} \theta \hat{W}^{n}(s) \mathrm{d}s + c_n t + \hat{L}^{n}(t). 
    \end{equation*}
    Let \( \tau_n (t) = \sup \{s: \hat{W}^{n}(s) =0, s \le t\} \). Then, we have 
    \begin{align} \label{eq:fclt_stoppingtime_bound}
        0 \le \hat{W}^{n}(t) &= \hat{W}^{n}(t) - \hat{W}^{n}(\tau_n (t)-)  \nonumber\\
        &= \Phi^{n}(t) - \Phi^{n}(\tau_n (t)-) - \int_{\tau_n (t)}^{t} \theta \hat{W}^{n}(s) \mathrm{d}s + c_n (t - \tau_n (t)) \nonumber \\
        &\le \Phi^{n}(t) - \Phi^{n}(\tau_n (t)-) + c_n (t - \tau_n (t)). 
    \end{align}
    By \eqref{eq:fclt_Phi}, the limit of \( \Phi^{n} \) is in \( \mathcal{C}_T \)  and \(\Phi(0)=\hat{W}_0=0\) a.s. By the Skorokhod representation theorem, we may work in a probability space where \(\|\Phi^n - \Phi\|_{T} \to 0\) a.s. Then the exact same proof for Lemma 6.4 (ii) in \cite{ChenY01} applies and we can obtain   \(\|\tau_n - \mathfrak{e}\|_{T} \to  0 \) a.s.\ as \( n \to \infty \).  Furthermore, since \( c_n < 0 \), we have \( 0 \le \hat{W}^{n} (t) < \Phi^{n}(t) - \Phi^{n}(\tau_n (t)-) \), which implies \( \hat{W}^{n} \Rightarrow 0 \).

    \textbf{Case iii (b):} Now suppose \( \mu < 0 \) with \( \bar{W}^{*}= 0 \) and \( \theta <  0 \). 
    The main idea is to bound the system  by another reflected queue without state dependence. Take \( \epsilon > 0 \) so that \( \mu - \theta \epsilon < 0 \). By Theorem \ref{thm:fluid_limit}, the probability of the event \( \{ \| \bar{W}^{n} \|_{T} \ge  \epsilon \} \) is asymptotically negligible, so we can restrict our consideration to the event \( \{ \| \bar{W}^{n} \|_{T}  < \epsilon \} \). In this case, we have  
    \begin{align} \label{eq:fclt_case_2_bound}
        &\Phi^{n}(t) +  \sqrt{n} \mu t  + \int_{0}^{t} ( - \theta) \hat{W}^{n}(s) \mathrm{d}s \nonumber \\
        & \qquad \le    \Phi^{n}(t) +   \sqrt{n}  \mu t + \int_{0}^{t} ( - \theta) \hat{W}^{n}(s) \mathrm{d}s + \int_{0}^{t} (-\theta) \sqrt{n} \left( \epsilon - \bar{W}^{n}(s) \right) \mathrm{d}s \nonumber \\
        & \qquad =  \Phi^{n}(t) +   \sqrt{n} ( \mu - \theta \epsilon) t . 
    \end{align}
    If we let \( x \equiv  \Phi^{n} + c_n \mathfrak{e}\), and $\mathcal{U}_\theta[x] = u $ be the solution to the integral equation \( u(t) = x(t) - \int_{0}^{t} \theta \mathcal{R}(u) (s) ds  \), then by \eqref{eq:fclt_linref_rep} and \eqref{eq:lgrm}, we have
    \[
        \hat{W}^{n} = \mathcal{R}_\theta \left(x \right)   = \mathcal{R}  \mathcal{U}_{\theta} [x]  = \mathcal{R} \left( x  - \int_{0}^{\cdot} \theta \mathcal{R} \mathcal{U}_{\theta}[x](s) \mathrm{d}s \right) = \mathcal{R} \left( x  - \int_{0}^{\cdot} \theta \hat{W}^{n}(s) \mathrm{d}s \right). 
    \]
    By \eqref{eq:fclt_case_2_bound} and Lemma \ref{lem:Skorokhod_comparison}, we can obtain the bound 
    \begin{equation*}
        0 \le \hat{W}^{n} = \mathcal{R}_\theta \left( \Phi^{n} + c_n \mathfrak{e} \right)  \le \mathcal{R} \left( \Phi^{n} + \sqrt{n} \left( \mu - \theta \epsilon \right) \mathfrak{e} \right). 
    \end{equation*}
    By the Skorokhod representation theorem and Lemma 6.4 (ii) in \cite{ChenY01}, the reflected system \( \mathcal{R} \left( \Phi^{n} + \sqrt{n} \left( \mu - \theta \epsilon \right) \mathfrak{e} \right) \Rightarrow 0 \), and therefore
    \[
    \mathcal{R}_\theta \left( \Phi^{n} + c_n \mathfrak{e} \right) \Rightarrow 0. 
    \]
    This concludes the proof.  
\end{proof}

\section{Background and Useful Facts} \label{app:useful_facts}

\subsection{Miscellaneous Facts}

\begin{lemma}[Discrete Gronwall's lemma] \label{lem:discrete_gronwall}
    Consider a real sequence \( \{ u_k, k \ge 0 \} \) that satisfies 
    \[
        u_{k+1} \le (1+ \alpha) u_k + b_k, \quad \forall k \ge 0, 
    \]
    where \( \alpha \ge 0 \) and \( b_k \ge 0 \) for all \( k \ge 0 \). 
    Then 
    \begin{equation*}
        u_{k} \le  e^{k \alpha}   u_{0} + e^{k \alpha}  \sum_{j=0}^{k-1} b_{j}  . 
    \end{equation*} 
\end{lemma}

\begin{proof}
    This is straightforward by expanding the recursion and observing \( (1+\alpha)^{k} \le e^{k \alpha} \). 
\end{proof}

Here is a functional weak law of large numbers (FWLLN) for partial sums of triangular arrays that is sufficient for our purpose. 
\begin{lemma}
    [FWLLN] \label{lem:FWLLN}
    For each $ n \in \mathbb{N} $, let the random variables $ X_{n,1}, \ldots , X_{n,n} $ be i.i.d.\ with $ \mathbb{E}X_{n,1} = \mu_{n} $ and $ \operatorname{Var}(X_{n,1}) = \sigma_{n}^{2} $ such that $ \lim_{n \to \infty} \mu_{n} = \mu \in \mathbb{R} $ and $ \sup_{n \in \mathbb{N}} \sigma_{n}^{2} < \infty $. Consider a process defined by the partial sum 
    \[
        S^{n}(t ) = \frac{1}{n} \sum_{i=1}^{\left\lfloor nt \right\rfloor} X_{n,i}, \quad t \in [0,T].  
    \]
    Then $ \| S^n - \mu \mathfrak{e} \|_{T} \to 0 $ in probability. 
\end{lemma}

\begin{proof}
    We have 
    \begin{equation*}
        \sup_{t \in [0,T]} \vert S^{n}(t ) - \mu t \vert \le  \frac{1}{n} \max_{1 \le k \le \lfloor nT \rfloor} \vert \sum_{i=1}^{k} (X_{n,i} - \mu_{n}) \vert + \sup_{t \in [0,T]} \vert \mu_{n} t - \mu t  \vert .  
    \end{equation*}
    Then for $ \epsilon > 0 $, 
    \begin{equation*}
        \mathbb{P} \left( \sup_{t \in [0,T]}  \vert S^{n}(t ) - \mu t \vert \ge  \epsilon \right) \le \mathbb{P} \left(\max_{1 \le k \le \lfloor nT \rfloor} \left\vert \sum_{i=1}^{k}  (X_{n,i} - \mu_{n}) \right\vert \ge  \frac{\epsilon n}{2}   \right) + \mathbb{P} \left( \sup_{t \in [0,T]} \vert \mu_{n}t - \mu t \vert \ge  \frac{\epsilon}{2} \right). 
    \end{equation*}
    By Kolmogorov's maximal inequality, as $ n \to  \infty $, 
    \begin{equation*}
        \mathbb{P} \left( \max_{1 \le k \le \lfloor nT \rfloor} \left\vert \sum_{i=1}^{k}  (X_{n,i} - \mu_{n}) \right\vert \ge  \frac{\epsilon n }{2}    \right) \le  \frac{ 4 \operatorname{Var}(\sum_{i=1}^{n} Y_{n,i})}{ \epsilon^{2} T^2 n^{2}} = \frac{ 4 \sigma^{2}_{n}}{\epsilon^{2} T^2 n } \to  0 . 
    \end{equation*}
    Also, we assumed $ \mu_{n} \to  \mu $, therefore 
    \begin{equation*}
        \lim_{n \to \infty} \mathbb{P} \left( \sup_{t \in [0,T]}  \vert S^{n}(t ) - \mu t \vert \ge  \epsilon \right) = 0. 
    \end{equation*} 
    This completes the proof.
\end{proof}

\begin{remark}
    \label{rmk:analysis_facts}
    We frequently use several facts in the proofs, and we collect them here for clarity. 
    \begin{enumerate}
        \item [(a)]Note that for any $x,y \in \mathbb{R}$, $\log(x+y) \le \log(2) + \log (x \vee y)$. Also, we have from real analysis that for any sequences $\{x_n, \ n \in \mathbb{N}\}$, $\{y_n, \ n \in \mathbb{N}\}$ and $\{a_n, \ n \in \mathbb{N}\}$, 
        \begin{equation*}
            \limsup_{n \to \infty} \frac{1}{a_n} \log (x_n \vee y_n) \le \left( \limsup_{n \to \infty} \frac{1}{a_n} \log x_n \right) \vee \left( \limsup_{n \to \infty} \frac{1}{a_n} \log y_n \right). 
        \end{equation*}
        Then it follows that if $a_n \to \infty$ as $n \to \infty$, 
        \begin{align}
            \label{eq:rmk1}
            \limsup_{n \to \infty} \frac{1}{a_n} \log (x_n + y_n) &\le \limsup_{n \to \infty} \frac{\log(2)}{a_n} + \left( \limsup_{n \to \infty} \frac{1}{a_n} \log x_n \right) \vee \left( \limsup_{n \to \infty} \frac{1}{a_n} \log y_n \right) \nonumber \\
            &= \left( \limsup_{n \to \infty} \frac{1}{a_n} \log x_n \right) \vee \left( \limsup_{n \to \infty} \frac{1}{a_n} \log y_n \right). 
        \end{align}

        \item [(b)] Sometimes, we use a symmetry argument which relies on the following fact. Let $x \equiv \{x(t), \ t \in [0,T]\}$ be a process in $\mathcal{D}_T$. Then, for any $\delta>0$,
            \begin{align}
                \label{eq:rmk2}
                \mathbb{P} \Big( \sup_{t \in [0,T]}|x(t)| > \delta \Big) \le \mathbb{P} \bigg( \Big\{\sup_{t \in [0,T]} x(t) > \delta \Big\} \cup \Big\{ \sup_{t \in [0,T]} -x(t) > \delta \Big\} \bigg) \nonumber\\
                \le \mathbb{P} \bigg( \sup_{t \in [0,T]} x(t) > \delta \bigg) + \mathbb{P} \bigg( \sup_{t \in [0,T]} -x(t) > \delta \bigg) .
            \end{align}  
    \end{enumerate}

\end{remark}

\subsection{The Contraction Principle and Continuous Maps}
\label{subsec:LGRM}
In this paper, we prove the MDP results using the contraction principle. Here is a precise statement: 
\begin{theorem}[Contraction Principle]
     Let $f:\mathcal{D}_T \to \mathcal{D}_T$ be continuous and suppose the family $\{x_n\}_{n \in \mathbb{N}} $ satisfies an LDP in \(\mathcal{D}_T\) with rate $a_n$ and rate function $I$, then $\{f(x_n),\ n \ge 1\}$ satisfies an LDP with rate $a_n$ and rate function 
    \begin{equation}
        \label{eq:contraction_principle}
        I'(y) = \inf_{x: f(x) =y} I(x). 
    \end{equation}
\end{theorem}
We mention that the continuity of $f$ can be relaxed to having $f$ continuous on the set where the rate function $I$ is finite. This is sometimes referred to as the \textit{extended contraction principle}. See \cite{GCW04} Theorem 4.6 or \cite{PuhW97} Section 3 for details. 

Next, we provide details on several continuous maps on space \( \mathcal{D}_T \). Let \( x \in \mathcal{D}_T \) with \( x(0) \ge 0 \) and \( \theta \in \mathbb{R} \).  
\begin{enumerate}
    \item[(a)] We use \( (\mathcal{R}, \mathcal{R}')(x) \equiv (z,l) \) to denote the conventional Skorokhod reflection mapping of \( x \). The properties of this mapping are well known, see Section 6.2 in \cite{ChenY01} for details. We mention that it is Lipschitz continuous under the uniform topology and can be explicitly expressed as  
    \begin{equation*}
    \begin{aligned}
        l(t) &= \sup_{0 \le s \le t} [-x(s)]^{+}, \\
        z(t) &= x(t) + \sup_{0 \le s \le t} [-x(s)]^{+}. 
    \end{aligned} 
    \end{equation*}
    \item[(b)] We use \( \mathcal{M}_\theta (x) \equiv u \) to denote the solution to the integral equation 
    \begin{equation*}
        u(t) = x(t) - \int_{0}^{t} \theta u(s) \mathrm{d}s. 
    \end{equation*} 
    Lemma 1 in \cite{ReedW04} shows that the solution to such an integral equation exists and is unique, hence \( \mathcal{M}_\theta \) is a well-defined map from \( \mathcal{D}_T \) to \( \mathcal{D}_T \). It also shows that \( \mathcal{M}_\theta \) is Lipschitz continuous under the uniform topology. 
    \item[(c)]  We use \( (\mathcal{R}_\theta, \mathcal{R}'_{\theta})(x) \equiv (z,l) \) to denote the one-dimensional linearly generalized reflection mapping of \( x \). Specifically, for all \( t \in [0,T] \), we have 
    \begin{equation*}
        z(t) = x(t) - \int_{0}^{t} \theta z(s) \mathrm{d}s  + l(t) \ge 0, 
    \end{equation*} 
    with \( l \) being nondecreasing, \( l(0) = 0 \), and \( z(t) \mathrm{d}l(t) = 0 \) for all \( t \in [0,T] \). The multi-dimensional version of this mapping is analyzed in the Appendix of  \cite{ReedW04}. Here we mention that it has the representation: 
    \begin{equation}
        \label{eq:lgrm}
        (\mathcal{R}_\theta, \mathcal{R}'_{\theta}) (x ) = (\mathcal{R},\mathcal{R}') ( \mathcal{U}_{\theta}(x ) ) ,
    \end{equation}
    with $\mathcal{U}_{\theta}(x) = u $ being the solution to the integral equation 
    \begin{equation} \label{eq:lgrm_integral_eqn}
        u(t) = x(t) - \int_{0}^{t} \theta \mathcal{R}(u) (s) ds . 
    \end{equation}
    It is shown in the appendix of \cite{ReedW04} that the map \( \mathcal{U}_{\theta}: \mathcal{D}_T \to  \mathcal{D}_T \) is well-defined and Lipschitz continuous under the uniform topology. Due to  \eqref{eq:lgrm}, it is immediate that the mappings $\mathcal{R}_\theta$ and $ \mathcal{R}'_\theta$ are also Lipschitz continuous under the uniform topology. 
    Finally, we note \( (\mathcal{R}_0, \mathcal{R}'_{0}) \equiv (\mathcal{R}, \mathcal{R}')\). 
\end{enumerate}

\begin{lemma}
    The mappings $ (\mathcal{R},\mathcal{R}') $, $ \mathcal{M}_{\theta} $, $ (\mathcal{R}_{\theta}, \mathcal{R}'_{\theta}) $ are continuous under the $ J_1 $ topology. 
\end{lemma}

\begin{proof} 

    For the conventional one-dimensional Skorokhod reflection mapping $ (\mathcal{R},\mathcal{R}') $, $ J_1 $-continuity is standard. For the mapping $ \mathcal{M}_{\theta} $, Theorem 4.1 of \cite{PangTW07} directly applies. 
    
    It remains to show continuity for the linearly generalized Skorokhod mapping $(\mathcal{R}_{\theta}, \mathcal{R}'_{\theta})$.  Due to \eqref{eq:lgrm}, it suffices to show that the mapping $ \mathcal{U}_{\theta} $ is continuous under the $ J_1 $ topology. Suppose $ d_{J_1}(x_n,x)\to0 $ as $ n \to \infty $ with $ \mathcal{U}_{\theta}(x_{n}) = u_{n} $ and $ \mathcal{U}_{\theta}(x ) = u $. Then there exist  increasing homeomorphisms $ \lambda_{n} $ of $ [0,T] $ such that $ \| x_{n} - x \circ \lambda_{n} \|_{T} \to 0 $ and $ \| \dot{\lambda}_{n} -1 \|_{T} \to 0 $ as $ n \to \infty $. Then we have for some constants $ K_1, K_2 > 0 $, 
    \begin{align*}
        \left| u_{n}(t ) - u( \lambda_{n}(t)) \right| 
        &\le 
        \left| x_{n}(t) - x(\lambda_{n}(t))\right| + \left| \int_{0}^{t} \theta \mathcal{R}(u_{n} )(s) \mathrm{d}s - \int_{0}^{\lambda_{n}(t)}  \theta \mathcal{R}(u ) (s) \mathrm{d}s \right| \\
        & \le 
         \left| x_{n}(t) - x(\lambda_{n}(t))\right| +  \int_{0}^{t} \left|  \theta \mathcal{R}(u )(\lambda_{n}(s)) \right|  \vert 1- \dot{\lambda}_{n}(s  ) \vert  \mathrm{d}s\\
        & \qquad \qquad   +   \int_{0}^{t} \left| \theta \mathcal{R}(u_{n} )(s) - \theta \mathcal{R}(u )(\lambda_{n}(s)) \right| \mathrm{d}s  \\
        & \le 
        \| x_{n} - x \circ \lambda_{n} \|_{T} + K_1 T \| 1- \dot{\lambda}_{n} \|_{T} +  \int_{0}^{t} K_2  \vert  u_{n}(s) - u(\lambda_{n}(s)) \vert  \mathrm{d}s \,.   
    \end{align*}
    In the third inequality above, the constant $ K_1 $ is due to $ \theta \mathcal{R}(u) $ being bounded on the compact interval $ [0,T] $. The constant $ K_2 $ comes from the fact that $ \mathcal{R} $ is Lipschitz and that $ \mathcal{R}(u) \circ \lambda_{n} = \mathcal{R}(u \circ \lambda_{n}) $. 
    Then by Gronwall's inequality, 
    \begin{align*}
        \| u_{n}  - u \circ \lambda_{n} \|_{T} \le 
        \left( \| x_{n} - x \circ \lambda_{n} \|_{T} + K_1 T \| 1- \dot{\lambda}_{n} \|_{T} \right) e^{K_2T}. 
    \end{align*}
    This implies that \(u_n\to u\) in the \(J_1\) topology as $ n \to \infty $ and hence $ \mathcal{U}_{\theta} $ is continuous under the $ J_1 $ topology. 

\end{proof}

Here is a comparison result for the conventional Skorokhod mapping. 

\begin{lemma} \label{lem:Skorokhod_comparison}
    Let \( x,y \in \mathcal{D} \). Suppose \( y \) is a positive, nondecreasing process and let
    \begin{equation*}
    \begin{aligned}
        (z,l) &= (\mathcal{R},\mathcal{R}')(x+y), \\
        (z',l') &= (\mathcal{R},\mathcal{R}')(x).
    \end{aligned}
    \end{equation*}
    Then \( z \ge z' \) for all \( t \ge 0 \). 
\end{lemma}

\begin{proof}
    
    Let \( s \ge 0 \). By breaking down each case, it is straightforward to see that  
    \[
        y(s) + [-x(s)-y(s)]^{+}  \ge  [-x(s)]^{+}.
    \]
    Then using the fact that \( y \) is nondecreasing, we have
    \begin{align*}
        z(t) = x(t) + y(t) + l(t) &= x(t) + y(t) + \sup_{s \le t}[-x(s)-y(s)]^{+} \\
        &\ge  x(t) + \sup_{s \le t} \left( y(s) + [-x(s)-y(s)]^{+} \right) \\
        &\ge x(t) + \sup_{s \le t} [-x(s)]^{+} = x(t) + l(t) = z'(t). 
    \end{align*}
    This concludes the proof.  
\end{proof}

\subsection{Exponential Tightness} \label{subsec:exp_tightness}

Recall that for any $x \in \mathcal{C}_T$ and $\delta \in [0,T]$, we define the modulus of continuity as
\[
	w_T(x,\delta) \equiv \sup_{ \substack{|s-t| < \delta\\ 0\le s,t \le T}} |x(s) - x(t) | ,
\]
which is used to prove tightness in $\mathcal{C}_T$. 
For a function $x = \{x(t), t \ge 0\} \in \mathcal{D}_T$, let  
\begin{equation} 
	w_x	[s,t) \equiv \sup_{s \le u, v < t} | x(u) - x(v) | , \quad s < t, 
\end{equation}
and then, for $T>0$, $\delta>0$,  define the following notion of ``modulus of continuity":
\begin{equation}
    \label{eq:mod-cont-D}
    w'_T (x, \delta) \equiv \inf_{  \{t_j\}} \max_{0 < j \le k} w_x [t_{j-1},t_j),
\end{equation}
where $\{t_j\}_{j=0, 1, \ldots, k}$ are finite partitions of $[0,T]$ such that $t_j - t_{j-1} > \delta$, for all $j = 1, \ldots, k$. 

We restate the following necessary and sufficient condition from \cite{Puh91} Theorem 4.2 for exponential tightness of probability measures in space $\mathcal{D}_T$ with Skorokhod $J_1$ topology.

\begin{theorem}
	\label{thm:exp_tightness_D}
	A family  of processes $(x_n)_{n \in \mathbb{N}}$ on $(\mathcal{D}_T, J_1)$ is exponentially tight with rate $a_n$ if and only if: 
	\begin{enumerate}
		\item[(i)] We have 
			\begin{equation}
                \label{lem_exp_tightness_char_1}
				\lim_{A \to \infty} \limsup_{n \to \infty} \frac{1}{a_n}  \log P \left(  \sup_{0 \le t \le T} |x_{n}(t)| \ge A \right) = - \infty. 
			\end{equation}
		\item[(ii)] For any $\eta > 0$,   
			\begin{equation}
                \label{lem_exp_tightness_char_2}
				\lim_{\delta \to 0} \limsup_{n \to \infty} \frac{1}{a_n} \log P (  w'_T(x_n,\delta) \ge \eta) = -\infty .
			\end{equation}
	\end{enumerate}
\end{theorem}

We use the following lemma to check \(\mathcal{C}\)-exponential tightness. This is a special case of Theorem A.3 in \cite{Puh23}.   
\begin{theorem}
    \label{thm:C_exp_tightness_in_D}
    A family of processes \( \{ x_n \}_{n \in \mathbb{N}} \) on \(  \mathcal{D}_T \) is \( \mathcal{C} \)-exponentially tight with rate \( a_n \) if 
    \begin{enumerate}
        \item [(i)] For every \( t \in [0,T] \), the family of random variables \( \{ x_n(t) \}_{n \in \mathbb{N}} \) is exponentially tight with rate \( a_n \). That is, for all $\alpha > 0$, there exists some $K_\alpha > 0$ such that 
        \begin{equation*}
            \limsup_{n \to \infty} \frac{1}{a_n}  \log P \left(   |x_{n}(t)| > K_\alpha \right) < - \alpha. 
        \end{equation*}
        \item [(ii)] For every \( \epsilon > 0 \), we have 
        \begin{equation*}
            \lim_{\delta \to 0} \limsup_{n \to \infty} \sup_{t \in [0,T]} \frac{1}{a_n} \log \mathbb{P} \left(  \sup_{\substack{s \in [0,\delta], \\ t+s \le T}} \vert x_n(t+s) - x_n(t) \vert > \epsilon \right) = - \infty.  
        \end{equation*}
    \end{enumerate} 
\end{theorem}

The next lemma says that exponential tightness is preserved under continuous maps. 

\begin{lemma}
    \label{lem:exp_tightness_cts_map}
    Let \( \mathcal{X}, \mathcal{X}'\) be Polish spaces and the map \( h: \mathcal{X} \to \mathcal{X}' \) be continuous. Suppose that the family of random elements \( \{ x_n, n \in \mathbb{N} \} \) is exponentially tight in \( \mathcal{X} \) with rate \( a_n \). Then the family \( \{ h(x_n),\ n \in \mathbb{N} \} \) is exponentially tight in \( \mathcal{X}' \) with rate \( a_n \). 
\end{lemma}

\begin{proof}
    For Polish spaces, by Theorem (P) in \cite{Puh91}, exponential tightness is equivalent to partial LDP. 
    That is, for each subsequence \( \{ n' \} \) of \( \{ n \} \), there exists a further subsequence \( \{ n'' \} \) of \( \{ n' \} \) such that the family \( \{ x_{n''} \} \) obeys an LDP. By the contraction principle, \( \{ h(x_{n''}) \} \) also obeys an LDP, therefore \( \{ h(x_n) \} \) is exponentially tight.
\end{proof}

\subsection{Super-exponential Convergence in Probability} 
\label{subsec:sup_exp_cvg_prob}

A detailed study can be found in \cite{PuhW97}. We first state a useful result taken from that reference, which is a characterization of super-exponential convergence in probability when the limit is deterministic and  continuous. 

\begin{lemma}
    \label{lem:supexp_char}
    Let $x_0 \equiv (x_0 (t), \ t \ge 0)$ be continuous. Then $X_n \stackrel{P^{1/a_n}}{\longrightarrow} x_0$ if and only if 
    \begin{equation}
        \limsup_{n \to \infty} \frac{1}{a_n} \log \mathbb{P} \left( \sup_{t \in [0,T]} | X_n(t) - x_0(t)  | > \epsilon  \right) = - \infty, 
    \end{equation}
    for all $\epsilon>0$, $T>0$. 
\end{lemma}

The next lemma is a weaker version of Lemma 4.2 (b) in \cite{PuhW97}. 
\begin{lemma} \label{lem:exptight_fluid0}
    Let \( \{ c_n, \ n\ge 1 \} \) be a sequence such that \( c_n \to  \infty \) and \( \{ x_n,\ n \ge 1 \} \) be a family of processes on \( \mathcal{D}_T \). Suppose \( \{ c_n x_n,\ n\ge 1\} \) is exponentially tight on $\mathcal{D}_T$ with rate \( a_n \). Then, 
    \begin{equation*}
        x_n \stackrel{P^{1/a_n }}{\longrightarrow} 0. 
    \end{equation*}
\end{lemma}

\begin{proof}
    Let \( \alpha > 0 \). By the characterization of exponential tightness in Theorem \ref{thm:exp_tightness_D}, we can find \( K_\alpha \) such that 
    \begin{equation*}
        \limsup_{n \to \infty} \frac{1}{a_n} \log \mathbb{P} \left( \left\| c_n x_n \right\|_{T} > K_\alpha \right) < - \alpha. 
    \end{equation*}
    Now let \( \epsilon >0 \). Since \( c_n \to \infty  \), we can find \( n_0  \) such that \( \epsilon > K_\alpha / c_{n_0}  \). Then for all \( n > n_0 \), 
    \begin{equation*}
         \mathbb{P} \left( \| x_n \|_T  > \epsilon \right) \le \mathbb{P} \left( c_{n_0} \|  x_n \|_T  > K_{\alpha} \right) \le \mathbb{P} \left(  \|c_n x_n \|_T  > K_{\alpha} \right). 
    \end{equation*}
    Therefore,
    \begin{equation*}
         \limsup_{n \to \infty}
        \frac{1}{a_n} \log \mathbb{P} \left( \left\|  x_n \right\|_{T} > \epsilon \right) 
        \le  \limsup_{n \to \infty} \frac{1}{a_n} \log \mathbb{P} \left( \|c_n x_n \|_T  > K_{\alpha} \right) < -\alpha.  
    \end{equation*}
    Since \( \alpha \) is taken arbitrarily, we can take \( \alpha \to \infty \) and conclude the proof.  
\end{proof}

The following result can be seen as an analog of the random time-change theorem in \cite{ChenY01} Theorem 5.3. It describes when a process is exponentially equivalent to itself after performing a random time-change. 
For a proof, see \cite{FHP25} Theorem A.4. 
\begin{theorem}[Random time-change]
	\label{thm:time_change}
	Suppose that the processes $\{y^{n},\ n\in \mathbb{N}\} \subset \mathcal{D}_T$ satisfy $y^{n} \stackrel{P^{1/a_n}}{\longrightarrow} \mathfrak{e}$ and the family of processes $\{X^{n},\ n\in \mathbb{N}\} \subset \mathcal{D}_T$ is $\mathcal{C}$-exponentially tight with rate $a_n$, then 
    \begin{equation*}
        X^{n} - X^{n} \circ y^{n} \stackrel{P^{1/a_n}}{\longrightarrow} 0. 
    \end{equation*}
\end{theorem}

The next lemma deals with the conventional Skorokhod mapping when the input process has a negative drift component that goes to infinity. 

\begin{lemma}
    \label{lem:reflection_expequiv_0}
    Suppose that the family \( \{ x_n, \ n \in \mathbb{N} \} \) is $\mathcal{C}$-exponentially tight with rate \( a_n \), \\
     \( x_{n}(0) \stackrel{P^{1/a_{n}}}{\longrightarrow} 0\) and $ x_{n}(0) \ge 0 $ for all $ n \in \mathbb{N} $.
    Further, let \( c_n \) be a sequence such that \( \lim_{n \to \infty} c_n = - \infty \). Then 
    \begin{equation*}
        \mathcal{R}(x_n + c_n \mathfrak{e}) \stackrel{P ^{1/a_n}}{\longrightarrow} 0. 
    \end{equation*}
\end{lemma}

\begin{proof} 

    Since $ c_{n} \to - \infty $ as $ n \to \infty $, we shall assume without loss of generality that $ c_{n} < 0 $. 
    Let \( z_n = \mathcal{R}(x_n + c_{n} \mathfrak{e}) \). For each \( t \in [0,T] \), we write \( z_n (t) = x_n(t)  + c_n t + y_{n}(t) \) with \( y_n = \mathcal{R}'(x_n + c_n \mathfrak{e}) \). Define the stopping time 
    \[
        \sigma_{n} = \inf \{ s \in [0,T]: z_{n}(s) = 0 \}, 
    \]
    with the convention that \(\sigma_{n}=T\) if the set is empty. Let $ \delta \in ( 0, T) $ be arbitrary.  Observe on $ \{ \sigma_{n} > \delta \} $, we have $ z_{n}(\delta) = x_{n}(\delta) + c_{n} \delta > 0 $, which implies $ \| x_{n} \|_{T} > -c_{n}\delta $. Also, since $ \{ x_{n} \}_{n \in \mathbb{N}} $ is exponentially tight, by \eqref{lem_exp_tightness_char_1} in Theorem \ref{thm:exp_tightness_D}, we have 
    \begin{equation} \label{eq:reflection_expequiv_0_concl1}
        \limsup_{n \to \infty} \frac{1}{a_{n}} \log \mathbb{P} \left( \sigma_{n} > \delta \right) \le \limsup_{ n \to \infty} \frac{1}{a_{n}} \log \mathbb{P} \left( \| x_{n} \|_{T} > - c_{n} \delta \right) = -\infty. 
    \end{equation}
    
    Now let $ \epsilon > 0 $ and denote $ \Gamma_{\delta} = \{ \sigma_{n} \le \delta \} $. Observe on $ \Gamma_{\delta}  $, for all $ t \in [0, \sigma_{n}) $, $ z_{n}(t) = x_{n}(t ) + c_{n}t \le |x_{n}(0)| + \vert x_{n}(t ) - x_{n}(0) \vert \le |x_{n}(0)| + w_{T}(x_{n}, \delta) $. This implies that 
    \begin{align*}
        \Gamma_{\delta} \cap  \left\{ \sup_{t \in [0, \sigma_{n}]} z_{n}(t) > \epsilon \right\} \subseteq \left\{ |x_{n}(0)| > \frac{\epsilon}{2} \right\} \cup  \left\{ w_{T}(x_{n},\delta) > \frac{\epsilon}{2} \right\} . 
    \end{align*}
    By assumptions on $ x_{n} $, we can obtain  
    \begin{align*}
        &\limsup_{n \to \infty} \frac{1}{a_{n}} \log \mathbb{P} \left( |x_{n}(0)| > \frac{\epsilon}{2} \right) = - \infty, \\ 
        & \lim_{\delta \to 0} \limsup_{n \to \infty} \frac{1}{a_{n}} \log \mathbb{P} \left( w_{T}(x_{n},\delta) > \frac{\epsilon}{2} \right) = - \infty. 
    \end{align*}
    Together, these imply 
    \begin{equation} \label{eq:reflection_expequiv_0_concl2}
        \lim_{\delta \to 0} \limsup_{ n \to \infty} \frac{1}{a_{n}} \log \mathbb{P} \left( \left\{ \sup_{t \in [0, \sigma_{n})} z_{n}(t) > \epsilon \right\} \cap \Gamma_{\delta} \right) = -\infty. 
    \end{equation}
 
    Now we define a random time-change map $ \tau_{n} $.  On the event $ \Gamma_{\delta} $, let 
    \[
        \tau_n(t) = 
        \begin{cases}
            t, & t \in [0, \sigma_{n}), \\
            \sup \{ s \in [0,t]: z_n(s) = 0 \}, & t \in [\sigma_{n}, T],
        \end{cases}
    \]
    and on the event $ \Gamma_{\delta}^{c} $, let $ \tau_{n}(t) = t $, for all $ t \in [0,T] $. 

    On $ \Gamma_{\delta} $, for $ t \in [\sigma_{n}, T] $,  it is possible to have $ z_{n}(\tau_{n}(t )) = 0 $ or $ z_{n}(\tau_{n}(t)-) = 0 $. Subtracting this from $ z_{n}(t) $ in both cases, we obtain  
    \begin{equation} 
        0 \le z_n(t) \le  |x_n(t) - x_{n}(\tau_{n}(t))| + |x_{n}(\tau_{n}(t)) - x_n (\tau_n(t)-)| + c_n (t- \tau_n(t)). \label{eq:auxthm_reflection_exp_0_key_eq}
    \end{equation}
    Since \( c_n < 0 \),  rearranging the terms and taking supremum over $ [0,T] $ yields  
    \[
        \| \mathfrak{e} - \tau_n \|_T \le \frac{4}{-c_n} \| x_n \|_T.  
    \]
    Again, by exponential tightness of \( x_n \), for any $ \epsilon > 0 $, we can obtain through \eqref{lem_exp_tightness_char_1} and \eqref{eq:reflection_expequiv_0_concl1} that 
    \begin{align*}
        & \limsup_{n \to \infty} \frac{1}{a_n} \log \mathbb{P} \left(   \| \mathfrak{e} - \tau_n \|_T > \epsilon     \right) \\
        & \le 
        \limsup_{n \to \infty} \frac{1}{a_n} \log \mathbb{P} \left( \left\{ \| \mathfrak{e} - \tau_n \|_T > \epsilon \right\} \cap  \Gamma_{\delta} \right) \vee \limsup_{n \to \infty} \frac{1}{a_n} \log \mathbb{P} \left(   \Gamma_{\delta}^{c} \right)  \\
        & \le \limsup_{n \to \infty} \frac{1}{a_n} \log \mathbb{P} \left( \| x_n \|_T > \frac{-c_n \epsilon}{4}  \right)  = - \infty. 
    \end{align*}
    By Lemma \ref{lem:supexp_char}, this implies $ \tau_n \stackrel{P ^{1/a_n}}{\longrightarrow} \mathfrak{e} $.  
    Then immediately, using Theorem \ref{thm:time_change}, we obtain  
    \begin{equation}
        x_n - x_n \circ \tau_n \stackrel{P ^{1/a_n}}{\longrightarrow} 0. \label{eq:reflection_expequiv_0_pf_timechange}
    \end{equation}
    Equation \eqref{eq:auxthm_reflection_exp_0_key_eq} also implies that for all \( t \in [\sigma_{n},T] \), 
    \begin{equation*}
        0 \le z_n(t) \le |x_n(t) - x_{n}(\tau_{n}(t))| + |x_{n}(\tau_{n}(t)) - x_n (\tau_n(t)-)|. 
    \end{equation*}
    Hence, for arbitrary \( \epsilon >0 \), 
    \begin{align*}
        & \mathbb{P} \left( \left\{ \sup_{t \in [\sigma_{n},T]} z_n(t) > \epsilon \right\} \cap \Gamma_{\delta} \right)   \\
        &\le \mathbb{P} \left(   \| x_n - x_n \circ \tau_n \|_T + \sup_{t \in [0,T]} |x_n(t) - x_n(t-)| > \epsilon   \right) \\
        &\le  \mathbb{P} \left(  \| x_n - x_n \circ \tau_n \|_T > \frac{\epsilon}{2}  \right)  + \mathbb{P} \left( \sup_{t \in [0,T]} |x_n(t) - x_n(t-)| > \frac{\epsilon}{2} \right). 
    \end{align*}
    Then by \eqref{eq:reflection_expequiv_0_pf_timechange}, \( \mathcal{C} \)-exponential tightness of \( x_n \) and Remark \ref{rmk:analysis_facts}, we obtain 
    \begin{equation} \label{eq:reflection_expequiv_0_concl3}
        \limsup_{n \to \infty} \frac{1}{a_n} \log \mathbb{P} \left( \left\{ \sup_{t \in [\sigma_{n},T]} z_n(t) > \epsilon \right\} \cap \Gamma_{\delta}  \right) = - \infty. 
    \end{equation}
    Finally, by \eqref{eq:reflection_expequiv_0_concl1}, \eqref{eq:reflection_expequiv_0_concl2} and \eqref{eq:reflection_expequiv_0_concl3}, we have that for any $ \alpha > 0 $, there exist small enough $ \delta $ such that 
    \begin{align*}
        \limsup_{n \to \infty} \frac{1}{a_n} \log \mathbb{P} \left( \| z_n \|_T > \epsilon \right) 
        & \le 
        \limsup_{n \to \infty} \frac{1}{a_n} \log \mathbb{P} \left( \Gamma_{\delta}^{c} \right) \\
        & \qquad \qquad \vee \limsup_{ n \to \infty} \frac{1}{a_{n}} \log \mathbb{P} \left( \left\{ \sup_{t \in [0, \sigma_{n})} z_{n}(t) > \frac{\epsilon}{2} \right\} \cap \Gamma_{\delta} \right)  \\
        & \qquad \qquad \vee \limsup_{n \to \infty} \frac{1}{a_n} \log \mathbb{P} \left( \left\{ \sup_{t \in [\sigma_{n},T]} z_n(t) > \frac{\epsilon}{2} \right\} \cap \Gamma_{\delta}  \right) \\
        &< - \alpha. 
    \end{align*}
    Since $ \alpha $ is arbitrary, we can take $ \alpha \to \infty $ and this concludes the proof.  
\end{proof}

\section*{Acknowledgement}
The authors are grateful to  the anonymous reviewers for their careful reading of the manuscript and for their thoughtful and constructive comments, which substantially improved the presentation and clarity of the paper.
C. Feng and J. Hasenbein are partly supported by the NSF grant DMS 2108682.  
G. Pang is partly supported by NSF grants DMS 2216765 and CMMI 2452829.

\bibliographystyle{abbrvnat}
\bibliography{references}
\end{document}